\documentclass[10pt]{amsart} 
\usepackage[utf8]{inputenc}
\usepackage[T1]{fontenc}
\usepackage{amsmath}
\usepackage{amsxtra}
\usepackage{amscd}
\usepackage{amsthm}
\usepackage{amsfonts}
\usepackage{amssymb}
\usepackage{xspace}
\usepackage{MnSymbol}

\usepackage[colorlinks=true,breaklinks=true,urlcolor=black,linkcolor=black,citecolor=black,bookmarksopen=true,linktocpage=true,plainpages=false,pdfpagelabels]{hyperref}

\usepackage[capitalize]{cleveref}

\newtheorem{thm}{Theorem}[section]
\newtheorem{cor}[thm]{Corollary}
\newtheorem{lem}[thm]{Lemma}
\newtheorem{prop}[thm]{Proposition}

\newtheorem{question}[thm]{Question}
\newtheorem{problem}[thm]{Problem}
\newtheorem{claim}{Claim}[thm]

\theoremstyle{definition}

\newtheorem{defn}[thm]{Definition}
\newtheorem{notation}[thm]{Notation}

\newtheorem{rem}[thm]{Remark}
\newtheorem{example}[thm]{Example}
\newtheorem*{discussion}{Discussion}

\newcommand\push[2]{#1_\star #2}
\newcommand\tensor{\otimes}
\newcommand\normal{\trianglelefteq}
\newcommand\acl{\mathrm{acl}}
\newcommand\dcl{\mathrm{dcl}}
\newcommand{\restr}[2]{{\left.#1\right|{#2}}}
\newcommand\tp{\mathrm{tp}}
\newcommand\ACVF{\mathrm{ACVF}}
\newcommand\alg[1]{#1^{\mathrm{alg}}}
\newcommand\fg{\mathfrak{g}}
\newcommand\fh{\mathfrak{h}}
\newcommand\AFD{(FD)\xspace}
\newcommand\Aom{(FD$_\omega$)\xspace}
\newcommand\St{\mathrm{St}}
\renewcommand{\projlim}{\varprojlim}
\newcommand{\ind}{\downfree}
\newcommand{\germ}[2]{[#2]_{#1}}
\newcommand{\Zz}{\mathbb{Z}}
\newcommand{\dimst}{\dim_{\mathrm{st}}}
\newcommand{\dimo}{\dim_{\mathrm{o}}}
\newcommand{\RM}{\mathrm{MR}}
\newcommand{\transl}[2]{{^{#1}#2}}
\newcommand{\rtransl}[2]{#1^{#2}}
\newcommand{\defsc}[2]{(\mathrm{d}_{#1}#2)}
\newcommand{\fequiv}{\leftrightarrow}
\renewcommand{\mid}{:}
\newcommand{\Stab}{\mathrm{Stab}}
\newcommand{\cL}{\mathcal{L}}
\newcommand{\Sym}{\mathrm{Sym}}
\newcommand{\fF}{\mathfrak{F}}
\newcommand{\fG}{\mathfrak{G}}
\newcommand{\sminus}{\smallsetminus}
\newcommand{\Uu}{\mathbb{U}}
\newcommand{\union}{\cup}
\newcommand{\meet}{\cap}
\newcommand{\Oo}{\mathcal{O}}
\newcommand{\ACF}{\mathrm{ACF}}
\newcommand{\val}{\mathrm{val}}
\newcommand{\spec}{\mathrm{Spec}}
\newcommand{\Aa}{\mathbb{A}}
\newcommand{\trdeg}{\mathrm{trdeg}}
\newcommand{\Hh}{\mathbb{H}}
\newcommand{\Ff}{\mathcal{F}}
\newcommand{\End}{\mathrm{End}}
\newcommand{\Pp}{\mathbb{P}}
\newcommand{\tA}{\widetilde{A}}
\newcommand{\Gg}{\mathbb{G}}
\newcommand{\Ga}{\Gg_{\mathrm{a}}}
\newcommand{\Gm}{\Gg_{\mathrm{m}}}
\newcommand{\eq}[1]{#1^{\mathrm{eq}}}
\newcommand{\res}{\mathrm{res}}
\newcommand{\Cc}{\mathbb{C}}
\newcommand{\Rr}{\mathbb{R}}
\newcommand{\fM}{\mathfrak{M}}
\newcommand{\Qq}{\mathbb{Q}}
\newcommand{\GL}{\mathrm{GL}}
\newcommand{\cS}{\mathcal{S}}
\newcommand{\cG}{\mathcal{G}}
\newcommand{\SL}{\mathrm{SL}}
\newcommand{\dimK}{\dim_{\mathrm{K}}}
\newcommand\cB{\mathcal{B}}

\title{Valued fields, Metastable groups}

\author{Ehud Hrushovski}
 
\address{Institute of Mathematics, Hebrew
  University of Jerusalem, Jerusalem, 91904, Israel, and Mathematical Institute, University of Oxford, Andrew Wiles Building, Oxford, OX2 6GG, United Kingdom.} 
 
\email{ehud@math.huji.ac.il}

\author{Silvain Rideau}

\address{CNRS, Université Paris Diderot, Sorbonne Université, Institut de Mathématiques de Jussieu-Paris Rive Gauche, IMJ-PRG (UMR 7586), F-75013, Paris, France.}

\email{silvain.rideau@berkeley.edu}

\begin{document}

\begin{abstract}
  We introduce a class of theories called {\em metastable}, including
  the theory of algebraically closed valued fields ($\ACVF$) as a
  motivating example. The key local notion is that of definable types
  dominated by their stable part. A theory is metastable (over a sort
  $\Gamma$) if every type over a sufficiently rich base structure can
  be viewed as part of a $\Gamma$-parametrized family of stably
  dominated types. We initiate a study of definable groups in
  metastable theories of finite rank. Groups with a stably dominated
  generic type are shown to have a canonical stable quotient. Abelian
  groups are shown to be decomposable into a part coming from
  $\Gamma$, and a definable direct limit system of groups with stably
  dominated generic.  In the case of \(\ACVF\), among definable
  subgroups of affine algebraic groups, we characterize the groups
  with stably dominated generics in terms of group schemes over the
  valuation ring. Finally, we classify all fields definable in
  \(\ACVF\).
\end{abstract}

\maketitle

\section{Introduction}

Let $V$ be a variety over a valued field $(F,\val)$.  By a
\emph{$\val$-constructible set} we mean a finite Boolean combination
of sets of the form \(\{x \in U: \val(f(x)) \leq \val(g(x)) \}\) where
$U$ is an open affine, and $f,g$ are regular functions on $U$. Two
such sets can be identified if they have the same points in any valued
field extension of $F$, or equivalently, by a theorem of Robinson, in
any fixed algebraically closed valued field extension $K$ of $F$.

In many ways these are analogous to constructible sets in the sense of
the Zariski topology, and (more closely) to semi-algebraic sets over
real fields.  However, while the latter two categories are closed under
quotients by equivalence relations, the valuative constructible sets
are not. For instance, the valuation ring $\Oo = \{x: \val(x) \geq 0
\}$, is constructible, and has constructible ideals $\alpha \Oo = \{x:
\val(x) \geq \alpha \}$ and $\alpha \fM = \{x: \val(x) > \alpha \}$.
For any group scheme $G$ over $\Oo$, one obtains corresponding
congruence subgroups; but the quotients $\Oo / \alpha \Oo$ and $G(\Oo
/ \alpha \Oo)$ are not constructible. We thus enlarge the category by
formally adding quotients, referred to as the {\em imaginary sorts};
the objects of the larger category will be called the {\em definable
  sets}. It is explained in \cite{HasHruMac-ACVF} that instead of this
abstract procedure, it suffices to add sorts for the homogeneous
spaces $\GL_n(K)/\GL_n(\Oo)$, for all $n$, and also $\GL_n(K)/I$ for a
certain subgroup $I$; we will not require detailed knowledge of this
here.  One such sort that will be explicitly referred to is the value
group $\Gamma:= \GL_1(K)/\GL_1(\Oo)$; this is a divisible ordered
Abelian group, with no additional induced structure. We will refer to
{\em stable sorts} also. These include, first of all, the residue field
$k = \Oo/\fM$, but also vector spaces over $k$ of the form $L/\fM L$,
where $L\subseteq K^n$ is an $\Oo$-lattice.

The paper \cite{HasHruMac-Book}, continuing earlier work, studied the
category of quantifier-free definable sets over valued fields,
especially with respect to imaginaries. As usual, the direct study of
a concrete structure of any depth is all but impossible, if it is not
aided by a general theory.  We first tried to find a generalization of
stability (or simplicity) in a similar format, capable of dealing with
valued fields as stability does with differential fields, or
simplicity with difference fields. To this we encountered resistance;
what we found instead was not a new analogue of stability, but a new
method of utilizing classical stability in certain unstable
structures.
 
Even a very small stable part can have a decisive effect on the
behavior of a quite ``large,'' unstable type.  This is sometimes
analogous to the way that the (infinitesimal) linear approximation to
a variety can explain much about the variety; and indeed, in some cases,
casts tangent spaces and Lie algebras in an unexpected model theoretic
role. Two main principles encapsulate the understanding gained:  
\begin{enumerate}
\item Certain types are dominated by their stable parts.  They behave
  ``generically'' as stable types do.
\item Uniformly definable families of types make an appearance; they
  are indexed by the linear ordering $\Gamma$ of the value group, or
  by other, piecewise-linear structures definable in $\Gamma$. An
  arbitrary type can be viewed as a definable limit of stably
  dominated types (from principle (1)).
\end{enumerate}

A general study of stably dominated types was initiated in
\cite{HasHruMac-Book}; it is summarized in
\S\ref{S:prelim}. Principle (2) was only implicit in the proofs
there. We state a precise version of the principle, and call a theory
satisfying (2) \emph{metastable}. We concentrate here on finite rank
metastability.

Principle (1) is given a general group-theoretic rendering in
Proposition\ \ref{P:groupdom}. Stably dominated groups are defined,
and it is shown that a group homomorphism into a stable group controls
them generically. Theorem\ \ref{T:lmstd} clarifies the second
principle in the context of Abelian groups. A metastable Abelian group
of finite rank is shown to contain nontrivial stably dominated groups
$T_\alpha$, unless it is internal to $\Gamma$.  Moreover, the groups
$T_\alpha$ are shown to form a definable direct limit system, so that
the group is described by three ingredients: $\Gamma$-groups, stably
dominated groups, and definable direct limits of groups.

We then apply the theory to groups definable in algebraically closed
valued fields. Already the case of Abelian varieties is of
considerable interest; all three ingredients above occur, and the
description beginning with a definable map into a piecewise linear
definable group $L$ takes a different aspect than the classical one.
The points of $L$ over a non-Archimedean local field will form a
finite group, related to the group of connected components in the
Néron model. On the other hand over $\Rr(t)$ the same formulas will
give tori over $\Rr$.

Non-commutative groups definable in algebraically closed valued fields
include examples such as $\GL_n(\Oo)$, $\GL_n(\Oo/a\Oo)$ and
``congruence subgroups'' such as the kernel of $\GL_n(\Oo) \to
\GL_n(\Oo/a\Oo)$, which are all stably dominated. In
Corollary\ \ref{C:embed2} and Theorem\ \ref{T:genchar3} we relate
stably dominated groups to group schemes over \(\Oo\). Previous work
in this direction used other theories, inspired by topology; see
\cite{Pil-FieldsQp}, \cite{HruPil-GpPFF} regarding the
$p$-adics.\medskip

We now describe the results in more detail. Our main notion will be
that of a stably dominated group.  In a metastable theory, we will try
to analyze arbitrary groups, and to some extent types, using them.  In
the case of valued fields, this notion is related to but distinct from
compactness (over those fields where topological notions make sense,
i.e. local fields). For groups defined over local fields, stable
domination implies compactness of the group of points over every
finite extension. Abelian varieties with bad reduction show that the
converse fails; this failure is explained by another aspect of the
theory, definable homomorphic quotients defined over $\Gamma$.

Let $T$ be a first order theory.  It is convenient to view a
projective system of definable sets $D_i$ as a single object, a
pro-definable set; similarly for a compatible system of definable maps
$\alpha_i: P \to D_i$. Precise definitions of the terminology and
notations used in the following paragraphs can be found in
\S\ref{S:stdom} and \S\ref{S:ranks}.

\begin{defn}[Stable domination]
  A type $p$ over $C$ is {\em stably dominated} if there exists, over
  $C$, a pro-definable map $\alpha: p \to D$, $D$ stable and stably
  embedded, such that for any \(a\models p\) and tuple $b$,
  $\alpha(a)\ind_C D \meet \dcl(b)$ implies:
  \[\tp(b/C\alpha(a)) \models \tp(b/Ca).\]
\end{defn}

Let $\Gamma$ be a stably embedded definable set which is orthogonal to
the stable part: there is no definable function with infinite image from a power of $\Gamma$ to a stable definable set. In this paper we will mostly be focusing on the case
where $\Gamma$ is $o$-minimal, as is the case in algebraically closed
valued fields. Many of the results presented here should remain true
in Henselian valued fields with algebraically closed residue field and
different value groups, such as $\Zz$; where every definable subset of
$\Gamma$ is still a Boolean combination of $\emptyset$-definable sets
and intervals.

\begin{defn}[Metastability]
  A theory $T$ is {\em metastable} (over $\Gamma$) if:
\begin{itemize}
\item[(B)] Any set of parameters
  $C_0$ is included in a set $C$, called a metastability basis, such
  that for any $a$, there exists a pro-definable map \(\gamma\) (over C), from $p$ to some power of $\Gamma$ with $\tp(a/\gamma(a))$ stably
  dominated.
 \item[(E)] Every type over an algebraically closed subset (including imaginaries) of a
      model of $T$ has an automorphism-invariant extension to the
      model.
  \end{itemize}
\end{defn}

The main reason we require property (E) is because of descent
(Proposition\ \ref{P:base change}.(\ref{descent})) which is not known
to hold without the the additional hypothesis that \(\tp(B/C)\) has a
global \(C\)-invariant extension.

\begin{question}
  \begin{enumerate}
  \item Can descent be proved without the additional hypothesis that
    \(\tp(B/C)\) has a global \(C\)-invariant extension?

  \item If we assume \(\Gamma\) to be \(o\)-minimal, does property (E) follow from the existence of
    metastability bases?
  \end{enumerate}
\end{question}

\begin{rem}
  \begin{enumerate}
  \item Instead of considering all \(C\)-definable stable sets, it is
    possible, in the definition of stable domination, to take a proper
    subfamily $\cS_C$ with reasonable closure properties. The notion
    of $\cS$-domination is meaningful even for stable theories.
  \item In the definition of metastability, one can also replace the
    single sort $\Gamma$ with a family of sorts $\Gamma_i$, or with a
    family of parametrized families of definable sets $\cG_A$, with no
    loss for the results of the present paper.
  \end{enumerate}
 \end{rem}  

As in stability theory, a range of finiteness assumptions is
possible. Let \(T\) be metastable over \(\Gamma\).

\begin{defn}[{cf. Definition\ \ref{D:AFD}}]
The theory \(T\) has \AFD if:
\begin{enumerate}
\item $\Gamma$ is $o$-minimal.
\item Morley dimension  is uniformly finite and  definable in families.
\item Let $D$ be a definable set. The Morley dimension of $f(D)$,
  where $f$ ranges over all definable functions with parameters such
  that $f(D)$ is stable, takes a maximum value.
\item Similarly, the $o$-minimal dimension of $g(D)$, where $g$ ranges
  over all definable functions with parameters such that $g(D)$ is
  $\Gamma$-internal, takes a maximum value.
\end{enumerate}
\end{defn}

Let $A \leq M \models T$. We define $\St_A$ to be the family of all
stable, stably embedded $A$-definable sets. Some statements will be
simpler if we also assume:

\begin{defn}[{cf Definition\ \ref{D:Aom}}]
The theory \(T\) has \Aom if, in addition to \AFD, any countable set is
contained in a metastability basis $M$ which is an \(\aleph_1\)-saturated model.
Moreover, for any \(\acl\)-finitely generated $G \subseteq \Gamma$ and $S
\subseteq \St_{M\cup G}$ over $M$, isolated types over $M \cup S$ are dense.
\end{defn}

\begin{rem}
\begin{enumerate}
\item \AFD and \Aom both hold in \(\ACVF\), with all imaginary sorts
  included (cf Proposition\ \ref{P:ACVF hyp}). \AFD, at least, is
  valid for all $C$-minimal expansions of \(\ACVF\), in particular the
  rigid analytic expansions
  (cf. \cite{HasMac-Cmin,Lip-ACVFAn,LipRob-ACVFAn}).
\item Existentially closed valued fields with a contractive derivative
  (cf. \cite{Rid-VDF}) and separably closed valued fields
  (cf. \cite{HilKamRid}) are also metastable, but of infinite rank. The
  results presented here will only hold in finite rank definable
  groups.
\item In practice, the main structural results will use finite weight
  hypotheses, see Definition\ \ref{D:weight}; this is a weaker consequence of \AFD.
\end{enumerate}
\end{rem}

A group is {\em stably dominated} if it has a generic type which is
stably dominated (See Definitions \ref{D:gentype} and
\ref{D:stdomgp}). In this case, the stable domination is witnessed by
a group homomorphism (cf. Proposition\ \ref{P:groupdom}). One cannot
expect every group to be stably dominated. But one can hope to shed
light on any definable group by studying the stably dominated groups
inside it. We formulate the notion of a {\em limit stably dominated}
group: it is a direct limit of connected metastable groups by a
pro-definable direct limit system (cf. Definition\ \ref{D:lmstd}).
 
\begin{thm}[{cf. Theorem\ \ref{T:lmstd}}]
\label{T:lmstd1}
Let $T$ be a metastable theory with \Aom. Let $A$ be an interpretable Abelian
group. Then there exists a definable group $\Lambda \subset\Gamma^n$, and a
definable homomorphism $\lambda: A \to \Lambda$, with $H := \ker(\lambda)$ limit
stably dominated.
\end{thm}

In fact, under these assumptions, $H$ is the union of a definable
directed family of definable groups, each of which is stably
dominated.  Assuming only bounded weight (in place of \Aom), we obtain
a similar result but with $H$ $\infty$-definable.

In the non-Abelian case the question remains open.  
The optimal conjecture would be a positive answer to:  

\begin{problem}\label{Pb:double}
\Aom Does any definable group $G$ have a limit stably dominated
definable subgroup $H$ with $H \backslash G / H$ internal to $\Gamma$?
\end{problem}

Another goal of this paper is to relate definable groups in \(\ACVF\)
to group schemes over $\Oo$.  We recall the analogous results for the
algebraic and real semi-algebraic cases.  Consider a field $K$ of
characteristic zero. Then the natural functor from the category of
algebraic groups to the category of constructible groups is an
equivalence of categories.  This follows locally from Weil's group
chunk theorem; nevertheless some additional technique is needed to
complete the theorem. It was conjectured by Poizat and proved in
\cite{vdD-GpCh} using definable topological manifolds, and in
\cite{Hru-UniDim} using stability theoretic notions like definable
types and germs. This methods will be explained in
\S\ref{S:DefGen}. Let us only remark here that an irreducible
algebraic variety has a unique generic behavior, in that any definable
subset has lower dimension or a complement of lower dimension; this is
typical of stable theories.

In \(\ACVF\), we certainly cannot hope every definable group to be a
subgroup of an algebraic group or even of the definable homomorphic
image of an algebraic group. We can, however, hope for all definable
groups to be towers of such groups:

\begin{problem}
  Let $G$ be a definable group in \(K\models\ACVF\). Do there exist
  definable normal subgroups $(1) = G_0 \leq \cdots \leq G_n = G$ of
  $G$ and definable homomorphisms $f_i$, with kernel $G_i$, from
  $G_{i+1}$ into the definable homomorphic image of an algebraic group
  over $K$?
\end{problem}

Call a definable set $D$ {\em boundedly imaginary} if there exists no
definable map with parameters from $D$ onto an unbounded subset of
$\Gamma$. If $D$ is defined over a local field $L$, then $D$ is
boundedly imaginary if and only if $D(L')$ is finite for every finite
extension $L'$ of $L$. We prove:

\begin{prop}[{cf. Corollary\ \ref{C:embed2}}]\label{P:embed1}
Let $H$ be a stably dominated connected group definable in
$K\models \ACVF$. Then there exists an algebraic group $G$ over $K$ and a
definable homomorphism $f: H \to G(K)$ with boundedly imaginary
kernel.
\end{prop}

\begin{cor}
Let $H$ be a stably dominated group definable in $\ACVF$ with
parameters from a local field \(L\). Then there exists a definable
homomorphism $f: H(L) \to G(L)$ with $G$ an algebraic group over $L$,
with finite kernel.
\end{cor}
 
Proposition\ \ref{P:embed1} reduces, up to a boundedly imaginary kernel, the study of a stably dominated groups definable in \(\ACVF\) to that of stably
dominated subgroups of algebraic groups $G$. We proceed then to
describe these.

There exists an exact sequence $1 \to A \to G \to_f L \to 1$, with $A$
an Abelian variety and $L$ an affine algebraic group. We show
(cf. Lemma\ \ref{L:stdomext} and Corollary\ \ref{C:stdomsubgpvf}) that
a definable subgroup $H$ of $G$ is stably dominated if and only
if $H \meet A$ and $f(H)$ are. For linear groups, we have:

\begin{thm}[{cf. Theorem\ \ref{T:genchar3}}]\label{T;structure}
Let $G$ be an affine algebraic group and let $H$ be a stably dominated
definable subgroup of $G$. Then $H$ is isomorphic to $\Hh(\Oo)$, $\Hh$
a group scheme of finite type over $\Oo$.

If $H$ is Zariski dense in $G$, $\Hh$ can be taken to be $K$-isomorphic to $G$.
\end{thm}
 
Finally, we show that enough of the structure of definable groups in
\(\ACVF\) is known to be able to classify all the fields: as expected,
they are all isomorphic either to the valued field itself or to its
residue field (cf. Theorem\ \ref{T:acvf-fields}).\medskip

Let us conclude this introduction with a series of examples that
illustrate some of the structure results proved in this paper. Since this structure is reflected in the generics of the groups, we also discuss these generics (cf. Section 3 for definitions).

\begin{example}[Structure and generics of certain algebraic groups]
Let $K$ be an algebraically closed valued field.
\begin{enumerate} 
\item $\SL_n(\Oo)$ is stably dominated. Its unique generic is stably dominated via the residue map.
\item $\Gm(K)$ has a largest stably dominated subgroup $\Gm(\Oo)$ and \(\Gm(K)/\Gm(\Oo)\simeq\Gamma\). It has two generics: elements of small valuation
  and elements of large valuation (i.e. infinitesimals), both generics are attributable to $\Gamma$.
\item $\Ga(K)$ is a limit stably dominated group, it is the union of the stably dominated subgroups $\alpha \Oo$. It has a unique generic: elements of small
  valuation. Here the generic types correspond to cofinal definable types in the
  (partially) ordered set $(\Gamma, >)$ indexing the stably dominated
  subgroups $\alpha \Oo$. The fact that the poset has a unique cofinal
  definable type is however a phenomenon of dimension one.
  
  The limit stably dominated group $\Ga^2(K)$, however, has a large family of generics. For any
  definable curve $E$ in $\Gamma^2$, cofinal in the sense that for any
  $(a,b)$ there exists $(d,e) \in E$ with $d<a$ and $e<b$, there exists a
  generic type of $\Ga^2(K)$ whose projection to $\Gamma^2$
  concentrates on $E$.
\item Let $B_n \leq \SL_n(K)$ be the solvable group of upper diagonal matrices. The valuation of the diagonal coefficients is a group homomorphism into \(\Gamma^{n-1}\) with
  limit stably dominated kernel $U_n$.
  
  The group $B_n$ has left, right as well as two sided generics. One of
  the latter is given as follows. Let $(x_{i,j})$ be the matrix
  coefficients of an element of $x \in \GL_n(K)$. A two-sided generic
  of \(G\) is determined by: $\val(x_{i,j}) << \val(x_{i',j'})$ when
  $(i,j)< (i',j')$ lexicographically (and $i<j$).
\item For $n>1$, $\SL_n(K)$ is simple but neither limit stably dominated nor $\Gamma$-internal. However, the double coset space of $U_n\backslash\SL_n(K)/U_n$ is a disjoint union of \(n!\) copies of \(\Gamma^{n-1}\). The group $\SL_n(K)$ has no generic types.
\end{enumerate}
\end{example}

Parts of this text served as notes for a graduate seminar, given by the first
author, in the Hebrew University in Fall 2003; the participants have the
authors' warm thanks. The authors are also very grateful to the referee for his
two friendly and highly conscientious reports, written more than a decade apart.
The second author would like to express his profound gratitude to Reid Dale and
all the participants in the two Paris reading groups with whom he spent many
hours reading earlier drafts of this paper.

Lastly, both authors are deeply indebted to Paul Wang for spotting a gap in the
published version of this paper. The present version contains the modifications
of the corrigendum \cite{HRW-Meta} written with his collaboration.

The first author was supported by the European Research Council under
the European Union's Seventh Framework Programme (FP7/2007-2013) / ERC
Grant agreement no.\ 291111/ MODAG. The second author was partially
supported by ValCoMo (ANR-13-BS01-0006)

\section{Preliminaries}\label{S:prelim}

We recall some material from stability, stable domination,
\(o\)-minimality and the model theory of algebraically closed valued
fields, in a form suitable for our purposes.

We fix a theory \(T\) that eliminates imaginaries and a universal
model \(\Uu\) (a sufficiently saturated and homogeneous model). Any
model of \(T\) that we consider will be an elementary submodel of
\(\Uu\). We write $\models \phi$ as a shorthand for $\Uu \models
\phi$. By a \emph{definable} set or function we mean one defined in
$T_A$ for some $A \subseteq \Uu$.  If we wish to specify the base of
definition, we say $A$-definable.

A pro-definable set \(X\) is a (small) projective filtered system
\((X_i)_{i\in I}\) of definable sets and definable maps. We think of
\(X\) as \(\projlim_i X_i\). Pro-definable sets can equivalently be
presented as a collection of formulas in potentially infinitely many
variables. If all the maps in the system describing \(X\) are
injective, we say that \(X\) is \(\infty\)-definable and we can
identify \(X\) with a subset of any of the \(X_i\). If all the maps in
the system describing \(X\) are surjective, we say that \(X\) is
strict pro-definable; equivalently, the projection of \(X\) to any
\(X_i\) is definable. A relatively definable subset of \(X\) is a
definable subset of one of the \(X_i\); we identify it as a subset of
\(X\) by pulling it back to \(X\). Equivalently, if \(x = (x_i)\) is a
tuple of variables where each \(x_i\) ranges over \(X_i\), a
relatively definable subset of \(X\) is the set of realizations in
\(X\) of some formula \(\phi(x)\).

A pro-definable map \(f : \projlim_i X_i \to Y\) where \(Y\) is
definable is a definable map from some \(X_i\) to \(Y\). A
pro-definable map \(f:\projlim_i X_i \to \projlim_j Y_j\) is a
compatible collection of maps \(f_j : \projlim_i X_i \to Y_j\). A
pro-definable subset of a pro-definable set \(X = \projlim_i X_i\) is
a pro-definable set \(Y\) with an injective pro-definable map \(f:
Y\to X\); it can be identified with a compatible collection of
\(\infty\)-definable subsets of the \(X_i\), or with an intersection
of relatively definable subsets of \(X\).

More generally, by a {\em piecewise pro-definable} set, we mean a
family $(X_i)_{i\in I}$ of pro-definable sets over a directed order $I$,
with a compatible system of injective pro-definable maps $X_i \to
X_{j}$, for all $i\leq j$, viewed as inclusion maps. The direct limit
is thus identified with the union, and denoted $X = \bigcup_i X_i$. A
pro-definable subset of $X$ is a pro-definable subset of one of the
$X_i$.

\subsection{Definable types and filters}\label{S:def filter}

A type will mean a complete type in possibly infinitely many variables
(equivalently an ultrafilter on the Boolean algebra of relatively
definable sets). We will refer to possibly incomplete types as
filters.

A filter \(\pi(x)\) over \(\Uu\) in the (possibly infinite) tuple of
variables \(x\) is said to be \(C\)\emph{-definable} if for all
formulas \(\phi(x,y)\), there exists a formula \(\theta(y) =:
\defsc{\pi}{x}\phi(x,y)\) over \(C\), such that \(\pi = \{\phi(x,m)
\mid \Uu\models\defsc{\pi}{x}\phi(x,m)\}\). Note that \(\pi\) is
completely determined by the map \(\phi(x,y) \mapsto
\defsc{\pi}{x}\phi(x,y)\), its \emph{definition scheme}. If \(\pi\) is
a type over \(\Uu\), then this map is a Boolean homomorphism.

If \(C\) is a set of parameters, we define \(\pi|C := \{\phi(x,m) \mid
\Uu\models\defsc{\pi}{x}\phi(x,m)\text{ and }m\in C\}\). Note that if
\(C\) is a model, \(\pi|C\) completely determines \(\pi\).

We say that \(\pi\) \emph{concentrates} on a pro-definable set \(D\)
if \(\pi(x)\vdash x\in D\). By a \emph{pro-definable function on
  \(\pi\)}, we mean a pro-definable map \(f : D \to X\) such that
\(\pi\) concentrates on \(D\).

\begin{defn}\label{D:pushfwd}
If $f$ is a pro-definable function on a $C$-definable filter $\pi$, we
define the push-forward $f_\star \pi$ by:
\[\defsc{f_\star\pi}{u} \phi(u,v) = \defsc{\pi}{x} \phi(f(x),v).\]
\end{defn}

So that if $a \models \pi | C$ then $ f(a) \models f_\star\pi |C$.

\begin{defn}\label{D:freeprod}
Let $\pi(x)$, $\mu(y)$ be \(C\)-definable filters. Define $\rho(x,y) =: \pi(x)
\tensor \mu(y)$ by:
\[\defsc{\rho}{x y} \phi(x,y,z) = \defsc{\pi}{x}\defsc{\mu}{y} \phi(x,y,z).\]
\end{defn}
Then, if \(a\models \pi|C\) and \(b\models \mu|Ca\),
\(ab\models(\pi\tensor\mu)|C\). If \(\pi\) and \(\mu\) are types, then
\(f_\star\pi\) and \(\pi\tensor\mu\) are types.

We will occasionally use a more general construction. Assume $p(x)$
is a $C$-definable type. Let $a \models p$ and let $q_a(y)$ be
a \(Ca\)-definable type of the theory $T_{Ca}$.

\begin{lem}\label{L:concat}
  There exists a unique definable type $r(x y)$ such that for any
  $B\supseteq C$, if $ab \models r |B$ then $a \models p |B$ and $b
  \models q_a | Ba$.
\end{lem}

\begin{proof} Given a formula $\phi(x y,z)$, let $\psi(x,z)$ be a formula
such that $\psi(a,z) = \defsc{q_a}{y}\phi(a,y,z)$. The formula $\psi$
is not uniquely defined, but if $\psi$ and $\psi'$ are two
possibilities then $\defsc{p}{x}(\psi \leftrightarrow
\psi')$. Therefore we can define:
\[\defsc{r}{x y}\phi(x y,z) = \defsc{p}{x}\psi(x,y,z).\]

It is easy to check that this definition scheme works.
\end{proof}

In fact, over \(C = \acl(C)\), it suffices that $q_a$ be definable
over $\acl(Ca)$. This follows from:

\begin{lem}\label{L:def2}
  Let $M$ be a model and let \(C=\acl(C)\subseteq M\). Let $\tp(a/M)$
  be $C$-definable. Let $c \in \acl(Ca)$.  Then $\tp(ac/M)$ is
  $C$-definable. Indeed, $\tp(a/M) \cup \tp(ac/C) \vdash \tp(ac/M)$.
\end{lem}

\begin{proof}
Let $\phi(x,y)$ be a formula over $C$ such that $\phi(a,c)$ holds, and
such that $\phi(a,y)$ has $k$ solutions. Let \(\psi(x,y,z)\) be any
formula. Then the equivalence relation \(z_1 E z_2\) defined by
\(\defsc{p}{x}(\forall y)\phi(x,y)\implies(\psi(x,y,z_1)
\Leftrightarrow \psi(x,y,z_2))\) where \(p := \tp(a/M)\) is
\(C\)-definable and has at most \(2^k\) classes. It follows that each
of these classes, denote \(Z_i\), are \(C\)-definable. Finally, for
all \(m\), \(\models\psi(a,c,m)\) if and only if \(m\in Z_i\) for some
\(i\) such that \((\forall z \in Z_i)\psi(x,y,z) \in \tp(ac/C)\).
\end{proof}

See Proposition\ \ref{P:stdom}.(\ref{aclstdom}) for a stronger statement
in the stably dominated case.

\begin{defn}[Germs of definable functions]
  Let $p$ be a definable type. Two definable functions $f(x,b)$ and
  $g(x,b')$ are said to have the same \emph{$p$-germ}
  if \[\models\defsc{p}{x}f(x,b)=g(x,b').\]
\end{defn}
 
We say that {\em the $p$-germ of $f(x,b)$ is defined over $C$} if
whenever $\tp(b/C)=\tp(b'/C)$, $f(x,b)$ and $f(x,b')$ have the same
$p$-germ.  Note that the equivalence relation $b \sim b'$ defined by ``$f(x,b),f(x,b')$ have the same $p$-germ'' is definable; the $p$-germ of $f(x,b)$ is defined over $C$ if and only if $b/{\sim} \in \dcl(C)$.

\subsection{Stable domination}\label{S:stdom}

\begin{defn}
Let \(D\) be a \(C\)-definable set.
\begin{enumerate}
\item \(D\) is said to be \emph{stably embedded} if any definable
  \(X\subseteq D^n\), for some \(n\), is \(C\cup D\)-definable.
\item \(D\) is said to be \emph{stable} if every formula
  $\phi(x;y)$ with parameters in \(C\), implying \(x\in D^n\) for some
  \(n\), is a stable formula.
\end{enumerate}
\end{defn}

Since being stable implies being stably embedded, the latter is often
referred to as being \emph{stable, stably embedded} in the
literature. See e.g. \cite{Pil-Stab} for a treatment of basic
stability and \cite[Appendix]{ChaHru-ACFA} for the properties of
stably embedded and stable sets.

For all \(C\), let \((D_i)_i\) enumerate all \(C\)-definable stable
sets and \(\St_{C} := \prod_i D_i\). It is a strict pro-definable
set. We often consider \(\St_{C}\) as a (stable) structure, whose
sorts are the \(D_i\), with the full induced structure. By a
pro-definable map into \(\St_{C}\), we mean (somewhat abusively) a
collection of maps \(f = (f_j)_j\) where the range of each \(f_j\) is
stable. For all \(b\), let \(\St_{C}(b)\) denote
\(\dcl(Cb)\cap\St_{C}\).

\begin{notation}\label{N:ind}
For all \(a\), \(b\), \(C\), we write \(a\ind_C b\) if there exists an
\(\acl(C)\)-definable type \(p\) such that \(a \models p |\acl(Cb)\).
\end{notation}

Note that if \(a\), \(b\), \(C\in \St_D\) for some \(D\subseteq C\)
then \(a\ind_C b\) if and only if they are (forking) independent in
\(\St_D\) over \(C\). Let us now recall the definition of stable domination:

\begin{defn}
Let \(p = \tp(a/C)\) and \(\alpha : p \to \St_{C}\) be a
pro-\(C\)-definable map. The type \(p\) is said to be \emph{stably
  dominated} via \(\alpha = (\alpha_i)_i\) if for any tuple \(b\), if
\(\St_{C}(b)\ind_{C} \alpha(a)\), then \(\tp(b/C\alpha(a))\vdash
\tp(b/Ca)\).
\end{defn}

A type \(p\) over \(C\) is said to be stably dominated if it stably
dominated via some \(C\)-definable map \(\alpha : p\to \St_C\). For
\(a\models p\), let \(\theta_C(a)\) enumerate \(\St_C(a)\), then \(p\)
is stably dominated if and only if it is stably dominated via
\(\theta_C\).

Let us now recall some of the result from \cite{HasHruMac-Book}
regarding stable domination.

\begin{prop}[{\cite[Corollary\ 3.31.(iii) and Proposition\ 3.13]{HasHruMac-Book}}]\label{P:stdomdef}
For all \(a\) and \(C\),
\begin{enumerate}
\item \(\tp(a/C)\) is stably dominated if and only if
  \(\tp(a/\acl(C))\) is.
\item If \(C = \acl(C)\) and \(\tp(a/C)\) is stably dominated via
  \(f\), then \(\tp(a/C)\) has a unique \(C\)-definable extension
  \(p\). Moreover, for all \(B\supseteq C\), \(a\models p|B\) if and
  only if \(\St_C(B)\ind_C f(a)\).
\end{enumerate}
\end{prop}

\begin{prop}\label{P:stdom}
Assume \(\tp(a/C)\) is stably dominated.
\begin{enumerate}
\item\label{sym} If \(q\) is a global \(\acl(C)\)-definable type and \(b\models
  q|\acl(C)\), then \(a \ind_C b\) implies \(b\ind_C a\). In particular, if
  \(\tp(b/C)\) is stably dominated, \(a \ind_C b\) if and only if
  \(b\ind_C a\).
\item\label{rtrans} \(a\ind_C b d\) if and only if \(a\ind_C b\) and
  \(a\ind_{Cb}d\).
\item\label{ltrans} If \(\tp(b/Ca)\) is stably dominated, then so is
  \(\tp(ab/C)\).
\item\label{aclstdom} If \(b\in\acl(Ca)\), then \(\tp(b/C)\) is stably
  dominated.
\end{enumerate}
\end{prop}

\begin{proof}
(1) and (2) are easy to check. (3) is \cite[Proposition\ 6.11]{HasHruMac-Book} and (4) is \cite[Corollary\ 6.12]{HasHruMac-Book}.
\end{proof}

Let \(p\) be a global \(C\)-invariant type. We say that \(p\) is stably dominated over \(C\) is \(p|C\) is stably dominated.

\begin{prop}[{\cite[Proposition\ 4.1 and
        Theorem\ 4.9]{HasHruMac-Book}}]\label{P:base change} Let
  \(B\supseteq C\) and \(p\) be a global \(C\)-invariant type.
\begin{enumerate}
\item Let \(f\) be a pro-\(C\)-definable function. If \(p\) is stably
  dominated over \(C\) via \(f\), then it is also stably dominated
  over \(B\) via \(f\).
\item\label{descent} If \(p\) is stably dominated over \(B\) and \(\tp(B/C)\) has a global \(C\)-invariant extension, then \(p\) is stably dominated over \(C\).
\end{enumerate}
\end{prop}

\begin{prop}[Strong germs,
    {\cite[Theorem\ 6.3]{HasHruMac-Book}}]\label{P:stgerms}
Let \(C = \acl(C)\), \(p\) be a global \(C\)-invariant type stably dominated over \(C\) and \(f\) be a definable function defined at \(p\).
\begin{enumerate}
\item The \(p\)-germ of \(f\) is \emph{strong}; i.e. there exists a \(\germ{p}{f}C\)-definable function \(g\) with the same \(p\)-germ as \(f\).
\item If \(f(a)\in\St_{C}\) for \(a\models p\), then \(\germ{p}{f}\in\St_C\).
\end{enumerate}
\end{prop}

Let us now recall the definition of metastability. Let \(\Gamma\) be
an \(\emptyset\)-definable stably embedded set. We will also assume
that $\Gamma$ is orthogonal to the stable part: no infinite definable
subset of $\eq{\Gamma}$ is stable. For any \(C\) and \(a\), let
\(\Gamma_C(a)\) denote \((C\cup\eq{\Gamma})\cap\dcl(Ca)\).

\begin{defn}\label{def:meta}
The theory \(T\) is \emph{metastable} (over \(\Gamma\)) if, for any \(C\):
\begin{enumerate}
\item (Metastability bases) There exists \(D\supseteq C\) such that
  for any tuple \(a\), \(\tp(a/\Gamma_D(a))\) is stably dominated.
\item (Invariant extension property) If \(C = \acl(C)\), then, for all
  tuple \(a\), there exists a global \(C\)-invariant type \(p\) such
  that \(a\models p|C\).
\end{enumerate}
\end{defn}

A \(D\) such as in (1) is called a metastability basis.

\begin{rem}\label{R:add Gamma}
  Since \(\Gamma\) is orthogonal to the stable part, if \(\tp(a/C)\)
  is stably dominated, so is \(\tp(a/C\gamma)\) for any tuple
  \(\gamma\in\Gamma\). It follows that if \(C\) is a metastability
  basis, so is \(C\gamma\).
\end{rem}

\begin{prop}[Orthogonality to \(\Gamma\), {\cite[Corollary\ 10.8]{HasHruMac-Book}}]\label{P:orth}
Assume that the theory \(T\) is metastable over \(\Gamma\). A global
type \(C\)-invariant type \(p\) is stably dominated if and only if for
any definable map \(g:p\to \eq{\Gamma}\), the \(p\)-germ of \(g\) is
constant (equivalently, \(g_\star p\) is a realized type).
\end{prop}

Write $g(p)$ for the constant value of the $p$-germ. The property of
$p$ in the proposition above is referred to as {\em orthogonality of
  $p$ to $\Gamma$}.  %Note that this is strictly weaker than orthogonality of $D$ to $\Gamma$ for some definable $D \in p$.

\subsection{Ranks and weights}\label{S:ranks}

Let us now recall our ``finite rank'' assumptions. But, first let us
recall the definition of internality.

\begin{defn}
  A definable set $X$ is $\Gamma$-{\em internal} if $X \subseteq
  \dcl(F\Gamma)$ for some finite set $F$; equivalently for any $M
  \prec M' \models T$, $X(M') \subseteq \dcl(M\Gamma(M'))$.

  The same condition with $\acl$ replacing $\dcl$ is called
  \emph{almost internality}.
\end{defn}

Thus a definable $X$ is almost $\Gamma$-internal if over some finite
set $F$ of parameters and for some finite $m$, $X$ admits a definable
$m$-to-one function $f: X \to Y/E$ for some definable $Y \subseteq
\dcl(\Gamma^n)$ and definable equivalence relation $E$. When $\Gamma$
eliminates imaginaries, $E$ does not need to be mentioned. Note that
$X$ is $\Gamma$-internal whenever we can choose $m = 1$.

A definable set \(D\) is called \(o\)-\emph{minimal} if there exists a
definable linear ordering on \(D\) such that every definable subset of
\(D\) (with parameters) is a finite union of intervals and
points. There is a natural notion of dimension for definable subsets
of $D^m$, such that $D$ is dimension one (see \cite{vdD-book}).

\begin{defn}\label{D:AFD}
\(T\) has \AFD if:
\begin{enumerate}
\item \(\Gamma\) is \(o\)-minimal.
\item Morley rank (in the stable part, denoted \(\RM\)) is uniformly
  finite and definable in families: if \((D_t)_t\) is a definable
  family of definable sets, then, when \(D_t\) is stable, the Morley
  rank of \(D_t\) is finite and bounded uniformly in \(t\) and for all
  \(m\in\Zz_{\geq 0}\), \(\{t : \RM(D_t) = m\}\) is definable.
\item For all definable \(D\), the Morley rank of \(f(D)\), where \(f\)
  ranges over all functions definable with parameters whose range is
  stable, takes a maximum value \(\dimst(D)\in\Zz_{\geq 0}\).
\item For all definable \(D\), the \(o\)-minimal dimension of
  \(f(D)\), where \(f\) ranges over all functions definable with
  parameters whose range is \(\Gamma\)-internal, takes a maximum value
  \(\dimo(D)\in\Zz_{\geq 0}\).
\end{enumerate}
\end{defn}

Note that if \(\Gamma\) is an \(o\)-minimal group, it
eliminates imaginaries. So, in (4), we could equally well ask that
$f(D) \subseteq \Gamma^n$ for some $n$.

\begin{defn}
  If \(X = \projlim_i X_i\) is pro-definable, we define \(\dimst(X) =
  \max\{\dimst(X_i)\mid i\}\in\Zz_{\geq 0}\cup\{\infty\}\) and \(\dimo(X) = \max\{\dimo(X_i)\mid
  i\}\in\Zz_{\geq 0}\cup\{\infty\}\).

  We define $\dimst(a/B) := \min \{\dimst(D): a \in D\text{ and $D$ is
    $B$-definable}\}$ and $\dimo(a/B) := \min \{\dimo(D): a \in
  D\text{ and $D$ is $B$-definable}\}$.
\end{defn}

Note that if \(D\) is a stable definable set then \(\dimst(D) =
\RM(D)\) and, similarly, for all tuples \(a\in \St_C\), \(\dimst(a/C)
= \RM(a/C)\). Note that, in general, we may have $\dimst(d/B) >
\dimst(\St_B(d)/B)$. Similar statements hold for \(\dimo\).

A set $A$ is \(\acl\)-finitely generated over $B \subseteq A$ if \(A
\subseteq \acl(Ba)\) for some finite tuple $a$ from $A$. If \(c
\in\acl(Ba)\) is a finite tuple, then \(\dimst(ca/B) = \dimst(a/B)\).
In particular, it makes sense to speak of the stable dimension
(respectively the \(o\)-minimal dimension) over \(B\) of any set finitely
\(\acl\)-generated over \(B\).

\begin{lem}\label{L:afg}\AFD
Let \(C\) be a set of parameters and $a$ be a finite tuple. Then
$\Gamma_C(a)$ and $\St_C(a)$ are \(\acl\)-finitely generated over
$C$.
\end{lem}

\begin{proof}
  Any tuple $d \in \St_C(a)$ can be written $d=h(a)$ for some
  $C$-definable function \(h\). By the dimension bound in \AFD,
  $\dimst(d/C) \leq \dimst(a/C)$. It follows that if $d_i \in
  \St_C(a)$ then for sufficiently large \(n\), $d_n \in
  \acl(Cd_1,\ldots,d_{n-1})$. So $\St_C(a)$ is finitely
  \(\acl\)-generated. The proof for \(\Gamma_C(a)\) is the same.
\end{proof}

\begin{lem}\label{L:nfcp}
  \AFD Let $f: P \to Q $ be a definable map between definable sets $P$
  and $Q$.  Let $P_a = f^{-1}(a)$. Then there exists $m$ such that
  if $P_a$ is finite, then $|P_a| \leq m$.
\end{lem}

\begin{proof}
  Say $f,P,Q$ are $\emptyset$-definable. By compactness, it suffices
  to show that if $P_a$ is infinite then there exists a
  $\emptyset$-definable set $Q'$ with $a \in Q'$ and $P_b$ infinite
  for all $b \in Q'$.

  \begin{claim}
    If $P_a$ is infinite then either $\dimst(P_a) >0$ or $\dimo(P_a) >
    0$.
  \end{claim}

  \begin{proof}
    Let $M$ be a metastability base, with $a \in M$. Let $c \in P_a \setminus
    M$.  If $\Gamma_{M}(c) \neq \Gamma(M)$ then $\dimo(P_a) > 0$.
    Otherwise, by metastability, $\tp(c/M)$ is stably dominated, say
    via $f$.  If $f(c) \in M$, then $\tp(b/M) \vdash \tp(b/M c)$ for
    all $b$, and taking $b=c$ it follows that $c \in M$.  Thus $f(c)
    \in \St_M \setminus M$. It follows that $\dimst(P_a)>0$.
  \end{proof}

  If $\dimst(P_a) >0$, then there exists a definable family of stable
  definable sets $D_t$ with $\dimst(D_t) = k > 0$, and a definable
  function $f(x,u)$ such that for some $b$ and $t$, $f(P_a,b) = D_t$.
  Then the formula $(\exists u)(\exists t) f(P_y,u) = D_t$ is true of
  $y=a$, and implies that $P_y$ is infinite. The case $\dimo(P_a)>0$
  is similar.
\end{proof}

\begin{defn}
\label{D:Aom}
\(T\) has \Aom if, in addition to \AFD, any \emph{countable} \(C\) is contained
in a metastability basis \(M\) which is an \emph{\(\aleph_1\)-saturated} model
and such that for any finitely \(\acl\)-generated \(G\subseteq \Gamma\) and
\(S\subseteq \St_{M\cup G}\) over \(M\), isolated types over \(M\cup S\) are
dense.
\end{defn}

Say $\tp(a/C)$ is {\em strongly stably dominated} if there exists
$\phi(x,c) \in \tp(a/\St_C(a))$ such that for any tuple $b$ with
$\St_C(b)\ind_C \St_C(a)$, $\tp(a/\St_C(a)b)$ is isolated via $\phi$. In
particular, $\tp(a/\St_C(a))$ is isolated via $\phi$. Moreover,
$\tp(a/C)$ is stably dominated, since $\tp(a/\St_C(a))$ implies $\phi$
which implies $\tp(a/ \St_C(a)b)$ for any $b$ as above.

When $\tp(a/C)$ is strongly stably dominated, the parameters $c$ of
$\phi$ may be written as $c_0h(a)$ for some $C$-definable function $h$
into $\St_C$ and tuple $c_0\in C$. In this situation we say that
$\tp(a/C)$ is strongly stably dominated via $h$. Conversely, if
$\tp(a/C)$ is stably dominated via $h$ and $\tp(a/Ch(a))$ is isolated,
then $\tp(a/C)$ is strongly stably dominated via $h$.

\begin{lem}
\label{iso2}
\Aom Let $D$ be $C_0$-definable, for some countable \(C_0\) and $h: D \to S$ a
definable map to a stable definable set of maximal possible dimension $\dimst(S)
= \dimst(D)$. Then there exists a model $M$ containing $C_0$ and $a \in D$ such
that, with $C := \Gamma_M(a)$, we have $\dimst(D) = \dimst(h(a)/C) =
\dimst(\St_{C}(a)/{C})$, and $\tp(a/C)$ is strongly stably dominated via $h$.
\end{lem}

\begin{proof}
  We may assume that $h$ is defined over $C_0$. Let $a_0\in D$ be such
  that $\dimst(D) = \dimst(h(a_0)/C_0)$. Extend $C_0$ to a
  metastability basis $M$ as in \Aom, with $h(a_0)$ independent from
  $M$ over \(C_0\). Then $ \dimst(D) = \dimst(h(a_0)/M) =
  \dimst(\St_{M}(a_0)/ M)$. Choose $a_0$ such that
  $\dimo(\Gamma_M(a_0)/M)$ is as large as possible---given the other
  constraints. Let $C:=\Gamma_M(a_0)$ and $B := \St_C(a_0)$; then, by
  Lemma\ \ref{L:afg}, $C$ and $B$ are finitely generated over
  $M$. Moreover, $\dimst(h(a_0)/C) = \dimst(D) \geq \dimst(B/C)$ so
  $B\subseteq\acl(Ch(a_0))$. Let $D'$ be $B$-definable set contained
  in $D$, containing $a_0$, and such that for any $a\in D'$, $B\subseteq
  \acl(Ch(a))$. By \Aom, there exists $a \in D'$ such that $\tp(a/B)$
  is isolated. By choice of $D'$ we have $B \subseteq \acl(Ch(a))$. So
  $\tp(B/Ch(a))$ is atomic and, therefore, $\tp(a/Ch(a))$ is isolated.

  Similarly, by maximality of $\dimo(\Gamma_M(a_0) / M)$ and the fact
  that $\Gamma_M(a_0) \subseteq B\cap\Gamma \subseteq
  \acl(Ca)\cap\Gamma\subseteq \Gamma_M(a)$, we have $\Gamma_M(a) = \Gamma_M(a_0)
  =C$. By metastability, $\tp(a/C)$ is stably dominated. But
  $\dimst(\St_C(a)/C)\leq \dimst(D) = \dimst(B/C)$ and $B\subseteq
  \acl(Ch(a))\subseteq \acl(\St_C(a))$, so
  $\St_C(a)\subseteq\acl(B)\subseteq\acl(Ch(a))$ and $\tp(a/C)$ is
  stably dominated via $h$. It follows that $\tp(a/C)$ is strongly
  stably dominated via $h$.
\end{proof}

Finally, let us introduce weight.

\begin{defn}\label{D:weight}
  Let \(p\) be a definable type and \(X\) be a pro-definable set. We
  say that \(X\) has \emph{\(p\)-weight} smaller or equal to \(n\) if
  whenever $b \in X$, $(a_1,\ldots,a_{n}) \models p^{\tensor n}$, we
  have $a_i \models p | b $ for some $i$. The set \(X\) has bounded
  weight if for some $n$, for every stably dominated type \(p\)
  concentrating on \(X\), \(X\) has \(p\)-weight smaller or equal to
  $n$.
\end{defn}

For all definable $D$ and stably dominated types $p$ concentrating on
$D$, the $p$-weight of $D$ is bounded by $\dimst(D)$. It follows that
if $\AFD$ holds, every definable set has bounded weight. In $\ACVF$,
for example, the weight of any definable subset of a variety $V$ is
bounded by the dimension of $V$.

\subsection{\texorpdfstring{Some \(o\)-minimal lemmas}{Some o-minimal lemmas}}

This subsection contains some lemmas on $o$-minimal partial orders and
groups. The former will yield a generic type of limit stably dominated
groups. The latter will be used to improve some statements from
``almost internal'' to ``internal''.

% Let $(P,\leq)$ be a definable partial ordering, and $p$ a definable
% type concentrating on \(P\).  We say $p$ is cofinal if for any $c \in
% P$, $\models \defsc{p}{x} x \geq c$.  Equivalently, for every
% non-cofinal definable $Q \subseteq P$, $\models \neg
% \defsc{p}{x}Q(x)$. The following result can also be found in
% \cite[Lemma\ 4.2.18]{HruLoe}.

\stepcounter{thm}

\begin{lem}\label{L:almost-int}
Let $G$ be a definable group. Assume $G$ is almost internal to a
stably embedded definable set $\Gamma$. Then there exists a finite
normal subgroup $N$ of $G$ with $G/N$ internal to $\Gamma$.
\end{lem}

\begin{proof}
The assumption implies the existence of a definable finite-to-one
function $f: G \to Y$, where $Y \subseteq \dcl(\Gamma)$.  Given a
definable $Y' \subseteq Y$, let $m(Y')$ be the least integer $m$ such
that (possibly over some more parameters), there exists a definable
$m$-to-one map $f^{-1}(Y') \to Z$, for some definable $Z \subseteq
\dcl(\Gamma)$. We may assume that \(f\) is \(m(Y)\)-to-one. Let $I$ be
the family of all definable subsets $Y'$ of $Y$ with ($Y'=\emptyset$
or) $m(Y') < m(Y)$. This is clearly an ideal (closed under finite
unions, and definable subsets). Let $F := \{Y \sminus Y'\mid Y' \in
I\}$ be the dual filter. For $g \in G$, let $D(g)$ be the set of $y
\in Y$ such that for some (necessarily unique) $y'$, $g\cdot f^{-1}(y)
= f^{-1}(y')$; and define $g_\star(y) = y'$. The function $x \mapsto
(f(x),f(g\cdot x))$ shows that $\{y\mid |f (g\cdot f^{-1}(y))| > 1 \}
\in I$. Since \(X := \{y\mid |f^{-1}(y)| = m\} \in F\) and \((Y\sminus
I) \cap X \subseteq D(g)\), it follows that \(D(g)\in F\).

For any definable \(Z\subseteq Y\), let $F|Z = \{W \cap Z: W \in F
\}$. Let $G_0$ be the set of bijections $\phi: Y' \to Y''$ with
$Y',Y'' \in F$, carrying the filter $F|Y'$ to $F|Y''$. Note that
\(g_\star:D(g)\to D(g^{-1})\) lies in \(G_0\). Write $\phi\sim \phi'$ if
$\phi$ and $\phi'$ agree on some common subset of their domains, lying
in $F$; and let $G' = G_0/\sim$. Composition induces a group structure
on $G'$ and we obtain a homomorphism $G \to G'$, $g \mapsto g_\star/
\sim$.  Let $N$ be the kernel of this homomorphism. Let us prove that
\(|N|\leq m(Y)\). For suppose $n_0,\ldots,n_{m}$ are elements of
$N$. Then for some $y \in \cap D(n_i)$ we have $(n_i)_\star(y) = y$.  It
follows that $n_i (f^{-1}(y)) = f^{-1}(y)$. Fix any \(g\in f^{-1}(y)\)
and since $|f^{-1}(y)|=m$, for some $i \neq j$ we have $n_i \cdot g =
n_j \cdot g$ and hence \(n_i = n_j\).

Let \(\sigma\) be an automorphism fixing \(\Gamma\). Then for any \(g\in G\),
since \(g_\star\) is a function between \(\Gamma\)-internal sets,
\(\sigma(g)_\star = \sigma(g_\star) = g_\star\) and hence \(g\cdot N =
\sigma(g)\cdot N\). As $\Gamma$ is stably embedded, it follows that \(G/N\) is
\(\Gamma\)-internal.
\end{proof}

\begin{lem}
\label{L:almost-int-2}
Let $G$ be a definable group. Assume $G$ is almost internal to an \(o\)-minimal
stably embedded definable set $\Gamma$. Then $G$ is $\Gamma$-internal.
\end{lem}

\begin{proof}
By Lemma\ \ref{L:almost-int}, there exists a definable surjective
homomorphism $f: G \to B$ with $B$ a group definable over $\Gamma$,
and $N = \ker(f)$ a group of finite size $n$. Let $B^0$ be the
connected component of the identity in $B$; then $B/B^0$ is finite,
and it suffices to prove the lemma for $f^{-1}(B^0)$.  Assume
therefore that $B$ is connected.
  
If $G$ has a proper definable subgroup $G_1$ of finite index, then
$f(G_1) = B$ by connectedness of $B$. It follows that $N$ is not
contained in $G_1$, so $N_1 = N \meet G_1$ has smaller size than $N$.
Hence using induction on the size of the kernel, $G_1$ is
$\Gamma$-internal; hence so is $G$. Thus we may assume $G$ has no
proper definable subgroups of finite index. Since the action of $G$ on
$N$ by conjugation has kernel of finite index, $N$ must be central.

Let $Y = \{g^n: g \in B\} $.  By \cite[Theorem\ 7.2]{Edm}, there
exists a definable function $\alpha: Y \to B$ with $ \alpha(b)^n =
b$ for all $b\in B$. Define $\beta: Y \to G$ by $\beta(b)= a^n$ where $f(a)=\alpha(b)$;
this does not depend on the choice of $a$, and we have $f(\beta(b)) =
f(a)^n = \alpha(b)^n = b$. It follows that $f^{-1}(Y) = N \beta(Y)
\subseteq \dcl(N,\Gamma)$ is $\Gamma$-internal.

Similarly, let $[B,B] = \{[g,h]: g,h \in B\}$.  As above, there exists
a definable $\alpha_1: [B,B] \to B$ such that $(\exists y)
[\alpha_1(b),y] = b$, and $\alpha_2: [B,B] \to B$ such that
$[\alpha_1(b),\alpha_2(b)] = b$.  Define $\beta': [B,B] \to G$ by
$\beta'(b) = [a_1,a_2]$ where $f(a_i) = \alpha_i(b)$. Again $\beta'$
is definable and well-defined, and shows that $f^{-1}([B,B])$ is
$\Gamma$-internal.

Hence for any $k$, letting $X^{(k)} = \{x_1...x_k: x_1,...,x_k \in
X\}$, $ (f^{-1}(Y \union [B,B]))^{(k)} = f^{-1}( (Y \union
[B,B])^{(k)} )$ is $\Gamma$-internal.  So we are done once we show:

\begin{claim}
Let $B$ be any definably connected group definable in an $o$-minimal
structure.  Let $Y=Y_n(B) = \{g^n: g \in B\}$.  Then for some $k \in
\Zz_{\geq 0}$, $(Y \union [B,B]) ^{(k)} = B$.
\end{claim}

\begin{proof}
If the Claim holds for a normal subgroup $H$ of $B$ with bound $k'$,
and also for $B/H$ (with bound $k''$), then it is easily seen to hold
for $B$ (with bound $k'+k''$.)
 
We use induction on $\dimo(B)$.  If $B$ has a non-trivial proper
connected definable normal subgroup $H$, then the statement holds for
$H$ and for $B/H$.  We may thus assume $B$ has no such subgroups $H$.
Thus any definable normal subgroup of $B$ is finite and by
connectedness, central.

If $B$ is centerless, it is definably simple. By \cite{PetPilSta}, $B$
is elementarily equivalent to a simple Lie group.  In this case it is
easy to see that every element is the product of a bounded number of
commutators, by considering root subgroups.

Since the center $Z$ of $B$ is finite, then $B/Z$ is definably simple and
so the claim holds (say with $k_{B/Z}$).  Let $W_l = (Y \union [B,B])
^{(l)}$, and let $\pi: B \to B/Z$ be the quotient homomorphism.  Then
$\pi(W_l)= B/Z$ for $l \geq k_{B/Z}$.  For such $l$, \(W_{5l}\cap Z =
W_{l}\cap Z\) implies that \(W_{4l} = W_{2l}\). Thus $W_{2k}$ is a
subgroup of $B$ mapping onto $B/Z$, for some $k \leq k_{B/Z}5^{|Z|}$.
By connectedness $W_{2k}=B$.
\end{proof}

This concludes the proof.
\end{proof}

We conclude this section with a quick discussion of definable
compactness. First, one can define the limit of a function along a
definable type, generalizing the limit along a one-dimensional
curve. Let $X$ be a definable space over an $o$-minimal structure
$\Gamma$.

\begin{defn}\label{D:G-limit}
  Given a definable set $D$, a definable type $p$ on $D$, and a
  definable function $g: D \to X$, we define $\lim_p g = x \in X$
  if for any definable neighborhood $U$ of $x$, $g_\star p$ concentrates on
  $U$.
\end{defn}

When \(X\) is \(\Gamma\cup\{-\infty,+\infty\}\), the limit always exists. Consider the definable
type $g_\star p$. It must be of the form $r_\infty$,
$r_{-\infty}$, $r_a := (x=a)$, $r_a^+$ the type of elements infinitesimally bigger than $a$,
or $r_a^-$. By definition $\lim_p g $ is $\infty, - \infty, a,a,a$ in the respective cases.  

The following definition is equivalent to the one in \cite{PetSte}.

\begin{defn}\label{D:defcomp}
The definable set $X$ is {\em definably compact} if for any definable type $p$ on $\Gamma$ and
 any definable function $f: \Gamma \to X$, $\lim_p f$ exists.
\end{defn}

\begin{defn}
Let $A$ be a definable Abelian group.
\begin{itemize}
\item A definable set is called {\em generic}
if finitely many translates of that set cover the group.

\item Say $A$ has
the property (NG) if the non-generic definable sets form an ideal.
\end{itemize}
\end{defn}

\begin{prop}[{\cite[Corollary\ 3.9]{PetPil-Gen}}]\label{P:defcomp NG}
Any definably compact definable Abelian group $A$ in an $o$-minimal theory has (NG).
\end{prop}

Moreover any definable subsemigroup of \(A\) is a group
(cf. \cite[Theorem\ 5.1]{PetPil-Gen}). We include the deduction of
the latter fact in a case we will need.

\begin{lem}\label{L:semigroup}
Let $A$ be an Abelian group with (NG).  Let $Y$ be a definable semi-group of $A$, such that $Y-Y=A$. Then $Y=A$.
\end{lem}

\begin{proof}
Note first that  $A \sminus Y$ is not generic. Otherwise, some finite intersection $\bigcap_{i=1}^n (c_i+Y) = \emptyset$. 
By assumption, there exist $b_i \in Y$ with $b_i+c_i \in Y$. If $b:=\sum b_i$ then $b + c_i \in Y$; so by translating we may assume each $c_i \in Y$. But then $\sum c_i \in c_i + Y$
for each $i$, a contradiction. 

As $A$ is generic, $Y$ cannot be in the non-generic ideal, so $Y$ is
generic and $\bigcup_{i=1}^n d_i + Y = G$ for some $d_i \in G$. Again
we find $e \in Y$ with $e+d_i \in Y$.  We have $\union _{i=1}^n
(e+d_i) + Y = G$. But $e+d_i + Y \subseteq Y$.  So $Y = G$.
\end{proof}

\subsection{Valued fields: imaginaries and resolution}

Let $K$ be an algebraically closed valued field, with valuation ring
$\Oo$, maximal ideal \(\fM\) and value group \(\Gamma\).  The
geometric language for valued fields has a sort for the valued field
itself, and certain other sorts.  In particular, there is a sort $S_n$
such that $S_n(K)$ is the space of free $\Oo$-modules in $K^n$,
equivalently $S_n(K) = \GL_n(K)/\GL_n(\Oo)$. Until the end of Section\ 2.5, we work
in a model \(M\) of the theory \(\ACVF\) of algebraically closed
fields in the geometric language.

By a {\em substructure}, we mean a subset of $M$, closed under
definable functions. For any geometric sort $S$ and substructure
$A\subseteq M$, $S(A)$ denotes $S\cap A$; in particular $K(A)$ is the
set of field points of $A$. A substructure $A\subseteq M$ is called
{\em resolved} if $A\subseteq\dcl(K(A))$.  When $\Gamma(A) \neq 0$ and
$A$ is algebraically closed this just amounts to saying that $A \prec
M$. Recall that, in any first order theory, if $A$ is a substructure
of a model $M$, $M$ is {\em prime} over $A$ if any elementary map $A
\to N$ into another model, extends to an elementary map $M \to N$; and
$M$ is {\em minimal} over $A$ if there is no $M' \prec M$ with $A
\subset M'$. If a minimal model over $A$ and a prime one exist, then
any two minimal or prime models over $A$ are isomorphic.

\begin{prop}\label{P:resolution}
  Let $A$ be a substructure of $M$, finitely generated over a subfield
  $L$ of $K$, and assume $\Gamma(A)\neq 0$. Then there exists a
  minimal prime model $\tA$ over $A$ which enjoys the following
  properties.
  \begin{enumerate}
  \item $\tA$ is a minimal resolution of $A$. Moreover it is the
    unique minimal resolution, up to isomorphism over $A$.  It is
    atomic over $A$.
  \item $\St_L(\tA) = \St_L(A)$.
  \item Let $A \leq A'$, with $A'$ finitely generated over $A$.  Then
    $\tA $ embeds into $\tA'$ over $A$. If $A \leq A' \leq \tA$, then
    $\tA$ is the prime resolution of $A'$.
  \item For a valued field extension $L'$ of $L$, let $L'(A)$ be the
    structure generated by $L' \union A$.  Then $\alg{L'(\tA)}$ is a
    prime resolution of $L'(A)$.
  \item If $\tp(A/L)$ is stably dominated, then $\tp(\tA / L)$ is
    stably dominated.
  \item If $\tp(A'/A)$ is stably dominated, and $A'\ind_{A} \tA$, then
    there exists a prime resolution $\tA'$ of $A'$ such that
    $\tp(\tA'/\tA)$ is stably dominated.
  \end{enumerate}
\end{prop}

\begin{proof}
  The existence, uniqueness and minimality of $\tA$ are shown in
  \cite[Theorem 11.14]{HasHruMac-Book}. It is also shown there that
  $k(\tA)=k(\acl(A))$ and $\Gamma(\tA)=\Gamma(\acl(A))$, where $k$ is
  the residue field; and that $\tA/A$ is atomic, i.e. $\tp(c/A)$ is
  isolated for any tuple $c$ from $A$.

  \begin{enumerate}\setcounter{enumi}{1}
  \item Since $L$ is a field, for any $B =\acl(B) $ with $L \subset
    B$, $\St_L(B) = \dcl(B,k(B))$.
  \item This is immediate from the definition of prime resolution:
    since $\tA'$ is a resolution of $A$, $\tA$ embeds into $\tA'$. If
    $A \leq A' \leq \tA$, then $\tA$ is clearly a minimal resolution
    of $A'$; hence by (1) it is the prime resolution.
  \item Let $B$ be the prime resolution of $L'(A)$.  Then $\tA$ embeds
    into $B$.  Within $B$, $\alg{L'(\tA)}$ is a resolution of $L'(A)$;
    by minimality of $B$, $B = \alg{L'(\tA)}$.
  \item We may assume $L = \alg{L}$. Let $p$ be an $L$-invariant extension of $\tp(\tA/L)$. Let $\overline{L}$ be a maximal immediate extension of $L$.  Choose it in such a way that $\tA\models p|\overline{L}$. Since $A/L$ is stably dominated, $\Gamma(
    \alg{\overline{L}(A)}) = \Gamma(\overline{L})$. According to (4),
    $\alg{\overline{L}(\tA)}$ is the prime resolution of
    $\alg{\overline{L}(A)}$, and so $\Gamma(\alg{\overline{L}(\tA)}) =
    \Gamma(\overline{L})$. Since $\overline{L}$ is a metastability basis, $\tp(\tA / \overline{L})$ is stably
    dominated. Using descent (cf. Proposition\ \ref{P:base
      change}.(\ref{descent})), $\tp(\tA/L)$ is stably dominated.

  \item Since $A'\ind_{A}\tA$ and $\tp(A'/A)$ is stably dominated,
    $\tp(A'/\tA)$ is stably dominated. By (5), $\tp(B/\tA)$ is stably
    dominated, where $B$ is a resolution of $\tA(A')$. But $B$
    contains a resolution of $A'$ whose type over $\tA$ is therefore
    stably dominated.\qedhere
  \end{enumerate}
\end{proof}

\begin{prop}\label{P:ACVF hyp}
  The theory $\ACVF$ is metastable, with \AFD and \Aom.
\end{prop}

\begin{proof}
  The fact that, in $\ACVF$, maximally complete submodels are
  metastability bases is proved in
  \cite[Theorem\ 12.18.(ii)]{HasHruMac-Book} and the existence of
  invariant extensions is shown in
  \cite[Corollary\ 8.16]{HasHruMac-Book}. The fact that \AFD holds
  follows from the fact that the value group $\Gamma$ is a divisible
  ordered Abelian group, that the stable part, being internal to an
  algebraically closed field, has finite Morley rank and that, for any
  definable subset $D$ of the field, $\dimst(D)$ and $\dimo(D)$ are
  bounded by the dimension of the Zariski closure of $D$.

  To see that \Aom holds, let $M \models \ACVF$ and $C$ be finitely generated
  over $M$---$C$ might contain imaginaries. By Proposition\ \ref{P:resolution},
  the resolution $N$ of $C$ is a model of $\ACVF$ which is atomic over $M$.
  Hence isolated types over $C$ are dense. Also, note that a maximally complete
  model of \(\ACVF\) is \(\aleph_1\)-saturated if and only if its value group
  and residue fields are.
\end{proof}

\begin{defn}\label{D:pure im}
\begin{enumerate}
\item The (imaginary) element \(e\in M\) is said to be \emph{purely
  imaginary} over \(C\subseteq M\) if \(\acl(Ce)\) contains no field
  elements other than those in $\acl(C)$.
\item An pro-definable set \(D\) is \emph{purely imaginary} if
  there exists no pro-\(C\)-definable map (with parameters) from $D$ onto an
  infinite subset of the field $K$.
\end{enumerate}
\end{defn}

Note that \(D\) is purely imaginary if and only if, for all \(C\) such
that \(D\) is \(C\)-definable, any \(e\in D\) is purely imaginary over
\(C\). For every $\alpha \in \Gamma$, let $\alpha\Oo := \{x \in K: \val(x) \geq
\alpha\}$, and $\alpha\fM : = \{x \in K: \val(x) > \alpha\}$.

\begin{lem}\label{L:se1}The following are equivalent:
\begin{enumerate}
\item $e$ is purely imaginary over $C$.
\item $\dcl(Ce) \meet K \subseteq \acl(C)$.
\item For some $\beta_0 \leq 0 \leq \beta_1 \in \Gamma_C(e)$ and $d\in (\beta_0 \Oo/\beta_1 \fM)^n$, $e\in\dcl(\acl(C)\beta_0\beta_1d)$. 
\end{enumerate}
\end{lem}

\begin{proof}
Note that (1) implies (2) trivially. Let us now prove that (2) implies
(1). Let $d \in \acl(Ce)$ be a field element.  The finite set of
conjugates of $d$ over $Ce$ is coded by a tuple $d'$ of field elements
and $d' \in \dcl(Ce)$.  By (2), $d' \in \acl(C)$.  Since $d \in
\acl(d')$ we have $d \in \acl(C)$.

Let us now assume (3) and prove (2). Let $a = g(e) \in K$ for some $C$-definable map $g$. By (3), there exists an $\acl(C\beta_0\beta_1)$-definable map $f$ with domain $(\beta_0 \Oo/\beta_1 \fM)^n$ and whose range contains $e$. Then $g\circ f : (\beta_0 \Oo/\beta_1 \fM)^n\to K$  is definable. However, its image cannot contain a ball since in some models of $\ACVF$, $\beta_0\Oo/\beta_1 \fM$ is
countable while every ball is uncountable. Hence, the image of $g\circ f$ is finite and $a\in\acl(C)$.

Finally, let us prove that (2) implies (3): By
\cite[Theorem\ 1.0.2]{HasHruMac-ACVF}, and using (2), there exists an
$e$-definable $\Oo$-submodule $\Lambda$ of $K^m$ (for some $m$), such
that $e$ is a canonical parameter for $\Lambda$, over $\acl(C)$. The
$K$-vector space $V=K \tensor_{\Oo} \Lambda$ is coded by an element
$w$ of some Grassmanian $G_{m,l}$. Note that $w\in\dcl(Ce)\meet K = \acl(C)\meet K =: K_0$. Since $K_0$ is an algebraically closed field, $V$ is
$K_0$-isomorphic to $K^l$, so we may assume $V=K^l$. Dually, let $V' =
\{v \in V: K\cdot v \subseteq \Lambda \}$.  Then $V'$ is a $K_0$-definable $K$-vector subspace of $V$. Replacing $\Lambda$ by
the image in $V/V'$, we may assume $V'=(0)$.  It follows, by
\cite[Lemma\ 2.2.4]{HasHruMac-ACVF}, that $V' \subseteq \beta_0 \Oo^l$ for some
$\beta_0 <0$. The set of all such $\beta_0$ is $\acl(C)e$-definable. Since o-minimal groups have Skolem functions, we may choose $\beta_0$ in \(\Gamma_C(e)\).

Let $v_1,\ldots,v_l$ be a standard basis for $V=K^l$.  Since $v_i \in
K \tensor_{\Oo} \Lambda$, we have $c_i v_i \in \Lambda$ for some $c_i
\in \Oo$.  So $\sum r_i c_i v_i \in \Lambda$ for all $r_1,\ldots,r_l
\in \Oo$.  Let $\Lambda'(\beta)= \{\sum r_i v_i: \val(r_i) > \beta
\}$.  Then $\Lambda'(\beta) \subseteq \Lambda$ for sufficiently large
$\beta$; namely for $\beta \geq \max_i\val(c_i)$. The set $\{\beta:
\Lambda'(\beta) \subseteq \Lambda \}$ is definable, hence it contains
$\{\beta: \beta > \beta_1 \}$ for some $Ce$-definable $\beta_1$. It
follows that $\Lambda$ is determined by its image \(D\) in $(\beta_0
\Oo / \beta_1 \fM)^l$.  Pick $c \in \beta_0 \Oo / \beta_1 \fM$; then $r
\mapsto r\cdot c$ is an isomorphism $\Oo / (\beta_1-\beta_0)\fM \to b_0 \Oo
/ b_1\fM$.  By \cite{HruTat}, $\Oo / (\beta_1-\beta_0)\fM $ is stably
embedded; hence so is $\beta_0 \Oo / \beta_1 \fM$. Let $d\in (\beta_0 \Oo / \beta_1 \fM)^n$ be such that \(D\) is $\beta_0\beta_1d$-definable. Then $e\in\dcl(\acl(C)\beta_0\beta_1d)$, as required.
\end{proof}

\begin{rem}\label{R:Geom pure im}
As we can see from the proof above, every geometric sort other than
$K$ is purely imaginary.
\end{rem}

\begin{defn}
  \begin{enumerate}
  \item An element \(e\in M\) is said to be \emph{boundedly imaginary}
    over \(C\subseteq M\) if for any $\gamma\in\Gamma_C(e)$,
    $\tp(\gamma/C)$ is bounded, i.e. it is neither the type at
    \(+\infty\) nor the one at \(-\infty\).
  \item An \(\infty\)-definable set \(D\) is \emph{boundedly
      imaginary} if there exists no definable map (with parameters)
    from $D$ onto an unbounded subset of $\Gamma$.
  \end{enumerate}
\end{defn}

Note that \(D\) is boundedly imaginary if and only if any $e\in D$ is boundedly imaginary over any $C$ over which $D$ is defined.

\begin{lem}\label{L:se2}
Any definable function $f: (\alpha \Oo / \beta \fM)^n \to \Gamma$ is bounded.
\end{lem}

\begin{proof}
Suppose first $n=1$.  Since parameters are allowed, we may assume
$\alpha=0$, and consider $f: \Oo / \beta \fM \to \Gamma$, defined over
$C$. Let $q$ be the type of elements of $\Gamma$ greater than any
element of $\Gamma(C)$.  For $\gamma \models q$, let $X(\gamma) $ be
the pullback to $\Oo$ of $f^{-1}(\gamma) $.  This is a finite Boolean
combination of balls of valuative radii
$\delta_1(\gamma),\ldots,\delta_m(\gamma)$, with $0 \leq
\delta_i(\gamma) \leq \beta$.  But any $C$-definable function into a
bounded interval in $\Gamma$ is constant on $q$.  Thus $X(\gamma)$ is
a finite Boolean combination of balls of constant valuative radii
$\delta_1,\ldots,\delta_m$.  However, it is shown in
\cite[Proposition\ 2.4.4]{HasHruMac-ACVF} that any definable function
on a finite cover of $\Gamma$ into balls of constant radius has finite
image.  Hence $X(\gamma)$ is constant on $q$.  But if $X(\gamma) =
X(\gamma') \neq \emptyset$ then for any $x \in X(\gamma)$, $\gamma =
f(x + \beta \fM )= \gamma'$.  It follows that $X(\gamma) = \emptyset$
for $\gamma \models q$, i.e. $f$ is bounded.  Now, given $f: (\alpha
\Oo / \beta \fM)^2 \to \Gamma$, let $F(x) = \sup \{f(x,y): y \in
\alpha \Oo / \beta \fM \}$.  Then $F$ is bounded, so $f$ is bounded.
This shows the case $n=2$, and the general case is similar.
\end{proof}

\begin{cor}\label{C:bdd-im}
  Let $D$ be a $C$-definable set in $\ACVF$. Then the following are
  equivalent:
  \begin{enumerate}
  \item \(D\) is boundedly imaginary.
  \item There exists a definable surjective map $g: (\Oo / \beta
    \fM)^n \to D$.
  \item There is an $\acl(C)$-definable surjective map $g: (\beta_0\Oo / \beta_1\fM)^n \to D$, where $\beta_0 \leq 0 \leq \beta_1 \in \Gamma(C)$.
  \end{enumerate}
\end{cor}

Note that infinite definable subsets of \(K\) contain balls and that balls are not boundedly imaginary: given any point $a$ in a ball $b$, the set $\{\val(x-a)\mid x\in b\}$ is not bounded. It follows that boundedly imaginary definable
sets are purely imaginary.

\begin{proof}
  By Lemma\ \ref{L:se2}, (1) follows from (2) and (3) implies (2)
  easily. To prove that (1) implies (3), by compactness it suffices to
  fix $e \in D$ and find $\beta_0 \leq 0 \leq \beta_1 \in \Gamma(C)$ and an $\acl(C)$-definable map $g: (\beta_0\Oo / \beta_1\fM)^n \to D$ whose range contains $e$. Since
  $e$ is purely imaginary over $C$, by Lemma\ \ref{L:se1}.(3), this does hold with \(\beta_0, \beta_1\in\Gamma_C(e)\). But $\beta_i=f_i(e)$
  for some $C$-definable functions $f_i$; by (1), $f_i$ is bounded on
  $D$, and the upper and lower bounds are $C$-definable.
\end{proof}

It follows, by compactness, that any \(\infty\)-definable boundedly
imaginary \(D\) is contained in a definable boundedly imaginary
\(D'\).

Let us conclude this section with a characterization of independence of stably dominated types of field elements in $\ACVF$ in terms of a \emph{maximum modulus} principle.

\begin{prop}[Maximum Modulus, {\cite[Theorem\ 14.12]{HasHruMac-Book}}]
\label{P:maxmod}
Let $p$ be a stably dominated $C$-definable type concentrating on an
affine variety $V$ defined over $K \meet C$. Let $P=p|C$ and $F$ be a
regular function on $V$ over $L \supseteq C$.  Then $\val(F)$ has an
infimum $\gamma_{\min}^F \in \Gamma(L)$ on $P$.  Moreover for $a
\models P$, $a \models p|L$ if and only if $\val(F(a)) =
\gamma_{\min}^F$ for all such $F$.
\end{prop}

\begin{cor}[{\cite[Theorem\ 14.13]{HasHruMac-Book}}]\label{C:maxmod2}
  Let $U$ and $V$ be varieties over the algebraically closed valued
  field $C$.  Let $p$ and $q$ be stably dominated types over $C$ of
  elements of $U$ and $V$ respectively.  Let $F$ be a regular function
  on $U \times V$.  Then there exists $\gamma_F \in \Gamma$
  such that:
  \begin{enumerate}
  \item If $(a,b) \models p \tensor q$ then $\val(F(a,b)) = \gamma_F$.
  \item For any $a \models p$, $b \models q$, we have $\val(F(a,b))
    \geq \gamma_F$.
  \item Assume $U$ and $V$ are affine.  If $a \models p$ and $b
    \models q$ and $\val(F(a,b)) = \gamma_F$ for all regular $F$ on $U
    \times V$, then $(a,b) \models p \tensor q$.
  \end{enumerate}
\end{cor}

\begin{proof}
  Since $V$ admits a finite cover by open affines, and $q$
  concentrates on one of these affines, we may assume $V$ is affine.
  Let $a \models p$.  Then the statement follows from
  Proposition\ \ref{P:maxmod} applied to $q$ over $Ca$. As $p$ is stably
  dominated, the value of $\gamma_F$ does not depend on $a$.
\end{proof}

\subsection{Definable directed partial orders and cofinal types}

In this subsection, we show that nicely behaved definable filters can be
completed to a definable type (see \cref{complete filter}). This in turn implies
that definable directed orders admit definable cofinal types.

\begin{defn}
A {\em definable  filter basis} on a definable set $X$ is a definable family
$\cB$ of definable subsets of $X$, forming a filter basis; i.e. if $U,V \in \cB$
then there exists $W \in \cB$ with $W \subseteq U \meet V$.   We also assume
$\emptyset \notin \cB$.
\end{defn}

This is a strengthening of the earlier notion of definable filter (see
Section~\ref{S:def filter}): if \(\cB\) is a definable filter basis, then the
filter generated by \(\cB\) is definable. The converse does not hold in general.

We will be needing two operations on filters bases:
\begin{itemize}
\item Let $f: X \to Y$ be definable, and let $\cB$ be a definable filter basis
on $X$. Then the pushfoward \(\push{f}{\cB}  = \{f(U):  U \in \cB \}\) is a
definable filter basis.
\item Let \(\cB\) be a filter basis on \(X\times Y\) for some (pro)-definable
set \(X\) and \(Y\) (over some model \(M\)) and let \(a\in X\). Assume that, for
every \(U\in \cB\), \(U_a = U \cap \{a\}\times Y \neq \emptyset\) --- we say
that \(\tp(a/M)\) is consistent with (the filter generated by) \(\cB\). Then
\(\cB_a = \{U_a \mid U\in\cB\}\) is a filter basis.
\end{itemize}

Let \(M\) be an \(\aleph_1\)-saturated model and let \(\Gamma\) be a stably
embedded \(o\)-minimal \(M\)-definable set. Further assume that, for any finite
tuple \(a\), \(\gamma = \Gamma(Ma)\) is countably \(\dcl\)-generated over \(M\)
and \(\tp(a/M\gamma)\) is definable: for any formula \(\phi(x,y,z)\), the set
\(\{c\in M \mid \phi(x,c,\gamma)\in \tp(a/M\Gamma(Ma))\}\) is definable.

Assuming \Aom, any countable set \(C\) is contained in such an \(M\).

\begin{lem}
\label{ind complete filter}
Let \(X\) be a countably pro-\(M\)-definable set. Let \(\pi\) be a filter on
\(X\times \Gamma\) which is generated by countably many \(M\)-definable filter
bases. Let \(a\in X\) be such that \(\tp(a/M)\) is definable and consistent with
\(\pi\). Then there exists an \(M\)-definable type \(q \supseteq p\)
concentrating on \(X\times \Gamma\) which is consistent with \(\pi\).
\end{lem}

\begin{proof}
If there exists a \(\gamma\in \dcl(Ma)\) such that \(a\gamma\models \pi\), we
can choose \(q = \tp(a\gamma/M)\). Now, assume that there is not such
\(\gamma\). For every family of definable functions \(g_m : X \to \Gamma\), we
define the following equivalence relation : for all \(m, n\in M\), we say that
\(m \sim_g n\) if there exists \(U \in \pi\) such that \([g_m(a),g_n(a)] \cap
U_a = \emptyset\).

\begin{claim}
The relation \(\sim_g\) has finitely many classes which are all definable.
\end{claim}

\begin{proof}
Let \(U = (U_b)_b\) be a definable filter basis contained in \(\pi\). We say
that \(m \sim_{g,U} n\) if there exists a \(b \in M\) with \([g_{m}(a),g_{n}(a)]
\cap U_{b,a} = \emptyset\). This is an equivalence relation on the set of \(m\)
such that there exists a \(b\) with \(g_m(a) \nin U_{b,a}\) --- whose complement
we can see as another class. Since \(\Gamma\) is \(o\)-minimal, each \(U_{b,a}\)
is union of at most \(n\) intervals. Then, since the \(U_{b,a}\) form a filter
basis, there are at most \(n+2\) classes. Moreover, the relation \(\sim_g\) is
the intersection of all the \(\sim_{g,U}\). Choosing (countably many)
representatives for all the classes of the \(\sim_{g,U}\), we see that all the
classes of \(\sim_g\) are represented in the \(\aleph_1\)-saturated model \(M\).

Now, for any \(m \in M\), the class of \(m\) is the union, as \(U\) ranges over
definable filter bases contained in \(\pi\), of the sets \(\{m' \mid \exists b\
[g_m(a),g_m'(a)] \cap U_{b,a} = \emptyset\}\) which are definable. So all the
classes are (countably) ind-\(M\)-definable. Hence, by compactness and
\(\aleph_1\)-saturation of \(M\), there are finitely many classes and they are
all definable.
\end{proof}

Let \(E_g\) denote the \(\sim_g\)-class of tuples \(m\) such that there exists a
\(U\in\pi\) such that \((-\infty,g_m(a)] \cap U_{a} = \emptyset\) --- or \(E_g =
\emptyset\) if no such \(m\) exists. Let \(q\supseteq p\) be the type such that
\(q(x,y)\vdash g_m(x) < y\) if and only of \(m\in E_g\). By construction, it is
definable and consistent with \(\pi\).
\end{proof}

\begin{prop}
\label{complete filter}
Let \(\cB\) be an \(M\)-definable filter basis on some \(M\)-definable set
\(X\). Then \(\cB\) is consistent with an \(M\)-definable type.
\end{prop}

\begin{proof}
Let \(\fF\) consist of all $M$-definable functions $X \to \Gamma$, seen as a
pro-definable function. Using \cref{ind complete filter} iteratively, we find an
\(M\)-definable type $\tp(\gamma/M)$ of tuples from $\Gamma$ which is consistent
with $\push{\fF}{\pi}$.

Note that, at stage \(\alpha\), we have a definable type $\tp(\gamma/M)$ of
tuples from $\Gamma$ consistent with $\push{\fF}{\pi}$. So there exists some
\(a\models\restr{\pi}{M}\) with \(\gamma\in\Gamma(\acl(Ma))\) (via the first
$\alpha$ functions in $\fF$). So \(\gamma\) is countably \(\dcl\)-generated over
\(M\), we may assume that \(\gamma\) is a countable tuple to apply \cref{ind
complete filter}. In the end, we have \(a\models\restr{\pi}{M}\) with \(\fF(a) =
\gamma\). Then $\tp(a/M\gamma)$ is definable and, by transitivity, \(\tp(a/M)\)
is definable.
\end{proof}

\begin{cor}
\label{cofinal-def}
Let \(\leq\) be an \(M\)-definable directed partial order on an \(M\)-definable
set \(P\). Then there exists an \(M\)-definable type \(p\) cofinal in \(P\) :
for any $c \in P$, we have $\models \defsc{p}{x} x \geq c$.
\end{cor}

\begin{proof}
Consider the \(M\)-definable filter basis of all cones \(\{x \mid b \leq x\}\).
\end{proof}

\begin{discussion}
If there exists a definable weakly order preserving
map $j: \Gamma \to P$ with cofinal image, then we can use the
definable type at $\infty$ of $\Gamma$, $r_\infty$, to obtain a
cofinal definable type of $P$, namely $j_\star r_\infty$.

In general, it is not always possible to find a
one-dimensional cofinal subset of $P$.  For instance, when $\Gamma$ is
a divisible ordered Abelian group, consider the product of two closed
intervals of incommensurable sizes; or the subdiagonal part of a
square.
\end{discussion}

\section{Groups with definable generics}\label{S:DefGen}

A pro-definable group is a pro-definable set with a pro-definable
group law. We will, in a few cases, need to consider a larger class of
groups. A \emph{piecewise pro-definable group} is a piecewise
pro-definable set $G = \projlim_i G_i$ together with pro-definable maps
$m_i: G_i \times G_i \to G_{j}$, for some $j\geq i$, compatible with
the inclusions, and inducing a group structure $m: G \times G \to
G$. By a pro-definable subgroup $H$ of $G$ we mean a pro-definable
subset $H \subset G_i$ for some definable piece $G_i$ of $G$, such
that $m(H^2) \subset H$ and $(H,m)$ is a subgroup of $(G,m)$.

Let us now fix \(G\) a pro-definable group.

\subsection{Definable generics}

For all filters \(\pi\) concentrating on \(G\) and \(g\in
G\), let \(\transl{g}{\pi}:= \{\phi(x,a)\mid \pi(x)\vdash \phi(g\cdot
x,a)\}\). Similarly, we define \(\rtransl{\pi}{g}:= \{\phi(x,a)\mid
\pi(x)\vdash \phi(x\cdot g,a)\}\). Note that if \(\pi\) is \(C\)-definable, then
\(\transl{g}{\pi}\) and \(\rtransl{\pi}{g}\) are \(Cg\)-definable.

\begin{defn}\label{D:gentype}
Let \(\pi\) be a filter concentrating on \(G\). We say that \(\pi\) is
\emph{right generic}\footnote{When $\pi$ is a complete type, this notion is usually referred to as a definable f-generic in the literature.} in \(G\) over $C$ if, for all \(g\in G\)
\(\transl{g}{\pi}\) is \(C\)-definable. We say that it is \emph{left generic}
over $C$ if, for all \(g\in G\), \(\rtransl{\pi}{g}\) is
\(C\)-definable.
\end{defn}

We say that \(G\) admits a generic filter if there
exists a filter \(\pi\) concentrating on \(G\) which is left or
right generic in \(G\) (over some small set of parameters).

\begin{lem}
Let \(\pi\) be a definable filter concentrating on \(G\) then
\(\pi\) is right generic if and only if \(\pi\) has boundedly many
left translations by \(G\). Moreover if \(\pi\) is right generic then for
every formula \(\phi\), the set \(\{\restr{(\transl{g}{\pi})}{\phi}\mid
g\in G\}\) is finite.
\end{lem}

\begin{proof}
If \(\pi\) is right generic, then its orbit must be bounded since
there are only boundedly many \(C\)-definable filters. Conversely, if
the orbit of \(\pi\) is bounded we can find a small \(C\)
such that every translate is \(C\)-definable.

Let us now assume that \(\pi\) is right generic and \(\phi\) be
some formula. Then the equivalence relation \(g_1\sim g_2\) if
\(\restr{(\transl{g_1}{\pi})}{\phi} = \restr{(\transl{g_2}{\pi})}{\phi}\)
is definable by \((\forall y)\defsc{\pi}{x}\phi(g_1\cdot x,y)\fequiv
\defsc{\pi}{x}\phi(g_2\cdot x,y)\). Since this equivalence relation has
boundedly many classes, it must have finitely many.
\end{proof}

\begin{rem}\label{R:inv filter}
Let \(\pi\) be a right generic filter of $G$. Then it follows easily
from the above that $\mu := \bigcap_{g\in G}\transl{g}{\pi}$ is a
definable filter concentrating on $G$. It is obviously invariant under
left translation by $G$.
\end{rem}

\begin{prop}\label{P:prodefgp}
Assume \(G\) is pro-\(C\)-definable and admits a (right) generic filter over
\(C\). Then \(G\) is pro-\(C\)-definably isomorphic to a pro-limit of
\(C\)-definable groups. In particular, if \(G\) is
\(\infty\)-\(C\)-definable, \(G\) is \(C\)-definably isomorphic to
\(\bigcap_i G_i\) where the \(G_i\) are all \(C\)-definable subgroups
of some \(C\)-definable group \(H\).
\end{prop}

\begin{proof}
By remark\ \ref{R:inv filter}, there is a \(C\)-definable
$G$-invariant filter \(\pi\) concentrating on $G$. Let us first assume
that $G = \bigcap_i X_i$ is $\infty$-definable. By compactness, there
exists $i_{0}\in I$ such that if $x$, $y$ and $z\in X_{i_{0}}$ then
$(x\cdot y)\cdot z = x\cdot (y\cdot z)$ and $x \cdot 1 = 1 \cdot x =
x$. Because \(I\) is filtered, we may assume that $i_0$ is the smallest
element of $I$. Let $Y_{i} = \{a\in X_{i_{0}}\mid \pi(x)\vdash a \cdot
x\in X_i\}$, which is \(C\)-definable.

\begin{claim}
$G = \bigcap_{i\in I} Y_{i}$.
\end{claim}

\begin{proof}
Let $a\in G$ and \(c\models \pi\). Then \(a \cdot c\in G\subseteq X_i\) and
\(a\in Y_i\). Conversely, let $a\in \bigcap_{i} Y_{i}$. Then
\(\pi(x)\vdash a\cdot x\in \bigcap_i{X_i} = G\). Let \(c\models \pi\), we have
\(b := a \cdot c \in G\) and hence \(a = b c^{-1}\in G\).
\end{proof}

Let $i_1$ be such that $Y_{i_1}\cdot Y_{i_1}\subseteq X_{i_0}$ and
$S_i := \{a\in Y_i\mid Y_i\cdot a\subseteq Y_i\}$, which is
\(C\)-definable.

\begin{claim}\label{C:monoid}
For all $i\geq i_1$, $(S_{i},\cdot)$ is a monoid containing \(G\).
\end{claim}

\begin{proof}
Pick any $a$ and $c\in S_{i}$. For all $x\in Y_{i}$, $x \cdot a\in
Y_{i}$ and $x\cdot c \in Y_{i}$, hence $x\cdot a\cdot c \in Y_{i}$
and, since \(a\in Y_i\), \(a\cdot c\in Y_i\), so $a\cdot c\in S_i$. Since we obviously have
\(1\in S_i\), \(S_i\) is a monoid. Now, let $a\in G\subseteq Y_{i_1}$
and $c\in Y_{i}\subseteq Y_{i_1}$. We have \(c\cdot a\in
X_{i_0}\). Moreover, \(\pi(x)\vdash c\cdot x\in X_i\) and, therefore, by $G$-invariance of $\pi$,
\(\transl{a}{\pi} = \pi \vdash c\cdot a\cdot x \in X_i\). It follows that
\(c\cdot a\in Y_i\) and \(a\in S_i\).
\end{proof}

Then $G = \bigcap_{i\geq i_1} H_i$, where $H_i$ is the group of
invertible elements in $S_i$.

Let us now assume that \(G = \projlim_{i\in I} X_i\). For all $i\in
I$, let $G_i = \{a\in G\mid \pi^{\tensor 2}(x,y)\vdash
\rho_i(x^{-1}\cdot a\cdot y) = \rho_i(x^{-1}\cdot y)=
\rho_i(x^{-1}\cdot a^{-1}\cdot y)\}$, where \(\rho_i : G \to X_i\) is
the projection.

\begin{claim}
\(G_i \normal G\).
\end{claim}

\begin{proof}
By definition, if \(a\in G_i\), then \(a^{-1}\in G_i\). Let now \(a\),
\(c\in G_i\). For all \((x,y)\models \pi^{\tensor 2} |C a c\),
\(c\cdot y\models \transl{c}{\pi} |C a c x = \pi|C a c x\). It
follows that \((x^{-1}\cdot a\cdot (c\cdot y))_i = (x^{-1}\cdot
(c\cdot y))_i = (x\cdot y)_i\). Similarly \((x^{-1}\cdot (a\cdot
c)^{-1}\cdot y)_i = ((c\cdot x)^{-1}\cdot a^{-1} \cdot y)_i = (x\cdot
y)_i\). Finally, if \(a\in G_i\), \(c\in G\) and \((x,y)\models
\pi^{\tensor 2} |C a c\), then \((c\cdot x,c\cdot y)\models
(\transl{c}{\pi})^{\tensor 2} | C a c = \pi^{\tensor 2} | C a c\). It
  follows that \((x^{-1}\cdot c^{-1}\cdot a\cdot c\cdot y)_i =
  ((c\cdot x)^{-1}\cdot a\cdot (c\cdot y))_i = ((c\cdot x)^{-1}\cdot
  (c\cdot y))_i = (x\cdot y)_i\).
\end{proof}

Since \(G_i\) is relatively \(C\)-definable, \(H_i := G/G_i\) is
an \(\infty\)-\(C\)-definable group.

\begin{claim}
The natural map \(G \to H := \projlim H_{i}\) is a group isomorphism.
\end{claim}

\begin{proof}
Let \(a\in \bigcap_i G_i\) and \((x,y)\models \pi^{\tensor 2}| C
a\). Then for all \(i\in I\), \((x^{-1}\cdot a \cdot y)_i =
(x^{-1}\cdot y)_i\) and thus \(x^{-1}\cdot a \cdot y = x^{-1}\cdot
y\). It follows that \(a = 1\). Surjectivity follows, by compactness,
from the fact that each map \(G\to H_i\) is surjective.
\end{proof}

The image of \(\pi\) in \(H_i\) is generic, so we can conclude by the
\(\infty\)-definable case.
\end{proof}

Similar results hold for rings:

\begin{prop}
Let \(R\) be an pro-\(C\)-definable ring. Assume that \((R,+)\) admits a generic filter over \(C\) and that there exists a \(C\)-definable filter \(\pi\) concentrating on \(R^\star\) such that the stabilizer of \(\pi\) under (left) multiplication generates \(R\). Then \(R\) is pro-\(C\)-definably isomorphic to a pro-limit of
\(C\)-definable rings. In particular, if \(R\) is
\(\infty\)-\(C\)-definable, \(R\) is \(C\)-definably isomorphic to an intersection of \(C\)-definable subrings of some \(C\)-definable ring.
\end{prop}

The above hypothesis on \(R\) hold in particular if both the additive group and the units of \(R\) admit a generic filter over \(C\) and that the units generate \(R\).

\begin{proof}
Let us first assume \(R = \bigcap_{i\in I}P_i\) is \(\infty\)-definable. By Proposition\ \ref{P:prodefgp}, we may assume that the \(P_i\) are subgroups of some group \((P,+)\) and by compactness we may assume that multiplication is an associative bilinear map on \(P\) and that \(1\) is the identity. For all \(i\in I\), let $Q_{i} = \{a\in P \mid \pi(x)\vdash a\cdot x\in P_i\}$ which is $C$-definable. It is easily checked that \(Q_i\) is a subgroup of \(P\). Moreover, the multiplicative stabilizer of \(\pi\) is a subset of \(Q_i\) which therefore contains \(R\). Conversely, if \(a\in\bigcap_i Q_i\), then for any \(c\models\pi|C a\), \(a\cdot c\in \bigcap_i P_i = R\) and hence \(a\in Rc^{-1}= R\). By compactness, there exists $i_{0}$ such that \(Q_{i_0}\cdot Q_{i_0}\subseteq P\). Let \(R_i = \{a\in Q_{i}\mid Q_{i}\cdot a\subseteq Q_{i}\}\). One can check that, for all \(i\geq i_0\), \(R_i\) is a ring, and proceeding as in Claim\ \ref{C:monoid}, we can show that if \(a\) stabilizes \(\pi\) multiplicatively, and \(c\in Q_i\), then \(c\cdot a \in Q_i\). So \(R_i\) contains \(R\) and \(R = \bigcap_{i\geq i_0}R_i\).

Now, let \(R = \projlim_{i\in I} P_i\). By, Proposition\ \ref{P:prodefgp}, we may assume that each \(P_i\) is an additive group. For all \(i\in I\), let \(J_i := \{a\in R\mid \pi^{\tensor 2} (x,y)\vdash (x^{-1}\cdot a\cdot y)_{i} = 0_i\}\), which is relatively \(C\)-definable. It is clear that \(J_i\) is an additive subgroup. Now pick any \(a\in J_i\) and let \(S_a := \{c\in R\mid c\cdot a\in J_i\}\). Then \(S_a\) is an additive subgroup. Moreover, if \(c\) stabilizes \(\pi\), then so does \(c^{-1}\) and hence for any \((x,y)\models \pi^{\tensor 2}\), \((c^{-1}\cdot x,y)\models \pi^{\tensor 2}\). Since \(a\in J_i\), it follows that \((x^{-1}\cdot c\cdot a \cdot y)_i = ((c^{-1}\cdot x)^{-1}\cdot a\cdot y)_i = 0_i\), so \(c\in S_a\) and hence \(S_a = R\). Similarly, we show that \(J_i\) is a two sided ideal. Bijectivity of the ring homomorphism \(R\to \projlim_i R/J_i\) follows as in the group case and, since \(R/J_i\) is an \(\infty\)-\(C\)-definable ring, we conclude by the \(\infty\)-definable case.
\end{proof}

\subsection{Stabilizers}\label{S:Stab}

The stabilizer of a definable filter can be viewed as an adjoint
notion for the generic of a group. In this section, we will assume
that \(G\) is a pro-\(\emptyset\)-definable group.

Let \(\Delta(x)\) be a set of formulas \(\phi(x,y)\) where \(x\)
ranges over \(G\), which is preserved by \(G\) (on the
left): for all \(\phi(x,m)\) instance of \(\Delta\) and \(g\in G\),
\(\phi(g\cdot x,m)\) is equivalent to an instance of \(\Delta\).

\begin{defn}
Let \(\pi\) be a definable filter concentrating on
\(G\). We define \[\Stab_{\Delta}(\pi) := \{g\in G \mid
\restr{\transl{g}{\pi}}{\Delta} = \restr{\pi}{\Delta}\}.\]

If \(\Delta\) is the set of all formulas, we write \(\Stab(\pi)\) for
\(\Stab_{\Delta}(\pi)\).
\end{defn}

Note that since \(\Delta\) is preserved by \(G\), \(G\) acts on
restrictions of filters to instances of \(\Delta\). It follows that
\(\Stab_{\Delta}(\pi)\) is a pro-definable subgroup of \(G\).

\begin{rem}
Assume \(G\) is a pro-limit of \(\emptyset\)-definable groups. If
\(\phi(x,y)\) is any formula where \(x\) ranges over \(G\), then
\(\phi'(x,y,t) = \phi(t\cdot x,y)\) where \(t\) also ranges over \(G\)
is preserved by \(G\). It follows that any finite set of
formulas is contained in a finite set \(\Delta\) preserved by
\(G\). For such a \(\Delta\), \(\Stab_{\Delta}(\pi)\) is a
relatively definable subgroup of \(G\).

For infinite \(\Delta\), it follows that \(\Stab_{\Delta}(\pi) =
\bigcap_{\phi\in\Delta}\Stab_{\phi'}(\pi)\) is an intersection of
relatively definable subgroups of \(G\), i.e. a pro-limit of definable
groups.
\end{rem}

\begin{rem}
If \(\pi\) concentrates on \(\Stab(\pi)\), then \(\pi\) is generic in
\(\Stab(\pi)\).
\end{rem}

\begin{lem} Assume $G$ admits a (right) generic type, then there is a smallest pro-definable subgroup \(G^0\leq G\) of
bounded index and \(G/G^0\) is pro-finite. In fact, \(G^0\) is the left stabilizer of any right generic of \(G\) (and the right stabilizer of any left generic).
\end{lem}

\begin{proof}
Let $p$ be a right generic type of $G$ and $H\leq G$ be a pro-definable subgroup of bounded index. Then $H = \bigcap_i X_i$ is an intersection of relatively definable subsets and for each $i$, finitely many translates of $X_i$ cover $G$. It follows that $p$ concentrates on a coset of $H$ and hence that $\Stab(p)\leq H$. Since $p$ has boundedly many translates, $\Stab(p)$ is itself a pro-definable subgroup of bonded index and it is therefore the smallest one. In particular $G^0 = \Stab(p)$ does not depend on the choice of $p$.

Furthermore, for every formula \(\phi\) preserved by \(G\)---which exist since \(G\) is
a pro-limit of definable groups---\(G^{0}_{\phi}\) has finite index in
\(G\) and \(G/G^{0} = \projlim_{\phi}G/G^{0}_{\phi}\) is pro-finite.
\end{proof}

Any generic type of $G$ has a (unique) translate that concentrates on $G^0$. It is called a principal generic of $G$. If \(G = G^{0}\), we say that \(G\) is connected.

If \(\pi\) and \(\mu\) are definable filters, we also define
\(\Stab(\pi,\mu) := \{g\in G\mid \transl{g}{\pi} = \mu\}\). It might
be empty, but when it is non-empty, it is a left torsor of
\(\Stab(\mu)\) and a right torsor of \(\Stab(\pi)\). In particular,
these two groups are then conjugate (by any element in
\(\Stab(\pi,\mu)\)).

We will also consider the following more general variant of the
stabilizer. 

\begin{defn}\label{D:Gen Stab}
  Let $\pi$ be a definable filter. Let $\fF(\pi)$ be the semigroup under composition of $\Uu$-definable functions $h$ such
  that $\push{h}{\pi} = \pi$. This semigroup has a quotient
  consisting of the $\pi$-germs of elements of $\fF(\pi)$. The
  invertible germs form a group, denoted $\fG(p)$.
\end{defn}

Note that given an $\emptyset$-definable family $(h_a)_a$ of functions defined on $\pi$, the set of $a$ such that $\push{(h_a)}{\pi} = \pi$ is a $\infty$-definable set. It follows that $\fF(\pi)$ and $\fG(\pi)$ are piecewise pro-definable. If $\pi$ is a filter concentrating on a pro-definable group $G$, then the stabilizer
$\Stab(\pi)$ embeds naturally into $\fG(p)$: $c$ maps to
the germ of left translation by $c$.

\subsection{Symmetric generics}

Let \(p\) be a definable type. We say that \(p\) is symmetric
if for all definable type \(q\) and \(C\) such that \(p\) and
\(q\) are defined over \(C\), \(a\models p|\acl(C)\) and \(b\models
q|\acl(Ca)\), then \(a\models
p|\acl(Cb)\). Proposition\ \ref{P:stdom}.(\ref{sym}) exactly states that
stably dominated types are symmetric.

\begin{lem}\label{L:gen sym}
Assume \(G\) admits a symmetric right generic type, then left and
right generics coincide, they are all symmetric and there is a unique
left (respectively right) orbit of generics. In particular, there are
only boundedly many generics in \(G\).
\end{lem}

\begin{proof}
Let \(p\) be a symmetric right generic. Then \(p^{-1}\) is a symmetric
left generic. Let \(q\) be any left generic and \(C\) be a model over
which \(p\) and \(q\) are defined, \(a\models p|C\) and \(b\models
q|\acl(Ca)\). By definition \(b a \models
\rtransl{q}{a}|\acl(Ca)\). Since \(p\) is symmetric, \(a\models
p|\acl(Cb)\) and hence \(b a\models\transl{b}{p}|\acl(Cb)\). So
\(\rtransl{q}{a}\) and \(\transl{b}{p}\) agree on \(C\) but, since
both are \(C\)-definable and \(C\) is a model, \(\rtransl{q}{a} =
\transl{b}{p}\). Since right and left translation commute, it follows
that \(p\) is left generic. Moreover if \(q_1\) and \(q_2\) are both
left generic, then, by the previous argument, we find \(b_1\) and
\(b_2\) and \(a\) such that \(\rtransl{q_1}{a} = \transl{b_1}{p}\) and
\(\rtransl{q_2}{a} = \transl{b_2}{p}\). It follows that
\(\transl{b_2b_1^{-1}}{q_1} = q_2\) and all left generics are left
translates of one another. Since \(p\) is a symmetric left generic,
all left generics are symmetric.

By a similar argument starting with any (symmetric) left
generic, we get that all left generics are right generic, that all
right generics are right translates of each other and that all right
generics are symmetric (and hence left generic). We have proved that
left and right genericity coincide and that they are all left and
right translates of each other.
\end{proof}

\begin{lem}\label{R:1genr}
  Assume \(G\) admits a symmetric generic type \(p\).  Then the
  following are equivalent:
  \begin{enumerate}
  \item $p$ is the unique generic type of $G$.
  \item For all $g \in G$, $\transl{g}p = p$.
  \item $G$ is connected.
  \end{enumerate}
\end{lem}

\begin{proof}
Assume (1).  By definition of genericity, for any $g \in G$,
$\transl{g}{p}$ is right generic. Hence by uniqueness, $\transl{g}{p}
= p$. Conversely given (2), let $q$ be generic.  Then by
Lemma\ \ref{L:gen sym}, $q = \transl{g}{p}$ for some $g$.  By (2),
$p=q$. The equivalence with (3) is immediate, given that \(G^{0}\) is
the left stabilizer of any generic.
\end{proof}

In particular, if $G$ admits a symmetric generic type, then there is a unique principal generic.

\begin{lem}\label{L:defbdd}
Assume \(G\) is pro-\(C\)-definable and admits a symmetric
generic type . Then \(G^{0}\) is pro-\(C_0\)-definable for some
\(C_0\subseteq C\) of size at most \(|\cL|+|x|\) where \(x\) ranges
over \(G\).
\end{lem}

\begin{proof}
Let \(p\) be the principal generic of \(G\). Since \(p\) is fixed by every
automorphism fixing \(G\) globally, it is \(C\)-definable. Moreover,
it is determined by the function
\(\phi(x,y)\mapsto\defsc{p}{x}\phi(x,y)\) and there are at most
\(|\cL|+|x|\) formulas \(\phi\) to consider. So \(p\) is
\(C_0\)-definable for some \(C_0\subseteq C\) of size at most
\(|\cL|+|x|\) and hence so is \(G^{0} = \Stab(p)\).
\end{proof}

\subsection{Group chunks}

This idea, in the context of algebraic groups, is due to Weil. Let $G$
be a definable group, and \(\pi\) a left \(G\)-invariant definable
filter concentrating on \(G\). The $\pi\tensor \pi$-germ of
multiplication is called the \emph{group chunk} corresponding to
$(G,\pi)$.

\begin{defn}
  An \emph{abstract group chunk} over \(C\) is a \(C\)-definable filter
  \(\pi\) and a pro-\(C\)-definable map \(F\) defined on
  \(\pi^{\tensor 2}\) (or $\pi\tensor \pi$-germ of such a function)
  such that:
\begin{enumerate}
\item For all \(a\models\pi|C\), \(\push{(F_a)}{\pi} = \pi\), where
  \(F_a(x) = F(a,x)\).
\item For all \(a,b\models \pi^{\tensor 2}|C\), \(a\in\dcl(C,b,F(a,b))\)
  and \(b\in\dcl(C,a,F(a,b))\). 
\item \(\pi^{\tensor 3}(x,y,z) \vdash F(x,F(y,z)) = F(F(x,y),z)\).
\end{enumerate}
\end{defn}

\begin{prop}\label{P:gpchunk}
Let \((\pi,F)\) be an abstract group chunk over $C$. Then \((\pi,F)\) is
pro-$C$-definably isomorphic to the group chunk of a pro-$C$-definable group
\(G\): there exists an injective pro-$C$-definable map \(f: \pi \to G\)
such that \(\pi^{\tensor 2}(x,y)\vdash f(F(x,y)) = f(x)\cdot f(y)\)
and \(G = \Stab(\push{f}{\pi})\).
\end{prop}

\begin{proof}
  Let $P= \pi|C$. For any $a \in P$, by (2), the $\pi$-germ $f(a)$ of
  $F(a,x)$ is invertible. It is therefore an element of the group
  $\fG(\pi)$ (cf. Definition\ \ref{D:Gen Stab}). By (2) also, the map
  $f$ is injective. Let $G$ be the subgroup of $\fG(\pi)$ generated by
  the elements $(f(a))_{a\in P}$ and their inverses. Let $I$ enumerate
  all finite sets of the variables of $\pi$. For all $i\in I$, and
  $g\in\fG(\pi)$, let $g_i$ denote the $\pi$-germ of $\rho_i\circ h$,
  where $h$ is any map whose $\pi$-germ is $g$ and $\rho_i$ is the
  projection to the variables in $i$.
  
  To show that $G$ can be identified with a pro-definable group, it
  suffices to check that:
  \begin{itemize}
  \item Any element of $G$ is a product  $f(a) \cdot f(b)^{-1}$ for some $a,b\in P$.
  \item The set $\{(a_1,\ldots,a_6)\in P^6\mid (f(a_1)\cdot
    f(a_2)^{-1}\cdot f(a_3)\cdot f(a_4)^{-1})_i = (f(a_5)\cdot
    f(a_6)^{-1})_i\}$ is relatively definable.
  \end{itemize}

  The second point is immediate from the definability of $\pi$. As for
  the first point, it suffices to show that a product $f(a) \cdot
  f(b)^{-1} \cdot f(c) \cdot f(d)^{-1}$ has the required form. Note
  that, by (1), when $a\in P$ and \(b\models \pi|Da\) for any
  \(D\supseteq C\), \(F(a,b) \models \pi|Da\) and there exists
  \(c\models \pi|Da\) such that \(F(a,c) = b\) and, by (3),
  $f(a)^{-1}\cdot f(b) = f(c)$. Thus given any $a,b,c,d \in P$, let
  $e_0 \models \pi | C a b c d$; then $f(d)^{-1} \cdot f(e_0) =
  f(e_1)$ for some $e_1 \models \pi | C a b c d$; so $f(c)\cdot f(e_1)
  = f(e_2)$ for some $e_2 \models \pi | C a b d$; continuing this way,
  we obtain that $f(a)\cdot f(b)^{-1}\cdot f(c) \cdot f(d)^{-1}\cdot
  f(e_0)= f(e_4)$, for some \(e_4 \models \pi | C a b d\). Therefore,
  $f(a)\cdot f(b)^{-1} \cdot f(c) \cdot f(d)^{-1} = f(e_4)\cdot
  f(e_0)^{-1}$ as required.

  Finally, if \(a\in P\) and \(b\models \pi|Da\) for some \(D\supseteq
  C\), \(f(a)\cdot f(b) = f(F(a,b))\) where \(F(a,b)\models\pi|Da\) so
  \(\transl{f(a)}{\push{f}{\pi}} = \push{f}{\pi}\) and we do have \(G
  = \Stab(\push{f}{\pi})\). This concludes the construction of $G$ and
  the proof of its various properties.
\end{proof}

The uniqueness of \(G\) in Proposition\ \ref{P:gpchunk} is guaranteed by
the following:

\begin{prop}\label{P:gpchunk fct}
  Let \(G_1\) and \(G_2\) be pro-\(C\)-definable groups, \(\pi\) be a
  left \(G_1\)-invariant \(C\)-definable filter concentrating on
  \(G_1\) and \(f : \pi \to G_2\) be a pro-\(C\)-definable map such
  that \(\pi^{\tensor 2}(x,y)\vdash f(x\cdot y) = f(x)\cdot
  f(y)\). Then there exists a unique pro-\(C\)-definable homomorphism
  \(g:G_1\to G_2\) such that \(\pi(x)\vdash f(x) = g(x)\).
\end{prop}

\begin{proof}
  Uniqueness of $g$ is clear, since $\pi|C$ generates $G$: if $g \in
  G$, $a \models \pi | Cg$, then $g\cdot a \models\pi|Cg$ and $g =
  (g\cdot a)\cdot a^{-1}$.

  Existence is also clear, provided we show that, for all
  $a,b,c,d\models\pi|C$ such that $a\cdot b^{-1}=c\cdot d^{-1}$,
  $f(a)\cdot f(b)^{-1}=f(c)\cdot f(d)^{-1}$. It suffices to show that
  $f(a)\cdot f(b)^{-1}\cdot f(e) =f(c)\cdot f(d)^{-1}\cdot f(e)$ for
  any $e\models \pi|C a b c d$. But it follows from our
  hypothesis on $f$ that $f(b)^{-1}\cdot f(e) = f(b^{-1}\cdot e)$ and
  $f(d)^{-1}\cdot f(e) = f(d^{-1}\cdot e)$. Moreover, $b^{-1}\cdot
  e\models p|Cabcd$, so $f(a)\cdot f(b^{-1}\cdot e) = f(a\cdot
  b^{-1}\cdot e)$. Similarly, $f(c)\cdot f(d^{-1}\cdot e) = f(c\cdot
  d^{-1}\cdot e)$. Since $a\cdot b^{-1}\cdot e = c\cdot d^{-1}\cdot
  e$, the equality holds.

  Since $g(a) = b$ if and only if $\defsc{p}{x}f(a\cdot x) = b\cdot
  f(x)$, $g$ is pro-\(C\)-definable.
\end{proof}

Analogous statements for group actions exist.

\subsection{Products of types in groups}\label{S:convolution}

Let $G$ be a piecewise pro-definable group, let $p_1,\ldots,p_n$ be definable
types of elements of $G$, and let $w$ be an element of the free group
$F$ on generators $\{1,\ldots,n\}$. The words of the free group will
be denoted by expressions such as $1 \overline{2} 3$; $\overline{2}$ is the
inverse of the generator $2$.

We construct a definable type $p_w = w_\star(p_1,\ldots,p_n)$. Let $a_i
\models p_i | \acl(Ca_1 \ldots a_{i-1})$.  Let $a_w =
w(a_1,\ldots,a_n)$ be the image of $w$ under the homomorphism $F \to
G$ with $i \mapsto a_i$.  Let $p_w | C = \tp(a_w/C)$.

If $w$ is the product of the generators $1,2$, we write $p_1 \star
p_2$ for $p_w$. We denote \(p_{123\ldots n}\) as \(\bigoast_{i=1}^n
p_i\). Note that if \(G\) is Abelian and \(p\) is symmetric, then the
order of enumeration does not actually matter. If a single type $p$ is
given, rather than a sequence $p_w$ will refer to the sequence
$(p,p,\ldots,p)$. Let $p^{\star n}$ denote $p_{123\cdots n}$, $p^{\pm
  2n} = p_{\overline{1} 2 \overline{3} 4 \cdots (2n)}$ and
$p^{\pm 2n+1} = p_{1 \overline{2} 3 \overline{4} \cdots
  \overline{(2n)} (2n+1)}$.

\section{Stably dominated groups}\label{S:stdomgp}

In this section, $T$ is assumed to be metastable.

\begin{defn}\label{D:stdomgp}
A pro-definable group \(G\) is \emph{stably dominated} if \(G\) has
a stably dominated generic type.
\end{defn}

Since stably dominated types are symmetric, it follows from
Lemma\ \ref{L:gen sym} that all generics of \(G\) are both left and
right generics and they are all stably dominated.

\begin{rem}\label{R:gmpres}
The class of stably dominated pro-definable groups is closed under
Cartesian products and image under a definable group homomorphism. If
\(G\) and \(H\) are stably dominated pro-definable groups and \(p\)
(respectively \(q\)) is a stably dominated generic of \(G\)
(respectively \(H\)), then \(p\tensor q\) is a stably dominated
generic of \(G\times H\). If \(f:G\to H\) is a pro-definable
surjective group homomorphism, then \(\push{f}{p}\) is a stably
dominated generic of \(H\).
\end{rem}

\begin{lem}\label{L:stdomext}
Let \(G\) be a pro-definable group and \(N \leq G\) be a stably
dominated pro-definable subgroup. Assume that there exists a stably
dominated type concentrating on $G/N$ whose orbit under
$G$-translations is bounded. Then $G$ is stably dominated.

In particular, if $N \normal G$ and $N$, $G/N$ are stably dominated,
then so is $G$.
\end{lem}

\begin{proof}
Let \(p\) be the generic type of $N^0$ and $q$ be a stably dominated
type of $G/N$, with bounded orbit under \(G\)-translations. Let \(C\)
be such that \(G\), \(N\) are defined over \(C\) and \(p\), \(q\) are
stably dominated over \(C\).

Assume first that \(N = N^0\). For all \(n\in N\), \(\transl{n}{p} =
p\). Thus, if \(c=d n\) for some \(d\in G\) and \(n \in N\), then \(\transl{c}{p} =
\transl{d n}{p} = \transl{d}{p}\). So the type $\transl{c}{p}$ depends
only on the coset $s=c N$. We denote it $p_s$. Let us now define the
type \(r\) as follows: a realization of \(r\) over \(B\supseteq C\) is
a realization of $p_s|Bs$ where $s \models q|B$. Note that, for any $c\in G$, a realization of $\transl{c}{r}|Bc$ is a realization of $p_s|Bcs$ where $s \models \transl{c}{q}|Bc$. It follows that $r$ has at most as many $G$-translates as $q$.

Since \(p\) is stably dominated over \(C\), for any \(c\in s\), \(p_s\) is stably dominated over \(Cc\). By descent
(Proposition\ \ref{P:base change}.(\ref{descent})), \it is, in fact, stably
dominated over
\(Cs\). It follows, by transitivity
(Proposition\ \ref{P:stdom}.(\ref{ltrans})), that \(r\)
is stably dominated over \(C\). That concludes the proof in the case
\(N = N^0\).

In general, the map \(\pi : G/N^0 \to G/N\) has pro-finite fibers and
hence for any \(s\in G/N^0\), \(s \in\acl(C\pi(s))\). It follows, by
Proposition\ \ref{P:stdom}.(\ref{aclstdom}), that any \(q_0\) such that
\(\push{\pi}{q_0}= q\) is stably dominated over \(C\). Fix such a
\(q_0\), then \(q_0\) has bounded orbit under \(G\)-translations and,
by the above, \(G\) is stably dominated.
\end{proof}

\begin{lem}\label{L:stdomsubgp}
Let \(G\) be a stably dominated pro-definable group and \(N\leq G\) be
a pro-definable subgroup. Let \(\eta: G \to G/N\) be the map
\(\eta(g) = g\cdot N\). Assume that there exists a pro-definable \(Y \subset
G\) such that \(\restr{\eta}{Y}\) is surjective, bounded to one. Then \(N\) is
stably dominated.
\end{lem}

\begin{proof}
Let \(p\) be the principal generic of \(G\), \(C\) be such that \(G\),
\(N\), \(Y\), \(p\) are \(C\)-definable. Let \(Q\) be the set of types
\(\tp(h^{-1}\cdot g/C)\) where \(g\models p|C\), \(h\in Y\) and
\(h\cdot N = g\cdot N\). Because there are boundedly many choices for
\(h\), the set \(Q\) is bounded. Moreover, for all \(g\models p|C\),
\(h^{-1}\cdot g\in\acl(Cg)\) and hence, by
Proposition\ \ref{P:stdom}.(\ref{aclstdom}), the types in \(Q\) are
all stably dominated types concentrating on \(N\).

Let us now show that \(Q\) is \(N\cap G^0\)-invariant. Pick \(n\in
N\cap G^0\), \(g\models p|C n\), \(h\in Y\) such that \(h\cdot N = g\cdot N\)
and \(q := \tp(h^{-1}\cdot g/C) \in Q\). We have that \(h^{-1}\cdot
g\cdot n\models \rtransl{q}{n}|C n\), but since \(g\cdot n\models
\rtransl{p}{n} = p\) and \(g\cdot n\cdot N= g\cdot N = h\cdot N\),
\(\rtransl{q}{n}|C = \tp(h^{-1}\cdot g\cdot n/C) \in Q\).

It follows that the size of the orbits of the types in \(Q\) is bounded by $|Q|\cdot|N/N\cap G^0|$. Recall that $Q$ is bounded and that \(N/N\cap G^0\) can be embedded in \(G/G^0\) whose size is bounded. So all the types in \(Q\) are generic.
\end{proof}

\begin{cor}\label{C:stdomsubgpvf}
Let \(G\) be an algebraic group, \(N\leq G\) an algebraic
subgroup. Let \(H\leq G\) be definable in \(\ACVF\) and stably
dominated. Then \(H\cap N\) is stably dominated.
\end{cor}

\begin{proof}
By elimination of imaginaries in algebraically closed fields, the
coset space \(X := G/N\) can be identified with a subset of some
Cartesian power of the field sort. Let \(\eta\) be the projection
\(G\to X\). Let \(M_0\models\ACVF\) be such that \(G\), \(N\) and
\(H\) are defined over \(M_0\). For all \(h\in H\), \(M =
\alg{M_0(\eta(h))}\models\ACVF\), so there exists \(y\in H(M)\cap
h\cdot N\). By compactness, we find a definable subset \(Z\) of the
graph of \(\restr{\eta}{H}\) with finite projection to \(X\). Let
\(Y\subseteq H\) be the projection of $Z$ on \(H\). Then, applying
Lemma\ \ref{L:stdomsubgp} to \(H\), \(H\cap N\) and \(Y\), we get that
\(H\cap N\) is stably dominated.
\end{proof}

\subsection{Domination via a group homomorphism}

\begin{prop}\label{P:groupdom}
  Let $G$ be a stably dominated pro-definable group. There exists a
  pro-definable stable group $\fg$, and a pro-definable homomorphism
  $\theta: G \to \fg$, such that the generics of $G$ are stably dominated
  via $\theta$.

  If \AFD holds and \(G\) is definable then $\fg$ can be taken to be
  definable.
\end{prop}

We say that such a \(G\) is stably dominated via \(\theta\).

\begin{proof}
Let \(p\) be generic in \(G\), \(C\) be such that \(G\) and \(p\) are
\(C\)-definable. Let \(\theta(a)\) enumerate \(\St_C(a)\), and, for
all \(a\), \(b\in G\), let \(f_a(b) = \theta(a\cdot b)\). Fix \(g\in
G\). Since \(\transl{g}{p}\) is stably dominated, by
Proposition\ \ref{P:stgerms}, the \(\transl{g}{p}\)-germ of \(f_a\) is
strong and is in \(\St_C(a) = \theta(a)\). It follows that for all
\(b\models \transl{g}{p}|Ca\), \(f_a(b) = h_{\theta(a)}(b)\) for some
\(C\)-definable map \(h\). Since \(\St_C\) is stably embedded, \(h\)
factors through $\theta(b)$. Indeed, let $c = f_a(b) =
h_{\theta{a}}(b) $.  Since $c \in \St_C$, $\tp(\theta(a)c /
C\theta(b)) \vdash \tp(\theta(a)c/Cb)$.  Thus
$c\in\dcl(\theta(a)\theta(b))$. We have proved that, for all \(g_1\),
\(g_2\in G\), there exists a \(C\)-definable map \(F_{g_1,g_2}\) such
that, for all $a\models \transl{g_1}{p}|C$, $b \models
\transl{g_2}{p}|Ca$, $\theta(ab) =
F_{g_1,g_2}(\theta(a),\theta(b))$. By compactness, we may assume that
\(F = F_{g_1,g_2}\) does not depend on \(g_1\) or \(g_2\).

Recall that for all $g\in G$, \(\transl{g}{p}\) is stably dominated
over \(C\). It follows that, for all \(g_1\), \(g_2\in G\), \(a\models
\transl{g_1}{p}|C\), \(b\models\transl{g_2}{p}|Ca\) and \(c = a\cdot
b\), \(\tp(\theta(b)\theta(c)/C\theta(a))\vdash\tp(b c/Ca)\) and hence
\(\theta(b)\in\dcl(\theta(a),\theta(c)) =
\dcl(\theta(a),F(\theta(a),\theta(b)))\). Symmetrically, we have
\(\theta(a)\in\dcl(\theta(b),F(\theta(a),\theta(b)))\).

Thus, \(F\) is
a group chunk on \(\pi := \bigcap_{g\in
  G}\push{\theta}{\transl{g}{p}}\) and, by
Proposition\ \ref{P:gpchunk}, \(F\) is the restriction of the
multiplication map of some pro-definable group \(\fg\) in \(\St_C\). By
Proposition\ \ref{P:gpchunk fct}, \(\theta\) extends to a
pro-definable group homomorphism from \(G\) to \(\fg\).

Let us now assume \AFD. By Proposition\ \ref{P:prodefgp}, \(\fg :=
\projlim_{i\in I} \fg_i\) where the \(\fg_i\) form a (filtered)
projective system of definable groups. Let \(\theta_i : G \to \fg_i\) be
the canonical projection. Since there are no descending chains of
definable subgroups of $\fg_i$, every $\infty$-definable subgroup is
definable, in particular \(\theta_i(G)\) is definable. Then \(\fg =
\projlim_{i\in I} \theta_i(G)\) and we may assume that every \(\theta_i\) is
surjective.

Let \(a\) realize the principal generic of \(G\) over \(C\).  By
Lemma\ \ref{L:afg}, $\St_C(a)$ is $\acl$-finitely generated over $C$;
say $\St_C(a) \subseteq \acl(Cd)$ for some tuple \(d\in\St_{C}\).
Since $\dimst(d/C(\theta_i(a))$ decreases as $i$ increases, it stabilizes
at some $i_0$ (Note that, if \(C\) is sufficiently large, \(i_0\) does
not depend on \(C\)).  But all $\theta_i(a)$ are in $\acl(Cd)$, and hence
$\theta_i(a) \in \acl(C(\theta_{i_0}(a)))$ for all $i \geq i_0$.  It follows
$\tp(a/C)$ is dominated via $\theta_{i_0}$ and so are all of its left
translates.
\end{proof}

\begin{rem}
If \(\AFD\) holds, we have also proved the following. If \(\theta : H \to
\fg = \projlim_i \fg_i\) is a pro-definable surjective group
homomorphism and that \(\dimst(\fg) = n < \infty\), then there exists
\(i_0\) such that \(\dimst(\fg_{i_0}) = n\). Moreover, since there are no
descending chains of definable subgroups of $\fg_{i_0}$,
\(\theta_{i_0}(G)\) is definable. So we can actually find a pro-definable
surjective group homomorphism onto a \emph{definable} stable group of
Morley rank \(n\).

In such a situation, we say that \(H\) has stable homomorphic image of
Morley rank $n$.
\end{rem}

The homomorphism $\theta$ of Proposition\ \ref{P:groupdom} is not uniquely
determined by $G$; even if $G$ is stable, there may be a non-trivial
homomorphism of this kind. We do however have a maximal one.

\begin{rem}
There exists a pro-$C$-definable stable group $\fg $, and a
pro-$C$-definable homomorphism $\theta: G \to \fg$, maximal in the sense
that any pro-$C$-definable homomorphism $\theta': G \to \fg'$ into a
pro-$C$-definable stable group factors through $\theta$.
\end{rem}

The kernel of this maximal $\theta$ is uniquely determined. If $G$ is
stably dominated it will be stably dominated via this maximal
homomorphism.

\begin{lem}\label{L:groupdomcor}
Let $G$ be a pro-\(C\)-definable group stably dominated via some
pro-\(C\)-definable surjective group homomorphism $\theta: G \to \fg$. Then $\tp(a/C)$
is generic in $G$ (i.e. $a \models r |\acl(C)$ for some generic $r$ of
$G$) if and only if $\tp(\theta(a)/C)$ is generic in $\fg$.
\end{lem}

\begin{proof}
  Assume $\tp(\theta(a)/C)$ is generic in $\fg$, then so is
  $\tp(\theta(a)/\acl(C))$. So we may assume $C = \acl(C)$. Let $p$ be a
  generic type of $G$, stably dominated via $\theta: G \to \fg$. Then
  $q=\transl{a}{p}$ is a generic in \(G\) and hence stably dominated
  via \(\theta\). Let $b \models p | Ca$. Note that $a\cdot b \models q
  |Ca$. Since $\theta(a)$ is generic in $\fg$, we have $\theta(a\cdot b) =
  \theta(a)\cdot \theta(b)\ind_C \theta(b)$. Since \(\theta|C\) is stably dominated via
  \(\theta\), \(b\models p|C\theta(a\cdot b)\) and hence \(\St_C(b)\ind_C
  \theta(a\cdot b)\). Since \(q|C\) is stably dominated via \(\theta\), \(a\cdot
  b \models q|Cb\). So $a \models \rtransl{q}{b^{-1}} |Cb$, and in
  particular $a \models \rtransl{q}{b^{-1}}|C$, where
  \(\rtransl{q}{b^{-1}}\) is generic in \(G\).

  The converse is obvious since for any generic \(r\) of \(G\),
  \(\transl{\theta(x)}{\theta_\star r} = \theta_\star(\transl{x}{r})\).
\end{proof}

\begin{cor}\label{C:finindex}
Let $G$ be a pro-definable group stably dominated via some
pro-definable group homomorphism $\theta: G \to \fg$. Let $H\leq G$ be a
relatively definable.  If $\theta(H) = \fg$, then $H$ has finite index in
$G$.
\end{cor}

\begin{proof}
Work over some $C$ over which \(G\) and \(H\) are definable. Let
\(b\in H\) be such that \(\theta(b)\) is generic in \(\fg\) over \(C\). By
Lemma\ \ref{L:groupdomcor}, $\tp(b/C)$ is generic in $G$. Thus a
generic of $G$ lies in $H$, so $H$ has finite index in $G$.
\end{proof}

\begin{cor}\label{C:def G0}\Aom
Let $G$ be a \(C\)-definable group stably dominated via some
\(C\)-definable group homomorphism $\theta: G \to \fg$. Then $G^0$ is
definable, i.e. $[G:G^0]$ is finite.
\end{cor}

\begin{proof}
We may assume \(\theta\) surjective. Note that \(\theta^{-1}(\fg^0)\) is a definable subgroup of \(G\) with
finite index, so we may assume that \(\fg\) is connected. Since
\([G:G^0]\) is bounded, we can find \(C\) such that every coset of
\(G^0\) in \(G\) contains a point from \(C\). We can also assume that
\(C\) is a metastability basis as in the definition of \Aom. Let
\(a_0\) be such that \(\theta(a_0)\) is generic in \(\fg\) over \(C\). By
\Aom, there exists \(a\in G\) such that \(\theta(a) = \theta(a_0)\) and
\(\tp(a/C\theta(a))\) is isolated. Since \(\theta(a)\) is generic in \(\fg\)
over \(C\), by Lemma\ \ref{L:groupdomcor}, \(\tp(a/C)\) is generic in
\(G\). Let \(p\) be the principal generic of \(G\). Then there exists
\(c\in G(C)\) such that \(b := c\cdot a\models p|C\). It follows that
\(\tp(c\cdot a/C\theta(a)) = \tp(b/C\theta(b))\) is isolated.

On the other hand, let \(G^0 = \bigcap_i G_i\) where $\{(G_i)_{i\in
  I}\}$ is a bounded family of definable subgroups of finite index,
closed under intersections. Note that by Lemma\ \ref{L:groupdomcor},
\(\{x\in G_i\mid i\in I\}\cup\{\theta(x) = \theta(b)\}\) generates
\(\tp(b/C\theta(b))\). This type being isolated, there exists \(i_0\in I\)
such that it is generated by \(\theta(x) = \theta(b)\) and \(x\in
G_{i_0}\). Since \(\fg\) is connected, any generic of \(G_i\) has
a realization \(d\) such that \(\theta(d) = \theta(b)\) and hence is equal to
\(p\). So $G_i=G^0$, and $G^0$ is definable.
\end{proof}

Here is a characterization of generics that does not explicitly mention $\theta$ and $\fg$.

\begin{cor}    
Let $G$ be a stably dominated pro-$C$-definable group, $\dimst(G) <
\infty$.  Then the generic types of $G$ over $C$ are precisely the
types $\tp(c/C)$ such that for $h$ generic in $G$ over $Cc$,
$\dimst(\St_C(h\cdot c) / Ch)$ is maximal (as \(c\) varies).
\end{cor}
 
\begin{proof}
Let $\theta$ and $\fg$ be as in Proposition\ \ref{P:groupdom}. Since
$h\cdot c$ is generic in $G$ over $C$, $\acl(\St_C(h\cdot c)) =
\acl(C\theta(h\cdot c)) = \acl(C\theta(h)\theta(c))$. So $\dimst(\St_C(h\cdot c)/Ch) = \dimst(\theta(h)\cdot \theta(c)/Ch) =
\dimst(\theta(c)/Ch) \leq \dimst(\theta(c)/C) \leq \dimst(\fg)$. 

If $c\models p|\acl(C)$ for some stably dominated generic
$p$, by symmetry, $c\models p|\acl(Ch)$ and hence $h\cdot c$ is
generic in $G$ over $Ch$. Thus $\theta(h\cdot c)$ is generic in $\fg$ and
\(\dimst(\St_C(h\cdot c)/Ch) = \dimst(\fg)\) is maximal. Conversely, if \(\dimst(\St_C(h\cdot c)/Ch)\) is maximal, \(\dimst(\St_C(h\cdot
c)/Ch) = \dimst(\theta(c)/C) = \dimst(\fg)\). It follows that $\theta(c)$ is
generic in $\fg$ over $C$ and, by Lemma\ \ref{L:groupdomcor}, $c$ is generic in
\(G\) over $C$.
\end{proof}

Note that the proof uses, rather than proves, the existence of a
generic.

\begin{rem}
If $G$ is piecewise-definable, in a superstable theory, and $p$ is a
type of maximal rank in $G$, then $p$ is a translate of a generic of
$\Stab(p)$, a definable subgroup of $G$.

For stably dominated types, the analog need not hold. In $\ACVF$,
consider $G = \Ga(\Oo)^2$, fix \(\gamma\) a positive element of the
value group and let $p$ be the stably dominated type of elements such
that $x$ is generic in $\Oo$ and $y$ is generic (over $x$) in the
closed ball of radius $\gamma$ centered at $x^2$. Then $\dimst(G) = 2$
and the stable part of \(p\) has dimension $2$. But, although $p$ has
maximal rank, it is not a generic type. However, we do have that
$p^{\star 2}$ is a stably dominated generic of $G$.
\end{rem}

Let us now prove a certain converse to Proposition\ \ref{P:groupdom}:

\begin{prop}\label{P:a2def}
Let $G$ be a pro-definable group. Assume $G$ admits a surjective
pro-definable homomorphism $\theta:G \to \fg$, with $\fg$ stable and with
$\dimst(\fg) = \dimst(G) = n < \infty$. Then, there exists a
pro-definable subgroup $T$ of $G$ with:
\begin{enumerate}
\item $T$ connected stably dominated;
\item $T$ normal and $G/T$ almost internal to $\Gamma$;
\item $\theta(T)= \fg^0$;
\end{enumerate}

and $T$ is uniquely determined by (1,2) or by (1,3). Moreover, \(T\)
is stably dominated via \(\restr{\theta}{T}\).
\end{prop}

\begin{proof}
Let $Z$ be the collection of types $q$ of elements of $G$, definable
over some set of parameters, with $\theta_\star q$ generic in $\fg$, and
$q$ stably dominated.

\begin{claim}\label{C:Z non empty}
\(Z\neq\emptyset\).
\end{claim}

\begin{proof}
Let \(C=\acl(C)\) be a metastability basis over which \(G\) and \(\theta\)
are pro-definable. Let \(a \in G\) be such that \(\theta(a)\) generic in
\(\fg\) over \(\acl(C)\). Let \(\gamma := \Gamma_C(a)\). Any
relatively $\acl(C\gamma)$-definable subset of $\fg$ is already
relatively $C$-definable, so \(\theta(a)\) is generic in \(\fg\) over
\(\acl(C\gamma)\). By metastability, \(\tp(a/\acl(C\gamma))\) is
stably dominated and extends to a unique \(\acl(C\gamma)\)-definable
type which is in \(Z\).
\end{proof}

\begin{claim}\label{C:unique orbit}
Any two elements of \(Z\) are left and right translates of each other.
\end{claim}

\begin{proof}
Pick any $q,r \in Z$; say both are $B$-definable for some \(B =
\acl(B)\) over which \(G\) and \(\theta\) are also defined. Let $a \models
q | B$, $b \models r | B a$ and let \(c = a\cdot b^{-1}\). Then $\theta(a)$
and $\theta(b)$ are $B$-independent generic elements of $\fg$. Hence so are
$\theta(a)$ and $\theta(a)\cdot \theta(b)^{-1} = \theta(c)$. Since $\dimst(\theta(a)/B) = n =
\dimst(G) \geq \dimst(\St_B(a)/B) \geq \dimst(\theta(a)/B)$, we have
$\St_B(a) \subseteq \acl(B\theta(a))$. Similarly \(\St_B(c) \subseteq
\acl(B\theta(c))\). Thus $\St_B(c) \ind_B \St_B(a)$.  By stable domination,
$a \models q | B c$.  Similarly $b \models r | B c$. It follows that
$q=\transl{c}{r}$. By the symmetric argument, $q=\rtransl{r}{d}$ where
\(d = b^{-1}a\).
\end{proof}

It follows that the pro-definable subgroup \(T := \Stab(q) = \Stab(\rtransl{q}{c})\) does not depend on the choice of $q\in Z$. Also, since one can easily check that $Z$ is closed under left translation, $T$ is normal. Moreover if $a$ and $b$ are
independent realizations of $q$ over $B$ and \(c = a\cdot b^{-1}\),
then the proof of Claim\ \ref{C:unique orbit} shows that
\(\transl{c}{q} = q\) and hence \(r := \rtransl{q}{b^{-1}}\) is the
unique generic type of \(T\). Thus \(T\) is connected stably
dominated. Since \(\theta(T)\) contains a generic of \(\fg\) and is
connected, we have \(\theta(T) = \fg^{0}\). Also, as in the above claim, we
have that \(\St_C(c) \subseteq \acl(\theta(c))\). It follows that
\(\tp(c/C)\), and hence \(r\), is stably dominated via
\(\restr{\theta}{T}\).

Since \(T = \Stab(q)\), it is the intersection of a family of
relatively definable subgroups \((T_i)_{i\in I}\) of \(G\).

\begin{claim}
For all \(i\in I\), $G/T_i$ is almost $\Gamma$-internal.
\end{claim}

\begin{proof}
Pick any \(a \in G\) and let \(C\) be a metastability basis over which
\(G\), \(T\) and \(\theta\) are defined. If \(c\models r|Ca\), then
\(a\cdot c \in a\cdot T\) and \(\theta(a\cdot c) = \theta(a)\cdot \theta(c)\) is
generic in \(\fg\) over \(\acl(C)\). So we may assume that \(\theta(a)\) is
generic in \(\fg\) over \(\acl(C)\) without changing the coset
\(a\cdot T\). By the proof of Claim\ \ref{C:Z non empty}, there exists
a generic \(q\), \(\acl(\Gamma_C(a))\)-definable, such that \(a\models
q|\Gamma_C(a)\). By the proof of Claim\ \ref{C:unique orbit}, \(q\)
concentrates on \(a\cdot T\). It follows that for all \(i\in I\),
\(a\cdot T_i \in \acl(\Gamma_C(a))\) and hence $G/T_i$ is almost
$\Gamma$-internal.
\end{proof}

Let us now prove the uniqueness of $T$. Let $T'$ be a connected stably
dominated pro-definable subgroup of $G$. Then \(T'/(T\cap T') \simeq
T\cdot T'/T\leq G/T\) is both almost \(\Gamma\)-internal and connected
stably dominated, hence trivial. So \(T'\leq T\). If \(G/T'\) is
almost \(\Gamma\)-internal, the symmetric proof shows that \(T\leq
T'\). So \(T\) is indeed characterized by (1,2). Now assume that
\(\theta(T') = \fg^0\) and let $T' = \bigcap_j T_j'$ where the $T_j'$ are
relatively definable. By Corollary\ \ref{C:finindex}, \([T:T_j'\cap
  T]\) is finite and therefore \(T\leq T_j'\). It follows that \(T\leq
T'\) and \(T\) is also characterized by (1,3).
\end{proof}

\subsection{Stably dominated subgroups of maximal rank}

While not all stably dominated subgroups are definable, we will show that
subgroups whose residual rank is $\dimst(G)$ are.  

\begin{lem}\label{L:a3def}
Let $T$ be an $\infty$-$C$-definable connected and stably dominated
normal subgroup of some definable group $G$ with $G/T$ almost internal
to $\Gamma$.
\begin{itemize}
\item[\AFD] There exists a $C$-definable subgroup stably dominated $S$
  of $G$ with $S^0=T$.
\item[\Aom] $T$ itself is definable.
\end{itemize}
\end{lem}

\begin{proof}
Let us assume \AFD. Let \(p\) be the principal generic of \(T\), then \(T
= \Stab(p)\) is an intersection of $C$-definable subgroups \(S_i\) of
\(G\); and \(G/S_i\) is almost \(\Gamma\)-internal. By
Proposition\ \ref{P:groupdom}, there exists a definable homomorphism
$\theta: T \to \fg$, with $\fg$ stable, and $T$ stably dominated via
$\theta$. By compactness, $\theta$ extends to a definable homomorphism on some
$S_i$; replacing $G$ by this $S_i$, we may assume $\theta: G \to \fg$.

By Lemma\ \ref{L:almost-int}, some quotient of $G/S_i$ by a finite
normal subgroup is $\Gamma$-internal. Define $\dimo(G/S_i)$ to be the
$o$-minimal dimension of any such quotient. By \AFD, for any $a \in G$
and any \(B\), $\dimo(\Gamma_B(a)/B)$ is finite, so $\dimo(G/S_i)$ is
bounded independently of $i$. Thus for some $i$, for all $j \geq i$,
$\dimo(G/S_j) = \dimo(G/S_i)$. It follows that the natural map $G/S_j
\to G/S_i$ has zero-dimensional fibers, hence it has finite kernel.
This shows that for some $i$, for all $j \geq i$, $S_i/S_j$ is
finite. Thus \([S_i:T]\) is bounded, \(S_i\) is stably dominated,
\(S_i^0 = T\), and (1) holds.

If we assume \Aom, definability of $T$ follows from
Corollary\ \ref{C:def G0}.
\end{proof}

\begin{cor}\label{C:maxrkdef}
  Let $G$ be a definable group, $H$ be a connected
  $\infty$-$C$-definable stably dominated subgroup of $G$, with a
  stable homomorphic image of Morley rank $n:=\dimst(G)$.
  \begin{itemize}
  \item[\AFD] There exists a $C$-definable stably dominated subgroup
    $S$ of $G$ with $S^0=H$.
  \item[\Aom] $H$ itself is definable.
  \end{itemize}
\end{cor}

Note that the group $S$ is stably dominated and has a stable
homomorphic image of Morley rank $n:=\dimst(G)$.

\begin{proof}
Assume \AFD. Let \(\theta : H \to \fg\) be a definable surjective
homomorphism with \(\fg\) definable stable and \(\dimst(\fg) =
n\). Since \(H\) is the stabilizer of its principal generic, \(H =
\bigcap_i H_i\) for some definable subgroups \(H_i\) of \(G\). By
compactness, \(\theta\) extends to a definable surjective group
homomorphism \(H_{i_0} \to \fg\). Note that \(n = \dimst(\fg) \leq
\dimst(H_{i_0})\leq \dimst(G) = n\) so we may assume that \(\theta\)
extends to \(G\). Note also that \(\theta(H) = \fg = \Stab(\theta_\star p) =
\fg^0\). By the uniqueness in Proposition\ \ref{P:a2def}, \(H\) is
normal in \(G\) and \(G/H\) is almost \(\Gamma\)-internal. We can now
conclude using Lemma\ \ref{L:a3def}.
\end{proof}

\section{Abelian groups}

We will use the notation of \S\ref{S:Stab} and \S\ref{S:convolution}.

\begin{lem}\label{L:abelian1}
Let \(A\) be a piecewise pro-definable Abelian group and \(p\) a
symmetric definable type concentrating on \(A\). Assume that \(A\) has
\(p\)-weight strictly smaller than \(2n\). Then there exists a
pro-definable connected \(H\leq A\) with generic \(p^{\pm 2n}\) such
that \(p\) concentrates on a coset of \(H\).
\end{lem}

\begin{proof}
  Let $(a_1,a_2,\ldots,a_{2n}) \models p^{\tensor {2n}}|C$, where
  \(C\) is such that \(A\) is pro-\(C\)-definable and \(p\) is
  \(C\)-definable. Let \(b = a_1^{-1}\cdot a_2 \cdot \ldots \cdot
  a_{2n}\models p^{\pm 2n}|C\). By the weight assumption, $ {a_i}
  \models p|Cb$ for some $i$.  Say $i$ is odd. Since \(p\) is
  symmetric and \(A\) is Abelian, $\tp(a_1,a_2,\ldots,a_{2n}/b)$
  is $\Sym(n)$-invariant (for the action by permutation on the pairs
  \(a_{2i-1},a_{2i}\)), so \({a_1} \models p|Cb\) and hence \(b\models
  p^{\pm 2n}|Ca_1\). Since \(a_1\cdot b\models\transl{a_1}{p^{\pm
      2n}}|Ca_1\) and \(a_1\cdot b = a_2\cdot\ldots a_{2n}\models p^{\pm 2n
    -1}|Ca_1\), it follows that \(a_1\in X:= \Stab(p^{\pm 2n},p^{\pm
    2n-1})\). If \(i\) is even, then \(a_{2n}\models p|Cb^{-1}\) and
  hence \(a_{2n}\in X\).

  In both cases, \(p\) concentrates on \(X\), a coset of \(H :=
  \Stab(p^{\pm 2n})\). It follows that \(p^{\pm 2}\) and therefore
  \(p^{\pm 2n}\), concentrate on \(H = \Stab(p^{\pm 2n})\). So
  \(p^{\pm 2n}\) is the generic of \(H\).
\end{proof}

\begin{rem}
If \(p\) is stably dominated, so is \(p^{\pm 2n}\) and \(H\) is
connected stably dominated.
\end{rem}

For the rest of this section, we assume \(T\) metastable.

\begin{prop}
\label{P:sums}
Let \(C\) be a metastability basis, let $A$ be a pro-$C$-definable Abelian group
of bounded weight. Let $(B_i)_{i\in I}$ be connected stably dominated
pro-$C$-definable subgroups of $A$. Then there exists a stably dominated
connected pro-definable subgroup $B$ containing all the $B_i$.
\end{prop}

\begin{proof}
Let $\fF$ be the collection of all $C$-definable functions on \(A\) into
$\Gamma$, seen as a pro-definable function. Then, for any $c \in A$: $\tp(c/
C,\fF(c))$ is stably dominated.

Let $p_i$ be the principal generic of $B_i$. Consider the partial
type: \[q_0 = \{\defsc{p_i}{y} f(x) = f(y\cdot x ) : i \in I, f \in
  \fF \}\]

\begin{claim} $q_0$ is consistent.
\end{claim}

\begin{proof}
It suffices to show that any finite number of formulas, concerning
the types \((p_i)_{i\in I_0}\) for some finite \(I_0\subseteq I\), can
be satisfied. Let $B = \sum_{i\in I_0} B_i$. Note that \(p := \bigoast_{i\in I_0} p_i\) is the unique stably dominated generic of $B$.

If $(a,c) \models p_i  \tensor p|C$, then by genericity of $p$,
$a\cdot c \models p |Ca$. By symmetry,
$a \models p_i|Cc$. Now, for any $f \in \fF$, as $p$ is stably
dominated, and thus orthogonal to $\Gamma$, there exists $\gamma_f$
such that for any $c'\models p|C$, $f(c')=\gamma_f$. In particular,
$f(c)=\gamma_f = f(a\cdot c)$.  Thus
$\models \defsc{p_i}{y} f(c) = f(y\cdot c)$.
Since this is true for each $i\in I_0$ and $ f \in \fF$, $q_0$ is
consistent.
\end{proof}

Let $c \models q_0$, $C' = C \cup \fF(c)$ and \(C''=\acl(C')\). Recall
that \(\tp(c/C')\) is stably dominated. Let \(p\) be the unique
\(C''\)-definable extension of \(\tp(c/C'')\).

\begin{claim}
For all \(a\models p_i\), \(\transl{a}{p}|C' = p|C'\).
\end{claim}

\begin{proof}
  Let \(a\models p_i|\acl(Cc)\) and $X\subseteq A$ be some $C$-definable subset. Let $f_X \in \fF$ be the function sending $x$ to $0\in\Gamma$ if $x\in X$ and $x$ to $\infty\in\Gamma$ if $x\nin X$. Since $c\models q_0$, $f_X(a\cdot c) = f_X(c)$. It follows that  $\tp(a\cdot c/C) = \tp(c/C)$ and thus $\tp(a\cdot c, \fF(a\cdot c)/C) = \tp(c, \fF(c)/C)$. But,
  for all \(f\in\fF\), $f(a\cdot c)=f(c)$, so $\tp(a\cdot c, \fF(c) /C
  ) = \tp(c,\fF(c)/C)$. Since, by symmetry, \(c\models p|\acl(C'a)\),
  this shows that $p|C' = \tp(c/C')=\tp(a\cdot c/C') =
  \transl{a}{p}|C'$.
\end{proof}

Note also that since \(p\) and \(p_i\) are invariant over \(C''\), so
is \(\transl{a}{p}\). It is therefore one of the boundedly many
\(C''\)-definable extensions of \(\tp(c/C')\).

Moreover, by Lemma\ \ref{L:abelian1}, there exist a stably dominated
pro-definable group \(B\) with generic \(p^{\pm 2n}\), for some \(n\),
such that \(p\) concentrates on a coset of \(B\). Since the orbit of
\(p^{\pm 2n}\) under left translation by \(B_i B\) is bounded, it
follows that \(B_i/(B\cap B_i) \simeq (B_i B)/B\) is bounded and, as
\(B_i\) is connected, \(B_i\leq B\).
\end{proof}

In the previous proof, the subgroup \(B\) we constructed is,
\emph{a priori}, pro-\(C''\)-definable, but we can assume it is
pro-\(C'\)-definable:

\begin{lem}\label{L:sums alg}
Let \(A\) be a pro-limit of \(C\)-definable Abelian groups and \(B\leq
A\) be a stably dominated connected pro-\(\acl(C)\)-definable
group. Then, there exists a stably dominated pro-\(C\)-definable
\(B'\leq A\) containing \(B\).
\end{lem}

\begin{proof}
Let \(p\) be the principal generic of \(B\) and \((p_i)_{i\in I}\) be
its conjugates over \(C\). For all finite sets \(\Delta\) of formulas
(with parameters in \(C\)) preserved by left translation by \(A\), the set of codes for the defining schemes of the types $\restr{p_i}{\Delta}$, for $i\in I$, is equal to the set of $C$-conjugates of $\restr{p}{\Delta}$. It is therefore finite. Let $I_\Delta\subseteq I$ be a finite set such that for any $i\in I$, there exists $j\in I_{\Delta}$ with $\restr{p_i}{\Delta} = \restr{p_j}{\Delta}$. Let \(q_{\Delta} =
\bigoast_{i\in I_{\Delta}} p_i\).  Note that, since \(A\) is a
pro-limit of definable groups, any formula is contained in such a
\(\Delta\).

\begin{claim}
For any \(i_0\in I\), \(p_{i_0}\) concentrates on
\(\Stab_{\Delta}(q_{\Delta})\).
\end{claim}

\begin{proof}
Let \(c\models p_{i_0}|M\) for some model \(M\) containing \(C\) and
\((a_i)_{i}\models (\tensor_{i\in I_{\Delta}} p_i) | M c\).
Also, pick \(i_{1}\in I_{\Delta}\) such that \(\restr{p_{i_0}}{\Delta}
= \restr{p_{i_1}}{\Delta}\).  For every
\(\phi(x,y)\in\Delta\) and \(m\in M\), we have that \(\phi(x,m)\in
q_{\Delta}\) if and only if \(\models\phi(\sum_i a_i,m)\), which is
equivalent to \(\phi(\sum_{i\neq i_1}a_i + x,m)\in p_{i_1}\), by
symmetry of stably dominated types. Recall that $\Delta$ is preserved by left translation by elements of $A$, so that last formula is also an instance of $\Delta$ and hence \(\phi(\sum_{i\neq i_1}a_i + x,m)\in p_{i_1}\) if and only if \(\phi(\sum_{i\neq i_1}a_i + x,m)\in p_{i_0}\). Since \(p_{i_0}\)
concentrates of \(\Stab_{\Delta}(p_{i_0})\), this is equivalent to
\(\phi(\sum_{i\neq i_1}a_i + c + x,m)\in p_{i_0}\) and hence to
\(\models\phi(c + \sum_i a_i,m)\),
i.e. \(\phi(x,m)\in\transl{c}{q_{\Delta}}\).
\end{proof}

It immediately follows that \(q_{\Delta}\) concentrates on
\(\Stab_{\Delta}(q_{\Delta})\) and that, for any \(J\supseteq
I_{\Delta}\), \(\restr{(\bigoast_{i\in J} p_i)}{\Delta} =
\restr{q_{\Delta}}{\Delta}\).  Let \(q\) be the unique type such that
for all \(\Delta\), \(\restr{q}{\Delta} =
\restr{q_{\Delta}}{\Delta}\).  For all \(i\in I\), \(p_i\)
concentrates on \(\Stab_{\Delta}(q) = \Stab_{\Delta}(q_{\Delta})\) and
hence so does \(q_{\Delta}\).

Note that for all finite set \(\Delta\), \(\Stab_{\Delta}(q)\) is
relatively \(C\)-definable and thus \(B' = \Stab(q)\) is
pro-\(C\)-definable. There remains to show that \(q\) concentrates on
\(B'\).  For all \(\Delta\), let \(\Delta'\supseteq \Delta\) be a
finite set of formulas preserved by left translation by \(A\) such
that \(\Stab_{\Delta}(q)\) is defined by an instance of \(\Delta'\).
Since \(q_{\Delta'}\) concentrates on
\(\Stab_{\Delta'}(q_{\Delta'})\leq \Stab_{\Delta}(q_{\Delta'}) =
\Stab_{\Delta}(q)\), so does \(q\).
\end{proof}

%erratum, change def of FD to contain metastability model?
%do we have subadditivity of MR with our hypotheses.

\begin{lem}\label{L:ab-inertial} \AFD
Let \(C\) be a metastability basis which is a model and $A$ be a pro-limit of
$C$-definable Abelian groups with $\dimst(A) = n<\infty$. Then $A$
contains a stably dominated pro-$C\Gamma$-definable subgroup $S$
with stable homomorphic image of dimension $n$.
\end{lem}

\begin{proof}
Since \(C\) is a model and Morley rank is definable, we can find a $g \in A$
such that $\dimst(\St_C(g)/C)=n$. Let \(C' = \Gamma_{C}(g)\), \(C'' = \acl(C')\)
and $p = \tp(g/C'')$. So $p$ is stably dominated. Moreover, for all \(B\supseteq
C''\) and \(a\models p|B\), \(\dimst(\St_{B}(a)/B) = n\).

\begin{claim}
For all \(C\), \(a\), \(b\), \(\dimst(\St_C(a,b)/C) \leq
\dimst(\St_{Ca}(b)/Ca)+\dimst(\St_{C}(a)/C)\).
\end{claim}

\begin{proof}
Since Morley rank is definable, it is subadditive. Note that, since Morley rank
is definable, it is subadditive. Indeed, it suffices to check that, by
induction, for every definable function \(f : X\to Y\) whose fibers have
constant Morley rank \(n\), \(\RM(X) = \RM(f(Y)) + n\). Assume not, then there
exist \((S_i)_{i\in\omega}\subseteq X\) definable pairwise disjoint with
\(\RM(S_i)\geq \RM(f(Y)) + n\). We may assume that, for all \(i\), the Morley
rank of the fibers of \(f\) restricted to \(S_i\) is constant. If, for some
\(i\), \(\RM(f(S_i)) < \RM(f(Y))\), then, by induction, \(\RM(S_i) < \RM(f(Y)) +
n\), a contradiction. So, for all \(i\), we have \(\RM(f(S_i)) = \RM(f(Y))\) and
hence we may assume that there is some \(a\in \bigcap_i f(S_i)\neq\emptyset\).
Then, the fibers \(S_{i,a} = S_i\cap f^{-1}(a)\) are all disjoint and
\(\RM(S_{i,a}) < n\). By induction, it follows that \(\RM(S_i) < \RM(f(Y)) +
n\), a contradiction once again.

By subadditivity, we have
\begin{align*}
\dimst(\St_C(a,b)/C) &\leq \dimst(\St_C(a,b)/\St_C(a)) + \dimst(\St_C(a)/C)\\
&=\dimst(\St_C(a,b)/Ca) + \dimst(\St_C(a)/C)\\
&\leq \dimst(\St_{Ca}(b)/Ca)+ \dimst(\St_C(a)/C).\qedhere
\end{align*}
\end{proof}

\begin{claim}\label{C:rk n}
  If \(p\) and \(q\) are stably dominated over \(C\) concentrating on
  \(A\), $a \models p |C$, $b \models q | Ca$, $c = a\cdot b$ and
  \(\dimst(\St_{C}(a)/C) = \dimst(\St_{C}(b/C)) = n\), then
  \(\dimst(\St_C(c)/C) = n\).
\end{claim}

\begin{proof}
Since $\tp(a,b/C)$ is stably dominated, so is $\tp(c/C)$. Since
\(\St_C(a)\ind_C\St_C(b)\), \[\dimst(\St_C(a,b)/C) \geq
\dimst(\St_C(a)\St_C(b)/C) = 2n.\] Moreover, \(\dimst(\St_C(a,b)/C) =
\dimst(\St_C(b,c)/C) \leq \dimst(\St_{Cc}(b)/Cc) + \dimst(\St_C(c)/C) = n
+\dimst(\St_C(c)/C)\).  So $\dimst(\St_C(c)/C) \geq n$.  But $\dimst(A)=n$,
so $\dimst(\St_C(c)/C) =n$.
\end{proof}

By Lemma\ \ref{L:abelian1}, $A$ contains a pro-$C''$-definable stably
dominated group $R$, with generic type $p^{\pm 2m}$ for some \(m\). By
Claim\ \ref{C:rk n}, for any \(a\models p^{\pm 2m} | C''\),
\(\dimst(\St_{C''}(a)/C'') = n\). By Lemma\ \ref{L:sums alg}, we find
\(S\geq R\) pro-\(C'\)-definable with stably dominated type \(q\)---we
can check that, by construction and Claim\ \ref{C:rk n}, for all
\(a\models q|C'\), \(\dimst(\St_{C'}(a)/C') = n\). We can now conclude
with Proposition\ \ref{P:groupdom}.
\end{proof}

\subsection{Limit stably dominated groups}

\begin{defn}\label{D:lmstd}
Let \(G\) be a pro-\(C\)-definable group and \(q\) be a (potentially
infinitary) type over \(C\). For all \(t\models q\), let
\(S_t\) be a pro-\(Ct\)-definable subgroup of \(G\) (uniformly in
\(t\)). We call $(S_t)_{t\models q}$ a \emph{limit stably dominated family}
for $G$ if:
\begin{enumerate}
\item $S_t$ is a connected stably dominated subgroup of $G$.
\item If $W \leq G$ is connected and stably dominated,
  pro-$C$-definable, then $W \leq S_t$ for some $t\models q$.
\item The family \((S_t)_{t\models q}\) is directed: any small set of
  realizations of \(q\) has an upper bound in the order defined by
  \(t_1 \leq t_2\) if \(S_{t_1}\leq S_{t_2}\).
\end{enumerate} 
The group \(H = \bigcup_{t\models q} S_t\) is called the \emph{limit stably
dominated subgroup} of \(G\). If \(G = H\), we say that \(G\) is limit
stably dominated.
\end{defn}

We view \(H\) as the limit of a directed system of stably dominated groups. It
is clearly independent of the particular choice of limit stably dominated family
\((S_t)_{t\models q}\).

\begin{lem}\label{L:lstdom pro-def}
Let \(G\) be a pro-definable group and \(H\leq G\) be its limit stably
dominated subgroup, if it exists. Then \(H\) is pro-definable.
\end{lem}

\begin{proof}
Since the \(S_t\) are pro-definable, there exists a directed family of
relatively definable sets \(S^i_t\subseteq G\) such that \(S_t = \bigcap_i S^i_t\). By
compactness, we have that $a\in H$ if and only if, for all $i$ and all
finite $q_0\subseteq q$, there exists $t$ such that $t\models q_0$ and
$a\in S_t^i$. So $H$ is an intersection of relatively definable
subsets of $G$, i.e. a pro-definable subset.
\end{proof}

\begin{lem}
Let \(G\) be a definable group and \((S_t)_{t\models q}\) a limit
stably dominated family for \(G\). Assume that \(S_t\) is the
connected component of some definable group \(R_t\), then the order
\(t_1 \leq t_2\) is relatively definable.
\end{lem}

\begin{proof}
If \(S_{t_1}\leq R_{t_2}\), then \(S_{t_1}\leq S_{t_2}\). It follows
that \(t_1\leq t_2\) if and only if \(S_{t_1} \leq R_{t_2}\),
equivalently \(\defsc{p_{t_1}}{x} x\in R_{t_2}\) where, \(p_{t}\)
denotes the principal generic of \(S_{t}\).
\end{proof}

Two different behaviors are possible, according to whether
or not the direct limit system has a maximal element.  The latter
is equivalent to the existence of a maximal connected stably dominated
subgroup of $G$.

\begin{prop}
\label{P:exists-lstdom}
Let \(A\) be a pro-limit of \(C\)-definable Abelian groups. Assume \(A\) has
bounded weight. Then, a limit stably dominated family for \(A\) exists over
\(C\).
\end{prop}

\begin{proof}
Let $C_1$ be a sufficiently saturated model containing \(C\) and let $C_2
\supseteq C_1$ be a metastability basis which is a model. Let $(A_i)_{i \in I}$
be the family of all connected stably dominated pro-\(C_1\)-definable subgroups
of $A$. Let \(S\) be as in Proposition\ \ref{P:sums}. By Lemma\ \ref{L:defbdd},
\(S = S_t\) is pro-\(Ct\)-definable for some small tuple \(t\). Let
$q=\tp(t/C)$.

If $W\leq A$ is a connected stably dominated group, we must show that $W \leq
S_{t'}$, for some $t' \models q$. For this purpose we can replace $W$ by a
conjugate, under the group of automorphisms of the universal domain over $C$.
Thus, we may assume $W$ is defined over $C_1$.  In this case, $W \leq S_t$. This
proves (2) of the definition of a limit stably dominated family. Directness of
the family \((S_t)_{t\models q}\) follows from Proposition\ \ref{P:sums},
together with (2).
\end{proof}

Note that if \(C = C_2\) can be assumed to be a sufficiently saturated
metastability basis, then we can choose \(q\) to be a type concentrating on some
power of \(\Gamma\).

\begin{lem}
\label{L:4def}
Let $C$ be a metastability basis which is a model. Let $A$ be a $C$-definable
Abelian group and let \(H\) be a connected stably dominated
\(\infty\)-$C$-definable subgroup of \(A\).
\begin{itemize}
\item[\AFD] \(H\) is contained in a $C\Gamma$-definable stably dominated
  subgroup.
\item[\Aom] \(H\) is contained in a $C\Gamma$-definable connected stably
  dominated subgroup.
\end{itemize}
\end{lem}

\begin{proof}
  Assume \AFD. Since \(H\) is the stabilizer of its principal generic,
  \(H = \bigcap_i H_i\) for some directed system of $C$-definable subgroups \(H_i\) of
  \(A\). Then $\dimst(A/H_i)$ is non-decreasing with $i$, and
  eventually stabilizes; we may assume it is constant, $\dimst(A/H_i)
  = n$. Clearly $\dimst(A/H) = n$. By Lemma\ \ref{L:ab-inertial},
  $A/H$ contains a stably dominated pro-$C\Gamma$-definable subgroup
  $S$ with stable homomorphic image of dimension $n$. Let \(S_i\) be
  the image of \(S\) in \(A/H_i\). Then \(S_i\) is stably dominated
  and \(S = \projlim_i S_i\). The stable homomorphic image of
  dimension $n$ of \(S\) factorizes through \(S_i\) for large enough
  \(i\), so we may assume that, for all \(i\), \(S_i\) has stable
  homomorphic image of dimension \(n = \dimst(A/H_i)\). By
  Corollary\ \ref{C:maxrkdef}, there exists a $C\Gamma$-definable stably dominated
  subgroup $R_i\leq A$ such that $S_i = (R_i/H_i)^0$. Note that for any
  \(j\geq i\), \(R_i/H_j\) has a stable homomorphic image of dimension
  $n = \dimst(A/H_j) = \dimst(R_i/H_j)$. So replacing $A$ by any of
  the $R_i$, we may assume that for all $i$, $A/H_i$ has a stable
  homomorphic image of dimension $n = \dimst(A/H_i)$.

  By Proposition\ \ref{P:a2def} there exists a unique
  $\infty$-definable group $T_i$ with $H_i \leq T_i \leq A$, $A/T_i$
  is almost $\Gamma$-internal and $T_i/H_i$ is connected stably
  dominated. Note that, since it is unique, $T_i$ is
  $\infty$-$C\Gamma$-definable. By Lemma\ \ref{L:a3def}, there exists
  $W_i\leq A$, $C\Gamma$-definable, such that $T_i/H_i = (W_i/H_i)^0$.

Let $i\leq j$. Since $T_j/H_j$ is connected stably dominated, so is is
homomorphic image \(T_j /(T_i\cap T_j)\) and, since $A/ T_i$ is almost
$\Gamma$-internal, so is its subgroup \(T_iT_j/T_i \simeq T_j /(T_i\cap
T_j)\). It is therefore trivial and $T_j \leq T_i$. Since $A/T_j$ is (almost)
$\Gamma$-internal, the argument of Lemma\ \ref{L:a3def} shows that for large
enough $i_0$, for all $i \geq i_0$, \([T_{i_0}:T_i]\) is bounded. Let \(T =
\bigcap_{i} T_i\). Then \([T_{i_0}:T]\) is bounded. Moreover, \(T/T\cap H_i\)
is (isomorphic to) a bounded index subgroup of \(T_i/H_i\), it is therefore
stably dominated. It follows that \(T/H = \projlim_i T/T\cap H_i\) is also
stably dominated. By Lemma\ \ref{L:stdomext}, so is \(T\) and hence so are the
\(T_i\). Since $[W_i:T_i]$ is bounded, any of the $W_i$ is $C\Gamma$-definable
and stably dominated.

If \Aom holds, since $H\leq T^0$, the second statement follows from
Corollary\ \ref{C:def G0}.
\end{proof}

\begin{defn}
A definable family $H_t$ of subgroups of some definable group $G$ is
said to be \emph{certifiably stably dominated} if for all \(t\), $H_t$
is stably dominated and there exists \(t'\) such that \(H_t\leq
H_{t'}\) and $H_{t'}$ is connected.
\end{defn}

\begin{prop}
\label{P:cert}
\Aom Let $A$ be a \(C\)-definable Abelian group. Then there exist
\(C\)-definable families $H^\nu$ of definable subgroups $H^\nu_t$ of $A$, such
that:
\begin{enumerate}
\item Any $H^\nu_t$ is stably dominated.
\item Any connected stably dominated $\infty$-definable subgroup of
  $A$ (over any set of parameters) is contained in some $H^\nu_t$.
\end{enumerate}
\end{prop}

\begin{proof}
By Lemma\ \ref{L:4def}, to prove (2), it suffices to consider definable
connected stably dominated subgroups of \(A\).

For a definable Abelian group $B$, define invariants $n$ and $k$ as follows.
First, set $n = \dimst(B)$.  Let $Z(B)$ be the collection of definable subgroups
$S\leq B$ with stable homomorphic images of dimension $n$; by Lemma\
\ref{L:ab-inertial} and Corollary\ \ref{C:maxrkdef}, $Z(B) \neq \emptyset$.  Let
$Z_2(B) = \{(S,T): S \in Z(B), T \leq S\text{ $C\Gamma$-definable}, S/T\text{
$\Gamma$-internal}\}$. Let $k = \max\{\dimo(S/T): (S,T) \in Z_2(B)\}$; by \AFD,
such a maximum exists. The pairs $(n,k)$ are ordered lexicographically.

Pick any definable connected stably dominated definable $B\leq A$. If
$(S/B,T/B)$ attains the maximum for $A/B$, then the pullbacks to $A$ show that
$(n,k)(A) \geq (n,k)(A/B)$. Thus increasing $B$ has the effect of decreasing
$(n,k)(A/B)$. Let $B_0$ be such that $(n,k)(A/B_0)$ is minimal. For any stably
dominated $H\leq A$, $H+B_0$ is also stably dominated and if $S_t$ is a family
of stably dominated subsets of $A/B_0$, its lifting to $A$ is a family of stably
dominated subsets of $A$; so it suffices to find families $\{S^\nu_t\}$ for
$A/B_0$. Thus, we may assume $(n,k)(A/B) = (n,k)(A)$ for any connected stably
dominated $C\Gamma$-definable $B \leq A$. Let $(n,k) = (n,k)(A)$.

\begin{claim}
\label{C:maxnk stdom}
Let $S$ be a definable subgroup of $A$ admitting a definable surjective
homomorphism $\theta: S \to \fg$ to a definable stable group of Morley rank $n$
and a definable surjective homomorphism $\xi: S \to W$ to a definable
$\Gamma$-internal group of \(o\)-minimal dimension \(k\). Then there exists a
\(C\)-definable family $H_t$ of stably dominated subgroups of $A$ such that
$\ker(\xi) = H_t$ for some $t$.
\end{claim}

\begin{proof}
Clearly $S,\xi,\fg,\theta,W$ lie in a \(C\)-definable family
$(S_t,\xi_t,\fg_t,\theta_t,W_t)$ such that \(S_t\leq A\), $\fg_t$ is stable,
$\dimst(\fg_t) = n$, $W_t$ is a $\Gamma$-internal, \(\dimo(W_t) = k\),
$\theta_t: S_t \to \fg$ is a surjective homomorphism, $\xi_t: S_t \to W_t$ is a
surjective homomorphism. Let $H_t = \ker(\xi_t)$. Let $T_t\leq S_t$ be as in
Proposition\ \ref{P:a2def}. By Lemma\ \ref{L:a3def}, $T_t$ is $Ct$-definable. By
Lemma\ \ref{L:almost-int-2}, $S_t/T_t$ is $\Gamma$-internal. The group
$T_t/H_t\cap T_t$ is $\Gamma$-internal and stably dominated, it is therefore
finite and hence trivial, since $T_t$ is connected. So $T_t\leq H_t$. By
maximality of $k$, $H_t/T_t$ is $\Gamma$-internal of $o$-minimal dimension $0$,
so it is finite. It follows that $H_t$ is stably dominated.
\end{proof}

Let $B$ be any definable stably dominated connected subgroup of $A$. Let
$(S,T)\in Z_2(A/B)$ be such that $S$ has a stable homomorphic image of dimension
$n$ and $S/T$ is $\Gamma$-internal of $o$-minimal dimension $k$. Let $(S',T')$
be the pullbacks of $(S,T)$ to $A$. Then $B\leq T'$. Clearly $S'$ admits a
stable homomorphic image of dimension $n$ and $S'/T'$ is $\Gamma$-internal of
$o$-minimal dimension $k$. By the Claim\ \ref{C:maxnk stdom}, there exists a
$C$-definable family $H_t$ of stably dominated groups such that $T' = H_t$ for
some $t$.
\end{proof}

\begin{cor}
\label{C:cert-2}
\Aom Let $A$ be a \(C\)-definable Abelian group. There exists a definable
certifiably stably dominated family $H_t$ such that any connected stably
dominated $\infty$-definable subgroup of \(A\) is contained in some \(H_t\).
\end{cor}

\begin{proof}
Let $H_t^{\nu}$ be as in Proposition\ \ref{P:cert} and let \(p'\) and \(S_{t'}\)
be as in Proposition~\ref{P:exists-lstdom}. Let \(a \models p'\), then, for some
\(\nu\) and \(t\), we have \(S_a \leq H_t^\nu\). Let \(p = \tp(t/C)\). Then, for
\(t'\models p'\), for some \(t\models p\), \(S_{t'} \leq H_{t}^\nu\). In
particular, any \(\infty\)-definable connected stably dominated subgroup of $A$
is contained in some \(H_t^{\nu}\), with $t\models p$.

For any $t\models p$, $H_{t}^{\nu,0}$ is definable and has finite index, say
$l$, in $H_{t}^{\nu}$. Let $L_t$ be a $C$-definable family of subgroups of
$H_t^\nu$ such that $[H_t^\nu:L_t]\leq l$ and whenever $t\models p$, $L_t =
H_t^{\nu,0}$. Then every $L_t$ is stably dominated and any \(\infty\)-definable
connected stably dominated subgroup of $A$ is contained in some $L_t$ with
$t\models p$.

Finally, let \(Q\) be the set of \(t\) such that for every \(s\), if
\(L_t/L_s\cap L_t\) is finite then it is trivial. This set is definable by
Lemma~\ref{L:nfcp}. If \(t\models p\), then \(L_t\) is connected and hence \(t
\in Q\). Moreover, for any \(s\in Q\), \(L_s^0 \leq L_t\) for some \(t\models
p\). But then \(L_s^0\leq L_s\cap L_t\), so \(L_s / L_s\cap L_t\) is finite and
hence trivial, by definition of \(Q\). It follows that \(L_s \leq L_t\),
concluding the proof.
\end{proof}

Before we can state the next theorem, we need to give an explicit
definition of $\Gamma$-internality for certain hyperimaginaries.

\begin{defn}
 Let $D$ be a pro-definable set and $E$ be a pro-definable equivalence
 relation on $D$. We say that $D/E$ is \emph{almost $\Gamma$-internal} if there
 exists a a pro-definable map $f : D \to X$ such that:
 \begin{enumerate}
 \item $X$ is a pro-definable set in \(\eq{\Gamma}\).
 \item Each fiber of $f$ intersects only boundedly many classes of $E$.
 \end{enumerate}
\end{defn}

\begin{rem}\label{R:int hyper}
Note that if $E$ is the intersection of relatively definable
equivalence relations on $D$, then $D/E$ is a pro-definable set and
the above definition is equivalent to the usual notion of almost
$\Gamma$-internality.

In particular, if $G$ is a pro-definable group and $H$ is a pro-definable
subgroup admitting a generic and \(G/H\) is almost \(\Gamma\)-internal, then
$G/H$ is a pro-limit of almost \(\Gamma\)-internal definable groups. If,
moreover, \(\Gamma\) is \(o\)-minimal, then, by Lemma\ \ref{L:almost-int-2},
each of these almost \(\Gamma\)-internal definable groups is in fact
\(\Gamma\)-internal and is therefore (isomorphic to) a group definable in
$\Gamma$. It follows that $G/H$ is (pro-definably isomorphic to) a group
pro-definable in $\Gamma$.
\end{rem}

\begin{thm}
\label{T:lmstd}
Let \(A\) be a pro-limit of interpretable Abelian groups. Assume \(A\)
has bounded weight. Then the limit stably dominated subgroup $H$
exists and $A/H$ is almost internal to $\Gamma$.

If \Aom holds and $A$ is interpretable, then \(A/H\) is \(\Gamma\)-internal and
\(H\) is definable, it admits a generic type \(p\) and it is connected.
\end{thm}

\begin{proof}
  The existence of $H$ is proved in
  Proposition\ \ref{P:exists-lstdom}. Let \(C\) be a metastability
  basis over which \(A\) and \(H\) are defined.  Let \(f\) be a
  pro-definable function enumerating all \(C\)-definable maps from
  \(A\) into \(\Gamma\).

  \begin{claim}
    For all \(c, d\in A\), if \(\tp(c/\acl(Cf(c))) =
    \tp(d/\acl(Cf(c)))\), then \(c\cdot d^{-1} \in H\).
  \end{claim}

  \begin{proof}
    By definition of a metastability basis, $\tp(c/Cf(c))$ is stably
    dominated. Let $p$ be the unique $\acl(Cf(c))$-definable extension
    of \(\tp(c/\acl(Cf(c)))\). By Lemma\ \ref{L:abelian1}, there
    exists a stably dominated pro-\(\acl(Cf(c))\)-definable connected
    stably dominated subgroup \(S\leq A\) such that \(p\) concentrates
    on a coset of \(S\). Note that, by definition of $H$, $S\leq H$. Since \(p\) is \(\acl(Cf(c))\)-definable, it
    follows that \(p|\acl(Cf(c))\) concentrates on a coset of \(S\);
    i.e. \(c\cdot d^{-1}\in S\leq H\).
  \end{proof}

  It immediately follows that each fiber of \(f\) intersects boundedly
  many cosets of \(H\) and, since the image of \(f\) is a
  pro-definable set in \(\Gamma\), \(A/H\) is almost
  \(\Gamma\)-internal.

  Let us now assume \Aom. By Corollary\ \ref{C:cert-2} we can find a certifiably
  stably dominated family \((H_{t})_{t\in Q}\) such that any
  \(\infty\)-definable connected stably dominated subgroup of \(A\) is contained
  in some connected \(H_{t'}\). Then \(H = \bigcup_{t\in Q} H_{t}\) and hence it
  is definable.

  Define a partial order on \(Q\) by \(t_1 \leq t_2\) if \(H_{t_1}\leq
  H_{t_2}\). Note that this order is directed. By \cref{cofinal-def}, \(Q\) has
  a cofinal definable type \(q\). For any \(t\models q\), let \(p_{t}\) be the
  principal generic of \(H_{t}^{0}\). Let \(r\) be the type such that, for all
  $D\supseteq C$ over which $q$ and \(H\) are defined, \(a\models r|D\) if and
  only if \(a\models p_{t}|Dt\) for some \(t\models q|D\). By Lemma\
  \ref{L:concat}, \(r\) is definable. Pick any \(g\in H\). Then there exists
  \(t\models \restr{q}{Dg}\) such that \(g\in H^0_t\). Let \(a\models p_{t}|D g
  t\), then \(g\cdot a \models p_{t}|D g t\) and hence \(g\cdot a\models r|Dg\).
  It follows that \(\transl{g}{p} = p\) and thus that \(H = \Stab(p)\) is
  connected.

  The fact that \(A/H\) is \(\Gamma\)-internal now follows from the
  fact that it is almost \(\Gamma\)-internal and
  Lemma\ \ref{L:almost-int-2} (cf. Remark\ \ref{R:int hyper}).
\end{proof}

\begin{rem}
The only use of \Aom in the proof above is to apply Corollary\ \ref{C:cert-2}.
It follows that, to get the second part of the statement in Theorem\
\ref{T:lmstd}, it suffices to assume that there exists a certifiably stably
dominated family \(H_t\) such that any \(\infty\)-definable connected stably
dominated subgroup is contained in some \(H_t\).
\end{rem}

\begin{cor}\label{C:max-G-quotient}  \Aom
Let $A$ be a definable Abelian group.  There exists a universal pair
$(f,B)$ with $B$ a $\Gamma$-internal definable group, and $f: A \to B$
a definable homomorphism. In other words for any $(f',B')$ of this
kind there exists a unique definable homomorphism $h: B \to B'$ with
$f' = h\circ g$.

Equivalently, there is a smallest
definable subgroup $H$ with $A/H$ $\Gamma$-internal.
\end{cor}
  
\begin{proof}
Let $H$ be the limit stably dominated subgroup of $A$. Then $H$ is
definable.  For any pair $(f',B')$ as above, $f'$ vanishes on any
definable connected stably dominated group.  Since $H$ is a union of
such groups, $f'$ vanishes on $H$. But $A/H$ is $\Gamma$-internal, so
the canonical homomorphism $A \to A/H$ clearly solves the universal
problem.
\end{proof}

\begin{cor}\label{C:finite-cc}  \Aom
Let $A$ be a definable Abelian group. There exists a smallest
definable subgroup $A^0$ of $A$ of finite index.
\end{cor}

\begin{proof} Let $H$ be the limit stably dominated subgroup of $A$ and \(B\leq A\) be any subgroup of finite index. Since $H/H\cap B$ is finite and $H$ is connected, it follows that $H\leq B$. The question
  thus reduces to the $o$-minimal group $A/H$, where it is known
  (cf. \cite[Proposition\ 2.12]{Pil-OMinGp}).
\end{proof}

\begin{question}\label{Q:finite-cc}
In $\ACVF$, is $G/G^0$ always finite, where \(G\) is any definable group and \(G^0\) denotes the
intersection of all definable finite index subgroups of \(G\)?
\end{question}

The answer to this question is presumably yes (cf. Problem\ \ref{Pb:double} and the
proof of Corollary\ \ref{C:finite-cc}) and it is already known to be
true of stably dominated groups (cf. Corollary\ \ref{C:def G0}).

\begin{cor}\label{C:max-st-quotient} \Aom 
Let $A$ be a definable Abelian group. There exists a universal pair
$(f,B)$ with $B$ a stable definable group, and $f: A \to B$ a
definable homomorphism. Equivalently, there exists a smallest
definable subgroup $H$ with $A/H$ stable.
\end{cor}

\begin{proof}
Clearly if $A/H$ and $A/H'$ are stable, so is $A/(H \meet H')$. It
suffices to show that there are no strictly descending chains $A
\supset H_1 \supset H_2 \ldots$ of such subgroups.  By \AFD,
$\dimst(A/H_n)$ is bounded; so we may assume it is constant,
$\dimst(A/H_n)=d$.  Then $H_n/H_{n+1}$ is finite.  By
Corollary\ \ref{C:finite-cc}, the chain stabilizes.
\end{proof}

As one can see, Corollary\ \ref{C:max-st-quotient} follows formally from the existence of a smallest definable subgroup of finite index and does not actually depend on the existence of the limit metastable subgroup. Using the fact that in a metastable theory, extensions of $\Gamma$-internal groups are $\Gamma$-internal, Corollary\ \ref{C:max-G-quotient} can similarly be deduced formally from the existence of a smallest definable subgroup of finite index. In particular, a
positive answer to Question\ \ref{Q:finite-cc} would imply that
Corollaries\ \ref{C:max-G-quotient} and \ref{C:max-st-quotient} hold of every definable group in $\ACVF$.

\begin{rem}\label{R:max-st-quotient} \Aom
  Since a pro-definable stable group always has a generic and hence, by Proposition\ \ref{P:prodefgp} is
  a prolimit of groups, any \(\infty\)-definable normal subgroup of a
  definable group $A$ such that $A/H$ is stable is the intersection of
  definable subgroups $H_i$ such that $A/H_i$ is stable. It follows
  that if $A$ is Abelian, the smallest \(\infty\)-definable subgroup
  $H$ with $A/H$ stable exists and is definable.

  It is then straightforward to prove that any pro-limit $A$ of
  Abelian groups---in particular, any pro-definable Abelian group with
  a generic---has a smallest pro-definable subgroup $H$ such that
  $A/H$ is stable.
\end{rem}

\begin{cor}\label{C:endos} \Aom
Let $A$ be a connected definable Abelian group.  Then for almost all
primes $p$, $p \cdot A=A$.
\end{cor}

\begin{proof}
  Let $H$ be the limit stably dominated subgroup of $A$.  We have $H =
  \union_{t \models q} S_t$ where $q$ is a complete type over some
  \(C\) and $S_t$ is a definable connected stably dominated group. Let $A_t$
  be the maximal stable quotient of $S_t$ and \(m = \dimst(A_t)\). Let $A_t(p) = \{x \in A_t: p\cdot x=0 \}$.  Then
  $A_t(p)$ can be infinite for at most $m$ values of $p$.  For any
  other $p$, $A_t(p)$ is finite and hence $p\cdot A_t = A_t$.  It
  follows that $p\cdot S_t$ has finite index in $S_t$
  (Corollary\ \ref{C:finindex}); so $p\cdot S_t = S_t$.  Thus $pH=H$.
  
  On the other hand, $A/H$ is an
  $o$-minimally definable group. So its $p$-torsion is finite (eg. \cite[Proposition\ 6.1]{Str}) and hence $\dimo((A/H)/p(A/H)) = 0$. Since $A/H$ is connected, $p(A/H) = A/H$. So $p\cdot A=A$.
\end{proof}

\subsection{Metastable fields}

By a rng we mean an Abelian group with a commutative, associative,
distributive multiplication (but possibly without a multiplicative
unit element.)  If a rng has no zero-divisors, the usual construction
of a field of fractions makes sense.

\begin{prop}\label{P:fields} \AFD
  Let $F$ be a definable field with $0 < \dimst(F) = n < \infty$. Then there
  exists an $\infty$-definable subrng $D$ of $F$, with $(D,+)$
  connected stably dominated. Moreover $F$ is the field of
  fractions of $D$.
\end{prop}

\begin{proof}
Using Lemma\ \ref{L:ab-inertial}, we find a connected $\infty$-definable
subgroup $M$ of $F^\star$ with a stably dominated generic type $p$ of stable
dimension $n$.
 
Recall that the \(p\)-weight of \(F\) is bounded by \(\dimst(F)\). We can
therefore apply Lemma\ \ref{L:abelian1} additively to find $D$ a connected
subgroup of $(F,+)$ with principal generic type $p^{\pm 2n}$ and such that $p$
concentrates on an additive coset of $D$. As in Lemma\ \ref{L:ab-inertial}, we
see that $\dimst(D) \geq n$. Let $C$ be such that $F$, $M$, $D$ and $p$ are
$C$-definable. If $(a,b_1,b_2,\ldots,b_{2n})\models p^{\tensor 2n+1}|C$, then
$a\cdot(\sum_{i} -b_{2i}+b_{2i+1}) = \sum_{i} -(a\cdot b_{2i}) + (a\cdot
b_{2i+1}) \models p^{\pm 2n}|Ca$, i.e. if $a\models p|C$ and $b\models p^{\pm
2n}|Ca$, $a\cdot b = p^{\pm 2n}|Ca$. Let now $a\models p|C$ and $d\in D$. Since
$p^{\pm 2n}$ is the principal generic of $D$, we find $b,c\models p^{\pm 2n}|Ca$
such that $d = b-c$. So $a\cdot d = a\cdot b - a\cdot c \in -p^{\pm 2n}|C +
p^{\pm 2n}|C \subseteq D$. Similarly, since $M$ is multiplicatively generated by
$p|C$, $M\cdot D\subseteq D$. Finally, since $D$ is generated by $M-M$, $D\cdot
D \subseteq D$; so $D$ is a subrng of $F$.

  Let $F'\subseteq F$ be the fraction field of $D$. If $a\in F\sminus
  F'$, then the map $(x,y) \mapsto x+a\cdot y$ takes $D^2 \to F$
  injectively---if $x+a\cdot y=x'+a\cdot y'$ and $(x,y) \neq (x',y')$ then $y \neq
  y'$ and $a = (x'-x)/(y-y')$. Thus $D + a\cdot D\subseteq F$ is definably isomorphic
  to $D^2$. But $\dimst(D^2) \geq 2n > n = \dimst(F)$, a
  contradiction. We have proved that $F=F'$ is the field of fractions
  of $D$.
\end{proof}

\begin{rem}
  Assuming \Aom, in the previous proposition, if \(D\) is a ring,
  i.e. \(1\in D\), then the generic type of \((D,+)\) is also generic
  in \((D^\star,\cdot)\).
\end{rem}  

\begin{proof}
  Note first that by construction, the generic type $q$ of \((D,+)\)
  has stable dimension \(n = \dimst(F)\), so, by
  Corollary\ \ref{C:maxrkdef}, \(D\) is definable. Let \(N\) be the
  smallest definable subgroup of \((D,+)\) with \(D/N\) stable
  (cf. Corollary\ \ref{C:max-st-quotient}). Clearly $N$ is
  multiplicatively invariant under the units of $D$.  Thus $N$ is an
  ideal of $D$ and $D/N$ is a definable stable ring. Then $D/N$ has
  only finitely many maximal ideals which are all definable
  (cf. \cite[Theorem\ 2.6]{CheRei}). It follows that the generic of
  $D/N$ avoids all these maximal ideals and therefore concentrates on
  $(D/N)^\star$.

  Let $a \models q$. The image in $D/N$ of the ideal $a\cdot D$
  contains the generic $a+N$ which is invertible. So that image equals
  $D/N$. By Corollary\ \ref{C:finindex}, $a\cdot D = D$---recall that
  $D$ is connected---and $a$ is invertible. So the additive generic of
  $D$ concentrates on $D^\star$ and since, for any $c\in D^\star$, the
  map $x\mapsto c\cdot x$ is a definable isomorphism, it must preserve
  $q$.
\end{proof}

\begin{rem}
  If $F$ is a definable field of finite weight, then the limit stably
  dominated subgroup $H$ of $(F,+)$ is an ideal so it is either
  trivial---in which case $F$ is \(\Gamma\)-internal---or coincides
  with $F$. If $(F,+)$ is stably dominated then its minimal definable
  subgroup $N$ with $F/N$ stable is a proper ideal and hence is
  trivial---in which case $F$ is stable.

  So, unless $F$ is stable or $\Gamma$-internal, $(F,+)$ is properly
  limit stably dominated.
\end{rem}

\begin{prop}\label{P:unstable fields}\Aom
  Let $F$ be an infinite definable field which is neither stable nor
  internal to $\Gamma$. Then $F^\star$ has a definable homomorphic $\Gamma$-internal non definably compact image.
\end{prop}

\begin{proof}
  Let $H$ be a stably dominated $\infty$-definable subgroup of
  $F^\star$ with generic $p$. As in the proof of
  Proposition\ \ref{P:fields}, we find a subrng $R$, additively
  generated by $p^{\pm 2}$, invariant under multiplication by $H$,
  whose additive group is connected stably dominated and such that $p$
  concentrates on an additive coset of $R$. Recall that, since we built $p^{\pm 2}$ by looking at the additive structure, realizations of $p^{\pm 2}$ are of the form $-a + b$ for some $(a,b)\models p^{\tensor 2}$.
  
  If $H_1 \leq H_2$ are two $\infty$-definable stably dominated subgroups of $F^\star$,
  let $p_i$ be their generics and $R_i$ the corresponding subrng. Assume that everything is defined over some $C$. If $a,b\in H_1$ and
  $e\models p_2|Cab$, then both $a\cdot e$ and $b\cdot e$ realize
  $p_2|Cab$. It follows that $a - b = e^{-1}(a\cdot e - b\cdot e)\in
  R_2$. In particular, the set of realizations of $p_1^{\pm 2}$ is contained in $R_2$. Since $R_1$ is additively generated by $p_1^{\pm 2}$, $R_1 \subseteq
  R_2$.
  
  Let $M = \bigcup_{t\models q} M_t$ be the limit stably dominated
  subgroup of \(F^\star\) and \(R_t\) the subrng corresponding to $M_t$. Note
  that, since $M_t$ can be---and is---chosen definable and that $M_t -
  M_t$ generates \(R_t\) in boundedly many steps, \(R_t\) is definable. Note that if
  \(\dimst(F) = 0\), working over a metastability basis $C$, for
  any $a\in F$, $\tp(a/\Gamma_C(a))$ has a unique realization and $F$ is internal
  to $\Gamma$. So \(\dimst(F) > 0\). Let \(R'\) (and $M'$) be as in
  Proposition\ \ref{P:fields}. In particular, $F$ is the fraction field of $R'$. By definition, there exists $t\models q$ such that $M'\leq M_t$. By the previous paragraph, \(R'\subseteq R_t\). It follows that for any
  \(t\models q\), $F$ is the fraction field of $R_t$. Let $N_t$ be the
  minimal quotient of $(R_t,+)$ such that $R_t/N_t$ is stable
  (cf. Corollary\ \ref{C:max-st-quotient}). Since $N_t$ is
  characteristic, it is invariant under multiplication by elements of
  $M_t$. It is therefore an ideal of $R_t$. By
  Proposition\ \ref{P:groupdom}, the principal generic of $R_t$ is
  dominated by a group morphism, so if \(N_t = R_t\), \(R_t =
  \{0\}\). On the other hand, if \(N_t = (0)\), $R_t$, and hence $F$
  is stable. It follows that \(N_t\) is a proper non-trivial ideal. If
  \(t\leq s\), \(R_t/N_s\cap R_t\) is stable as it embeds in $R_s/N_s$. So \(N_t\subseteq
  N_s\). If $a\in N_t$ is not \(0\), then $a^{-1}\nin R_s$ for any
  $s>t$ and it follows that $R := \bigcup_t R_t \neq F$.
  
  Let $V := F^\star/M$. It is $\Gamma$-internal. Let \(V^+ :=
  (R_t\sminus\{0\})/M\) for some choice of \(t\). It is a definable
  subsemigroup of \(V\). Since \(R\) is stabilized by $M$, \(M
  R_t\subseteq R \neq F\), so \(V^+\neq V\). Also, since \(F\) is the
  fraction field of \(R_t\), \(V = V^+ - V^+\). By
  Lemma\ \ref{L:semigroup}, $V$ does not have (NG). By Proposition\ \ref{P:defcomp NG}, it is not
  definably compact.
\end{proof}

\section{Valued fields: stably dominated groups and algebraic groups}

In this section, we work in $\ACVF$. Let $K$ be an algebraically
closed valued field; $\Oo$ denotes the valuation ring. Occasionally we
will assume $K$ to be sufficiently saturated.

Recall that a definable set is purely imaginary if there are no definable functions (with parameters)  from that set onto an infinite subset of \(K\), cf. Definition\ \ref{D:pure im}.

\begin{prop}\label{P:embed}
  Let \(C = \acl(C)\) and $H$ be a connected $\infty$-\(C\)-definable
  group with a (right) generic. Then there exists an algebraic group
  $G$ over \(C\cap K\) and a \(C\)-definable homomorphism $f: H \to
  G$, with purely imaginary kernel.
\end{prop}

\begin{proof}
  Let $F = C \meet K$. Let $p$ be the principal generic type of $H$;
  let $(a_1,a_2) \models p^{\tensor 2} | C$ and $a_3=a_1\cdot
  a_2$. Let $\tau = \{a_1,a_2,a_3,a_{12}=a_1\cdot a_2, a_{23}=a_2\cdot
  a_3,a_{123} = a_1\cdot a_2\cdot a_3\}$.  We view this six-element
  set as a \emph{matroid}, given by specifying the collection of
  algebraically dependent subsets of $\tau$.  This data is called a
  group configuration.

  For $c\in\tau$, let $A(c) = \acl(Cc) \meet K$; this is the set of
  field elements in the algebraic closure of $c$ over $C$. Pick a
  tuple $\alpha(c) \in A(c)$ such that $A(c) = \alg{F(\alpha(c))}$. We
  view $c \mapsto \alpha(c)$ as a map on $\tau$ into the matroid of
  algebraic dependence in the algebraically closed field $K$ over $F$.
  Then $\alpha$ preserves both dependence and independence.  For
  independence this is clear. Preservation of dependence follows from:

  \begin{claim}
    Let $E_1$ and $E_2$ be two algebraically closed substructures of a
    model of $\ACVF$, all sorts allowed.  Let $L_i$ be the set of
    field elements of $E_i$.  If $c \in \acl(E_1 \union E_2)\cap K$,
    then $c \in \alg{(L_1L_2)}$.
  \end{claim}

  \begin{proof}
    This follows immediately from the fact that there are no definable
    maps with infinite range from a geometric sort other than $K$ to
    $K$ (cf. Remark\ \ref{R:Geom pure im}).
  \end{proof}

  According to the group configuration theorem for stable theories,
  applied to the theory $\ACF$ over the model $F$, there exists a
  group $G$, $\ACF$-definable over $F$, such that $a(\tau)$ is a group
  configuration for $G$; in particular there exist $b_1, b_2 \in G$,
  $b_{12}=b_1b_2$ such that $A(a_i) = \alg{F(b_i)}$.  Since
  $\ACF$-definable groups are (definably isomorphic to) algebraic
  groups, we can take $G$ to be an algebraic group over $F$
  (cf. \cite[Proposition\ 3.1]{HruPil-GpPFF}).

  We work in the group $H \times G$. Let $c_i = (a_i,b_i)$. Since
  $a_i\models p|C$, $p$ is $C$-definable and $C=\acl(C)$, $\tp(c_i/C)$
  has a $C$-definable extension $q_i$. Let $Z= \Stab(q_2,q_3)$. This
  is a coset of $S := \Stab(q_2)$ and $q := \transl{c_1^{-1}}{q_1}$ is
  a generic type of $S$. Let $J = \{h \in H: (h,1) \in S\}$.

\begin{claim}
  $J$ is purely imaginary.
\end{claim}

\begin{proof}
  Pick any $h\in J$, $D\supseteq C$ and $(a,b)\models q_2
  |\acl(D h)$. Then $(h\cdot a,b)\models q_2 | D h$. Since all geometric
  sorts but $K$ are purely imaginary and $a\cap K \subseteq \acl(b)$,
  we have $\acl(D a)\cap K \subseteq \acl(Db)$ and hence $\acl(D,h\cdot
  a)\cap K \subseteq \acl(Db)$. So $\acl(D h)\cap K\subseteq
  \acl(D,a,h\cdot a)\cap K \subseteq \acl(Db)$. This inclusion holds
  for any $(a,b)\models q_2 |\acl(D h)$. Since $q_2$ is $D$-definable, in
  particular $D$-invariant, it follows that $\acl(D h)\cap K\subseteq
  \acl(D)$ and $h$ is purely imaginary over any $D\supseteq C$. Therefore $J$ is purely imaginary.
\end{proof}

Note that \(S\) projects onto $H$; this is because the projection
contains a realization of $p_1 | C$ and $p_1$ generates
$H$. Similarly, \(S\) projects dominantly onto \(G\). Let $J' = \{ g
\in G: (1,g) \in S \}$.  By an argument similar to the proof of the
above claim, for all \(h\in J'\) and $(a,b) \models q_2 | Ch$, since
$b \in \acl(a)$, we have that \(h\in\acl(Ca)\). Since \(q_2\) is
\(C\)-invariant, it follows that \(h\in\acl(C)\) and hence \(J'\) is
finite. Note also that \(J'\) is normal in the projection of \(S\) to
\(G\) and hence is a normal subgroup of \(G\).

Now $S$ can be viewed as the graph of a homomorphism $f: H \to G/J'$
with kernel $J$. Since $G/J'$ is isomorphic to an algebraic group
defined over $F$, replacing $G$ by $G/J'$ we may assume that $f: H \to
G$. Since \(S\subseteq H\times G\) is the \(\infty\)-\(C\)-definable
graph of a function, it is relatively definable---we have $(a,b) \notin
S$ if and only if $(a,b') \in S$ for some $b' \neq b$---and $f$ is a
\(C\)-definable function.
\end{proof}

\begin{rem}\label{R:embed-univ}
  Assume that $H$ is a connected $\infty$-definable Abelian group with
  a (right) generic, then there is a universal map $f$ into an algebraic group;  i.e. for any algebraic group $G'$ and
  any definable homomorphism $f' : H \to G'$, there exists a unique
  definable homomorphism $g: f(H) \to G'$ with $f' = g \circ f$.
\end{rem}

Instead of $H$ being Abelian, it suffices to assume that
Question\ \ref{Q:finite-cc} has a positive answer for all definable
subgroups of $H$.

\begin{proof}
  Let $f$ be as in Proposition\ \ref{P:embed}, then, since $J :=
  \ker(f)$ is purely imaginary, $f'(J)$ is finite. Thus $\ker(f')\cap
  J$ is a finite index subgroup of $J$. If we choose $\ker(f)$ to be
  as small as possible, i.e. $\ker(f) / \ker(f)^0$ to be least
  possible, then $J\leq \ker(f')$.
\end{proof}

The following example shows that the assumption that the group has a
generic cannot be eliminated. Without it, it seems likely that one can
find an isogeny between a quotient of $H$ by a purely imaginary group
and a quotient of an ind-algebraic group.

\begin{example}[Versions of tori]\label{E:versions-of-tori}
The value group $\Gamma$ of an algebraically closed valued field is a
divisible ordered Abelian group. So every group definable in $\Gamma$
is locally definably isomorphic to the group $V :=
\Gamma^n$. Nevertheless, interesting variants of $V$ are known. If
$\Delta$ is a finitely generated subgroup of $V$, the convex hull
$C(\Delta)$ of $\Delta$ (with respect to the product partial ordering
of $V$) is an ind-definable subgroup (a direct limit of definable
sets) which is itself not definable. But the quotient
$C(\Delta)/\Delta$ is canonically isomorphic to a group definable in
$\Gamma$ (cf. \cite{PetSta-Tori}). One example of such a group is, given any $\tau \in\Gamma_{>0}$, the group $\Gamma(\tau) = [0,\tau)$ with the law $x \star y = x+y$ if it stays in $\Gamma(\tau)$ and $x\star y = x+y-\tau$ otherwise.

Note that a non realized definable type concentrating on \(\Gamma(\tau)\) corresponds to a definable cut and hence has a trivial stabilizer. It follows that \(\Gamma(\tau)\) does not have generics.

These examples lift to the tori $T = \Gm^n$ over algebraically closed
valued fields. Let $r: T \to V$ be the homomorphism induced by the
valuation map. Let $\Delta$ be a finitely generated subgroup of $T$.
Then $C(\Delta) := r^{-1}(C(r(\Delta)))$ is ind-definable, and
$C(\Delta)/\Delta$ is definable in $\ACVF$. For example, fix $t \in K$
with $\val(t)=\tau > 0$, and define a group structure on $A(t) :=
\{x\mid 0 \leq \val(x) < \tau \}$ by: $x \star y = x\cdot y$ if
$x\cdot y \in A(t)$, $x \star y = x\cdot y/t$ otherwise. This group
admits a homomorphism onto the definably compact $\Gamma$-internal
group $\Gamma(\tau)$, with stably dominated kernel. Since $\Gamma(\tau)$ does not have generics, neither does $A(t)$.

Note that $A(t)$ contains no infinite purely imaginary
subgroup. Moreover, if, for some finite group $J$, $A(t)/J$ is a
Zariski-dense subgroup of an algebraic group $G$, then the group
configuration of $G$ is inter-algebraic with that of $\Gm$. It follows
that $G$ is isogenous to $\Gm$. But $A(t)$ has too many torsion points
for this to be possible.
\end{example}

\begin{cor}\label{C:embed2}
  Let \(C = \acl(C)\) and $H$ be a connected $\infty$-\(C\)-definable
  stably dominated group. Then, there exists an algebraic group $G$
  over \(C\cap K\) and a \(C\)-definable homomorphism $f: H \to G$,
  with boundedly imaginary kernel.
\end{cor}

\begin{proof}
  We follow the notation and proof of Proposition\ \ref{P:embed}. Note
  that, by Proposition\ \ref{P:stdom}.(\ref{aclstdom}) and
  Proposition\ \ref{P:stdomdef}, $q_i$ is stably dominated over
  $C$. There only remains to show that $J$ is boundedly imaginary. So
  pick an $h \in J$. Then for any $(a,b) \models
  q_2 | C h$, $(h\cdot a,b) \models q_2 | Ch$. Now $a$ is purely
  imaginary over \(Cb\). By
  Lemma\ \ref{L:se1}, there exists $\beta_0 \leq 0 \leq \beta_1 \in
  \Gamma_C(a,b)$ and a tuple $d\in \beta_0 \Oo
  / \beta_1 \fM$ such that $a \in \dcl(C, b, \beta_0,\beta_1, d) $. Since $\tp(ab/C)$ is stably dominated, we have
  $\beta_0,\beta_1 \in \Gamma(C)$, and $a \in \dcl(C, b, d)$. Since $\tp(h\cdot a,b/C) = \tp(a,b/C)$, we also have
  $h\cdot a \in \dcl(C, b, d')$, for some tuple $d'\in \beta_0 \Oo
  / \beta_1 \fM$. Thus $h \in
  \dcl(a,h\cdot a,C) \subseteq \dcl(C,b,d,d')$. Since $q_2$ is $C$-invariant and $\beta_0\Oo/\beta_1\fM$ is stably embedded $C$-definable, it follows that $h\in\dcl(C,d'')$ for some tuple $d''\in \beta_0 \Oo
  / \beta_1 \fM$. By
  compactness and Corollary\ \ref{C:bdd-im}, $J$ is boundedly
  imaginary.
\end{proof}

A very similar proof yields the following:

\begin{rem}\label{R:pureim-stdom}
  Let $H$ be a purely imaginary stably dominated definable group. Then
  $H$ is boundedly imaginary.
\end{rem}

\begin{proof}
  Let $p$ be the principal generic type of $H$; with $H$ and $p$
  defined over $C = \acl(C)$. Then, for any $a \models p |C$,
  $\Gamma_C(a) \subseteq C$. Since $a$ is purely imaginary over $C$,
  it follows from Lemma\ \ref{L:se1}, stable domination of $p|C$ and Corollary\ \ref{C:bdd-im}
  that $a$ is boundedly imaginary over $C$. Any element of $H^0$ is a
  product of two realizations of $p|C$, so $H^0$ is boundedly
  imaginary. It follows that some definable set $H'$ containing $H^0$
  is boundedly imaginary.  Finitely many translates of $H'$ cover $H$,
  so $H$ is also boundedly imaginary.
\end{proof}

\subsection{Stably dominated subgroups of algebraic groups}

Our goal in this section is to relate stably dominated subgroups of affine algebraic groups to algebraic geometric objects. To do so, we consider schemes over \(\Oo\). Note that schemes only appear this section and that this section is independent from the rest of the paper. Note also that the technical hypotheses are meant to insure that every thing works as one would expect. All of the required definitions regarding schemes can be found in \cite{Mum-Red}.

All the schemes over $\Oo$ that we consider will be assumed to be flat
and reduced over $\Oo$. By an $\Oo$-variety (or variety over $\Oo$),
we mean a (flat and reduced) scheme of finite over $\Oo$; it admits a finite
open covering by schemes isomorphic to $\spec(\Oo[X_1,\ldots,X_n] / I)$. In particular, an affine variety over $\Oo$ is of the form $\spec(\Oo[X_1,\ldots,X_n] / I)$. Note that we do not require varieties to be irreducible. Since $\Oo$ is a valuation ring, an affine scheme $\spec(R)$ is flat over $\spec(\Oo)$ if and only if no non-zero element of $\Oo$ is a $0$-divisor in $R$.
So $\spec(\Oo[X_1,\ldots,X_n] / I)$ is flat if and only if $ I = I K
[X_1,\ldots,X_n ] \meet \Oo[X_1,\ldots,X_n] $. Hence there are no
infinite descending chains of $\Oo$-subvarieties.

For any scheme $V$ over $\Oo$, we write $V(\Oo)$ for the set of its $\Oo$-points; the set of morphisms $\spec(\Oo) \to V$ over $\spec(\Oo)$. When $V = \spec(\Oo[X_1,\ldots,X_n] / I)$ is an affine variety over $\Oo$, $V(\Oo) = \{x\in \Oo^n\mid \forall f\in I, f(x) = 0\}$. Conversely, for any set $Z\subseteq \Oo^n$, let $I := \{f \in \Oo[X_1,\ldots,X_n]: f(Z) = 0\}$ and $R = \Oo[X_1,\ldots,X_n]/I$. Then $V =
\spec(R)$ is flat over $\spec(\Oo)$ and if $Z = W(K) \meet \Oo^n$ for
some affine variety $W$ over $K$, then $V(\Oo) = Z$. Note that the set $V(\Oo)$ is pro-definable in $\ACVF$. When $V$ is of finite type, $V(\Oo)$ is definable.

If $V$ is a scheme over $\Oo$, we write $V_K := V \times_{\Oo}K$ for the generic fiber and $V_k := V \times_{\Oo}k$ for the special fiber. Note that if $V$ is of finite type, both the generic fiber and the special fibers are, non necessarily reduced, varieties over, respectively, $K$ and $k$. Let $\rho : V(\Oo) \to V_k(k)$ be the natural (pro-)definable map. When $V$ is an affine variety over $\Oo$ it simply consists in coefficient-wise reduction modulo $\fM$. If $W'$ is a subvariety of $V_K$, there exists a unique $\Oo$-subvariety $W$ of $V$ such that $W_K =
W'$, and $W(\Oo) = W'(K) \meet V(\Oo)$. For example, if $V' \subset
\Aa^n$ is an affine variety over $K$, defined by a radical ideal $P
\subset K[X]$, we let $V=\spec(\Oo[X] / P \meet \Oo[X])$. Then $V_K = V'$ and $V_k$
is the zero set of the image of $P \meet \Oo[X] $ in $k[X]$. In this
case, we denote the affine coordinate ring $K[X]/P$ by $K[V]$, and
$\Oo[V] = \Oo[X] / (P \meet \Oo[X])$.

\begin{lem}\label{L:zdense}
  Let $V$ be an irreducible affine $\Oo$-variety and assume
  $V(\Oo)\neq\emptyset$. Then $V(\Oo)$ is Zariski dense in $V_K$.
\end{lem}

\begin{proof}
  Let $L$ be a sufficiently saturated algebraically closed field extending $K$. Let $a
  \in V_K(L)$ be a $K$-generic point. Thus for any $b \in V(K)$ there
  exists a $K$-algebra homomorphism $h: K[a] \to K$ with $h(a)=b$. In
  particular, there exists such an $h$ with $h(a) \in V(\Oo)$. It
  follows that the maximal ideal $\fM$ of $K$ generates a proper ideal
  of $\Oo[a]$: otherwise, for some $m \in\fM$ and $f \in \Oo[X]$,
  $m\cdot f(a) = 1$; but then applying $h$, we would have $m \cdot f(
  h(a)) = 1$. Thus $\fM$ extends to a maximal ideal $\fM'$ of
  $\Oo[a]$, and thence to a maximal ideal $\fM''$ of some valuation
  ring $\Oo_L$ of $L$, with $\Oo[a] \subset \Oo_L$.  Thus the
  valuation on $K$ can be extended to $L$ in such a way that $a \in
  \Oo(L)$. By model completeness of $\ACVF$, there exists $a' \in
  V_K(K)$, with coordinates in $\Oo$, outside any given proper
  $K$-subvariety of $V$.
\end{proof}

We denote the Krull dimension of schemes by $\dim$. Note that when $V$ is a variety over $k$, $\dim(V) = \RM(V) = \dimst(V)$.

\begin{lem}\label{L:bdd}
  Let $V$ be a scheme over $\Oo$, with $\dim(V_K) = n$.  Let $q$ be a
  $K$-definable type of elements of $V(\Oo)$ with $\dimst(\rho_\star q) =
  n$.  Assume $V$ is defined over $B=\acl(B)$; then $q$ is stably
  dominated over $B$.
\end{lem}

\begin{proof}
  Say $q$ is defined over $B' = \acl(B') \supset B$.  Let $a \models q |
  B'$. Since $\trdeg(k(B')(\rho(a))/k(B')) \geq n$, we have
  $\trdeg(k(B)(\rho(a))/k(B))\geq n$. It follows that $B(a)$ is an extension of
  $B$ of transcendence degree $n$ and residual transcendence degree
  $n$. The type \(\tp(a/B)\) is therefore stably dominated via
  $\rho$. Indeed, $a$ is in the algebraic closure over $B$ of $n$ independent generics of $\Oo$. Since $\rho(a)\ind_{B} B'$, $q$ is its unique $B$-definable
  extension.
\end{proof}

Recall that a group scheme $G$ over $\Oo$ is a scheme $G$ over $\Oo$ with three morphisms of schemes over $\Oo$: the multplication $\mu : G\times_\Oo G \to G$, the inverse map $\iota : G \to G$ and the identity $\epsilon : \spec(\Oo) \to G$, such that the obvious diagrams commute. When $G = \spec(R)$ is affine, these maps correspond to three operations: the comultiplication $R \to R\tensor_\Oo R$, the coinversion $R \to R$, and the counit $R \to \Oo$, making $R$ into a Hopf algebra. In fact, $\spec$ induces a correspondance between affine group schemes and Hopf algebras.

Let $G$ be a group scheme over $\Oo$, with generic fiber $G_K$ and
special fiber $G_k$. Then $G(\Oo)$ is a (pro-)definable subgroup of the (pro-)definable group $G_K(K)$, and the previously defined map $\rho : G(\Oo) \to G_k(k)$ is a group homomorphism. Note that, when $G$ is not a scheme of finite type, $\dim(G_k)$ might be smaller than $\dim( G_K )$; cf. Example\ \ref{E:noninertial}.

\begin{prop}\label{P:genchar}  
  Let $p$ be a definable type concentrating on $G(\Oo)$. Assume that
  $\rho_\star p$ is a generic type of $G_k$ and that \(\dim(G_K) =
  \dim(G_k) <\infty\).  Then $p$ is a stably dominated generic type of
  $G(\Oo)$. When $G$ is a scheme of finite type, $G(\Oo)$ has finitely many
  generic types.
\end{prop}
 
\begin{proof}
  Consider translates $q=\transl{g}{p}$ of $p$, for some $g \in
  G(\Oo)$. Clearly $\rho_\star q = \transl{\rho(g)}{(\rho_\star p)}$.  By
  Lemma\ \ref{L:bdd}, $q$ is stably dominated over $\acl(B)$, where $B$
  is a base of definition of $G$.  So $p$ itself is stably dominated and since all of its translates are defined over $B$, which is small, $p$ is generic.

  When $G$ is of finite type, by Corollary\ \ref{C:def G0}, $G^0$ is
  definable and hence, the orbit of generic types is finite.
\end{proof}

\subsubsection{Linear groups}
Our goal is to show that all stably dominated
$\infty$-definable subgroups of affine algebraic groups, may be obtained as the $\Oo$-points of a group scheme $G$
over $\Oo$. This result has since been generalized in \cite{Hal-GenStGp} to the
non-affine setting. Note that Example\ \ref{E:noninertial} shows that
we cannot always assume $G$ to be of finite type.

We start with a version of the maximum modulus principle for generics
of subgroups:

\begin{prop}\label{P:genchar2}
  Let $G$ be an affine algebraic group over $K$, $H$ be a
  Zariski dense definable subgroup of $G(K)$ and $p$ be a definable
  type of elements of $H$.  Then the following are equivalent:
     \begin{enumerate}
     \item $p$ is the unique stably dominated generic of $H$.
     \item For any regular function $f$ on $G$, $p$ attains the
       highest modulus of $f$ on $H$; i.e for some $\gamma_f$, for
       any $x \in H $, $\val f(x) \geq \gamma_f$ and
       \[\models \defsc{p}{x}( \val f(x) = \gamma_f).\]
\end{enumerate} 
\end{prop}

Note that, because any function constant on connected components of
$G$ is regular, connectedness of $G$ follows from these assumptions.

\begin{proof}
  Let us first prove that (1) implies (2). Since $p$ is stably
  dominated, for any regular $f$ there exists $\gamma_f$ with
  $\defsc{p}{x} (\val(f(x)) = \gamma_f)$. If $(a,b) \models p \tensor
  p $ then $a\cdot b \models p$; so $\val(f(a\cdot b)) = \gamma_f$. By
  Corollary\ \ref{C:maxmod2}, for any $a,b \models p$ we have
  $\val(f(a\cdot b)) \geq \gamma_f$.  But since $p$ is the unique
  generic, any element $c$ of $H$ is a product of two realizations of
  $p$.  Thus $\val(f(c)) \geq \gamma_f$.

  Let us now prove that (2) implies (1). Note that, using quantifier
  elimination for algebraically valued fields, (2) characterizes $p$
  uniquely. Also, since $\gamma_f$ does not depend on $x$, $p$ is
  orthogonal to $\Gamma$, hence stably dominated, by
  Proposition\ \ref{P:orth}. On the other hand, for any $h\in H$,
  replacing $f$ by $f(h\cdot x)$, we see that (2) is invariant under
  $H$-translations. Thus if (2) holds of $p$, it holds of every
  translate, so every translate of $p$ equals $p$. By
  Remark\ \ref{R:1genr}, $p$ is the unique generic type of $H$.
\end{proof}

\begin{prop}\label{genchar2.1}
  Let $G$ be an affine group scheme over $\Oo$. Let $p$ be a
  stably dominated type of elements of $G(\Oo)$. Assume:
  \begin{itemize}
    \item[($\star$)] if $f \in K[G]$ and $\val(f(x)) \geq 0$ for $x
      \models p$, then $f \in \Oo[G]$.
  \end{itemize}

 Then $p$ is the unique generic of $H:=G(\Oo)$ if and only if
 $\push{\rho}{p}$ is the unique generic type of $\rho(H)$.
\end{prop}

\begin{proof}
Let us assume that $p$ is the unique generic of $H$. In an algebraic
group, to show that a type is the unique generic is to show that any
regular function $f$ vanishing on the type, vanishes on the whole
group. Let $f \in k[\rho(H)]$ vanish on $\push{\rho}{p}$. Lifting to
\(\Oo\), we have $F \in \Oo[G]$ with $\val(F(a)) >0$ for $a \models
p$.  By Proposition\ \ref{P:genchar2}.(2), $\val(F(a')) >0$ for all
$a' \in H$. So $f$ vanishes on $\rho(H)$.

The converse uses ($\star$). Assume that $\push{\rho}{p}$ is generic in
$\rho(H)$. We want to show that Proposition\ \ref{P:genchar2}.(2) holds
to be able to conclude. Let $F \in K[G]$. Since $p$ is stably
dominated, for some $\gamma$, for any $a \models p$, $\val(F(a)) =
\gamma$.  If $\gamma = \infty$, then, by ($\star$), any $K$-multiple of $F$
lies in $\Oo[G]$, so $F=0$. Otherwise, we may assume $\gamma=0$. By
($\star$), it follows that $F \in \Oo[G]$. Suppose there exists $a'\in
G(\Oo)$ with $\val(F(a')) = \val(c) < 0$---we may take $a' \in G(\Oo
_0)$, a fixed submodel. Then $c^{-1}F \in \Oo[G]$, and
$\val(c^{-1}F(a)) > 0$ for $a \models p$, i.e. $\res(c^{-1}F(a)) = 0$.
By genericity of $\push{\rho}{p}$, $\res(c^{-1}F)$ vanishes on $\rho(H)$; so
$\val(c^{-1}F(a')) > 0$ for all $a' \in G(\Oo)$; a contradiction.
\end{proof}

\begin{thm}\label{T:genchar3}
  Let $G$ be an affine pro-algebraic group over $K$. Let $H$ be a
  Zariski dense pro-definable connected stably dominated subgroup
  of $G(K)$. Then there exists a prolimit $\Hh$ of affine group schemes of finite type over $\Oo$ and an isomorphism $\phi: G \to \Hh_K$, such that $\phi(H) = \Hh(\Oo)$.

  If $G$ is an affine algebraic group and $H$ is definable, $\Hh$ can
  be chosen to be an affine group scheme of finite type over $\Oo$.
\end{thm}

In the proof below, the duality between affine schemes and algebras, which induces a duality between affine group schemes and Hopf algebras, is used to rewrite the above statement as a purely algebraic result on Hopf $\Oo$-subalgebras of $K[G]$. We then prove this purely algebraic result.

\begin{proof}
  Let $K_0 = \alg{K_0}$ be some subfield of $K$ over which $G$,
  $H$ are defined, $\Oo_0 = \Oo\cap K_0$ its valuation ring,
  $R_0:=K_0[G]$ be the affine coordinate ring of $G$ and $p$ be the
  stably dominated generic of $H$. Note that since $p$ generates $H$,
  $p$ is Zariski dense in $G$. It follows that equalities in $R_0 =
  K_0[G]$ can be determined by evaluating at any realization of
  $p|K_0$. We define $S_0 = \{r \in K_0[G]: \defsc{p}{x} \val(r(x)) \geq
  0\}$.  This is an $\Oo_0$-subalgebra of $R_0$.

  \begin{claim}\label{C:tensor}
    $S_0 \tensor_{\Oo_0} K_0 = R_0$.
  \end{claim}

  \begin{proof}
    Pick any $0 \neq r \in R_0$. Since $p$ is stably dominated, it is
    orthogonal to $\Gamma$.  Thus for some $c \in K_0$, for $a \models
    p | K_0$, $\val(r(a)) = \val(c)$.  If $c = 0$, then $r$ vanishes on $p$
    and hence on $G$, i.e. $r=0 \in R_0$, contradicting the choice of
    $r$.  So $c \neq 0$, and $c^{-1}r \in S_0$.  This shows that the
    natural map $S_0 \tensor_{\Oo_0} K_0 \to R_0$ is surjective.
    Injectivity is clear since $S_0$ has no $\Oo_0$-torsion.
  \end{proof}

  Let $\Hh = \spec(S_0)$. We have $\Hh_{K} := \Hh \times _{\Oo_0}
  \spec(K) \simeq \spec( S_0 \tensor_{\Oo_0} K) = G$. We identify
  $\Hh_K$ with $G$. So $p$ is the type of elements of $G(K)$ and in
  fact, by definition of $S_0$, of $G(\Oo) = \Hh(\Oo)$.

  The morphisms $x \mapsto x^{-1}: G \to G$ and $(x,y) \mapsto x\cdot
  y: G^2 \to G$ correspond to two operations:
  \[i: R_0 \to R_0, \  i(r)(a) = r(a^{-1})\] and
  \[c: R_0 \to R_0 \tensor_{K_0} R_0, \  c(r) = \sum_{i=1}^n r_i \tensor
  s_j, \  r(a\cdot b) = \sum_{i=1}^n r_i(a)s_i(b).\]
  
  Note that any
  $\Oo_0$-subalgebra $S_0'$ of $R_0$ is $\Oo_0$-torsion-free, hence a flat
  $\Oo_0$-module.  Thus the maps $S_0' \tensor_{\Oo_0} S_0' \to S_0'
  \tensor_{\Oo_0} S_0 \to S_0 \tensor_{\Oo_0} S_0$ are injective.  We identify
  $S_0' \tensor_{\Oo_0} S_0'$ with its image in $S_0 \tensor_{\Oo_0} S_0$.  Let us
  say that an $\Oo_0$-subalgebra $S_0'$ of $R_0$ is {\em Hopf} if $i(S_0')
  \subseteq S_0'$ and $c(S_0') \subseteq S_0' \tensor_{\Oo_0} S_0'$. Note that Hopf $\Oo_0$-subalgebras of $R_0$ correspond, via $\spec$, to affine group schemes whose base change to $K$ is a quotient of $G$.

  \begin{claim}
    $S_0$ is Hopf
  \end{claim}

  \begin{proof}
    Let us first consider co-inversion. Let $s = i(r)$ for some $r \in
    S_0$. Clearly $\val(s(x)) = \val (r(x^{-1})) \geq 0 $ for any $x
    \models p$. Hence $s \in S_0$.

    Let us now consider co-multiplication. Pick any \((a,b)\models
    p^{\tensor 2}\) and \(r\in S_0\). Write $c(r) = \sum_j r_j\tensor
    s_j$ for finitely many \(r_j,s_j\in R_0\). By
    \cite[Lemma\ 12.4]{HasHruMac-Book}, and definability of $p$, we may assume that
    \(\val(r(a\cdot b)) = \min_j (\val(r_j(a))+\val(s_j(b)))\). We may
    also assume that no \(r_j(a)\) is zero. As in
    Claim\ \ref{C:tensor}, we renormalize so that \(\val(r_j(a)) = 0\)
    for all \(j\). Since \(a\cdot b\models p\), it follows that
    \(\min_j \val(s_j(b)) = \val(r(a\cdot b))\geq 0\). Since both $a$
    and $b$ realize $p | K_0$, for any $c \models p |K_0$ we have
    $\val(r_j(c))=0$ and $\val(s_j(c)) \geq 0$, i.e. $r_j ,s_j \in S_0$.
  \end{proof}
  
  It follows that $\Hh = \spec(S_0)$ is an affine group scheme over $\Oo_0$ whose base change to $K$ is $G$.
  
  \begin{claim}
    The type $p$ is the unique generic of $\Hh(\Oo)$.
  \end{claim}

  \begin{proof}
    Note that, by definition, for all \(x\in\Hh(\Oo)\) and \(r\in S_0\),
    \(\val(r(x))\geq 0\). Pick any \(r\in S_0\). As in
    Claim\ \ref{C:tensor}, we find $c\in K_0$ such that
    \(\defsc{p}{x}\val(r(x)) = \val(c)\) and \(c^{-1}r \in S_0\). Thus,
    for all \(x\in \Hh(\Oo)\), \(\val(r(x))\geq\val(c)\). The claim
    now follows by Proposition\ \ref{P:genchar2}.
  \end{proof}

  Since they have the same generic, it follows that \(\Hh(\Oo) =
  H\).
  
  Let us now prove that \(\Hh\) is a prolimit of affine group schemes of finite type. Let $\Ff$ be the family of finitely generated
  $\Oo_0$-subalgebras of $S_0$ that are Hopf. If $S_0' \in \Ff$ then
  $\spec(S_0')$ is an affine group scheme of finite type over \(\Oo_0\), which is a quotient of $\Hh$. So, to prove that $\Hh$ is the prolimit of its quotients of finite type, it suffices, dually,
  to show that $\Ff$ is filtered and that $S_0$ is the direct limit of the $S_0'\Ff$. Note
  that if $S$ is generated by $S_0'$ and $S_0''$ as an $\Oo_0$-algebra,
  then $S$ is closed under $c$ if $S_0'$ and $S_0''$ are. Indeed $c: S_0
  \to S_0 \tensor_\Oo S_0$ is an $\Oo_0$-algebra homomorphism, so $\{r:
  c(r) \in S \tensor_{\Oo_0} S\}$ is an $\Oo_0$-subalgebra of $S$,
  hence equal to $S$ since it contains $S_0'$ and $S_0''$. Moreover, by
  the commutativity rules between \(c\) and \(i\), the $\Oo_0$-algebra
  $i(S_0')$ is also closed under \(c\) and hence, since the
  $\Oo_0$-algebra generated by $S_0' \union i(S_0')$ is closed under $i$,
  it is Hopf. It follows that it suffices, given \(r\in S_0\), to find a
  finitely generated $\Oo_0$-subalgebra $S_0'$ of $S_0$ closed under \(c\)
  with $r \in S_0'$.

  Fix $r\in S_0$ and write $c(r) = \sum_{i=1}^n r_i \tensor s_i$, with
  $n$ least possible, and $r_1,\ldots,r_n, s_1,\ldots,s_n \in S_0$.

  \begin{claim}\label{C:InvMat}
    Let $(a_1,\ldots,a_n) \models p^{\tensor n} | K_0$. The matrix
    $s=(s_i(a_j))_{1 \leq i,j \leq n} $ is invertible over $K$.
  \end{claim}

  \begin{proof}
    If not, there exists a non-zero vector $\alpha =
    (\alpha_1,\ldots,\alpha_n)$ in $K$ with $\alpha \cdot s = 0$.  We may
    assume $\alpha_n=1$. Since the $a_i$ are independent and $\alpha$ has weight
    at most $n-1$, some $a_j$ must be independent (in the underlying
    algebraically closed field) from $\alpha$. Therefore, for any $a
    \models p | K_0\alpha$, $\sum_i \alpha_i s_i(a) = 0$ 
    and hence
    $\sum_i \alpha_i s_i = 0$. Since $p$ is $K_0$-definable and the $s_i$ are in $R_0 = K_0[G]$, we may assume that $\alpha_i \in K_0$. Thus $s_1,\ldots,s_n$ are
    $K_0$-linearly dependent, contradicting the minimality of $n$.
  \end{proof}

  The expression $r(x\cdot y)=\sum_{i=1}^n r_i(x)\cdot s_i(y)$ shows
  that $\{r(g\cdot y): g \in G(K_0)\}$ spans a finite-dimensional
  $K_0$-subspace of $K[G]$.  Similarly $\{r(g\cdot y\cdot h): g,h \in
  G(K_0)\}$ spans a finite-dimensional $K_0$-subspace $V_0$ of
  $R_0$. Let $S_0'$ be the $\Oo_0$-algebra generated by $V_0\cap S_0$.

  \begin{claim}
    $c(V_0)\subseteq V_0\tensor_{\Oo_0} V_0$.
  \end{claim}

  \begin{proof}
    Let $(g_1,\ldots,g_n,x) \models p^{\tensor n+1} | K_0$. The
    $r(x\cdot g_j) = \sum_i a_i(x)b_i(g_j)$ are in $V_0$ by definition. By
    Claim\ \ref{C:InvMat}, $a_i\in V_0$. By the symmetric argument, we
    also have $b_i\in V_0$ and hence $r \in A_0 := \{r\in V_0\mid c(r)\in
    V_0\tensor_{\Oo_0} V_0\}$. Note that, by definition, $A_0$ is closed by
    left and right multiplication by $G(K_0)$, so $V_0\subseteq A_0$ and
    thus $c(V_0)\subseteq V_0\tensor_{\Oo_0} V_0$.
  \end{proof}

  By the renormalization argument of Claim\ \ref{C:tensor},
  $V_0\tensor_{\Oo_0} V_0 \subseteq V_0\tensor_{\Oo_0} S_0 = S_0\tensor_{\Oo_0}
  V_0$ and hence $c(V_0\cap S_0) \subseteq (V_0\tensor_{\Oo_0} V_0) \cap
  (S_0\tensor_{\Oo_0} V_0)\cap (S_0\tensor_{\Oo_0} V_0)\cap (S_0\tensor S_0) =
  (S_0\cap V_0)\tensor_{\Oo_0}(S_0\cap V_0)$. It immediately follows that
  $c(S_0')\subseteq S_0'\tensor S_0'$. Since $r\in S_0'$, we are done.

  The finite type statement follows by compactness.
\end{proof}

In Theorem\ \ref{T:genchar3}, we cannot always assume $G$ to be finite
type:

\begin{example}\label{E:noninertial}
  Let $K_0 = \alg{\Cc(t)}$, with $\val(t)>0$.  Let $H_n := \{(x,y) \in (\Ga
  \times \Gm)(\Oo): \val(y - \sum_{i=1}^n 1/i! (t\cdot x)^i) \geq \val(t^{n+1})
  \}$, a definable subgroup of $(\Ga \times \Gm) (\Oo)$, isomorphic to
  $(\Ga \times \Gm)(\Oo)$. The group $H := \meet H_n$ is
  connected stably dominated; but its generic is dominated via the map
  $(x,y) \mapsto \res(x)$. The Zariski closure of $H$ has dimension
  $2$, but the generic type of $H$ has a residual part of
  transcendence degree one.
  
  It follows that $H$ cannot be the $\Oo$ points of a group scheme of
  finite type, nor can it be the connected component of a definable
  group.
\end{example}

\begin{lem}\label{conn}
  Let $G$ be a group scheme over $\Oo$. For each $n$, let $\phi_n(g)
  := g^n$, and assume $\phi_n: G \to G$ is a finite morphism.  Then
  $G(\Oo)$ is connected.
\end{lem}
 
\begin{proof}
  By properness, $\phi_n: G(\Oo) \to G(\Oo)$ is surjective. So any
  finite quotient has order prime to $n$.  This holds for all $n$, so
  $G(\Oo)$ has no finite quotients.
\end{proof}

\subsection{Abelian varieties}

\begin{defn}\label{D:bdd-serre}
  Let $V$ be an affine variety over $K$. A definable subset $W$ is
  \emph{bounded} if for any regular function $f$ on $V$,
  $\{\val(f(x))\mid x\in W\}$ has a lower bound. For a general
  $K$-variety $V$, a definable subset $W$ is \emph{bounded} if there
  exists an open affine covering $V = \bigcup_{i=1}^m U_i$, and a
  bounded $W_i \subseteq U_i$, with $W = \bigcup_i W_i$.
\end{defn}

This definition is due to \cite[\S 6.1]{Ser-MW}. The assumption there
that the valuation is discrete is inessential.

If $V'$ is a closed subvariety of $V$ and $W$ is bounded in $V$, then
clearly $W \cap V'$ is bounded in $V'$. In the affine case, if $V$ has
coordinate ring $K[f_1,\ldots,f_n]$, for $V$ to be bounded (as a subset of itself), it suffices that the $f_i$ have
bounded valuation. In the case of projective space $\Pp^n$, a standard
covering by bounded sets is given in projective coordinates by: $U_i =
\{(x_0: \cdots : x_n): x_i = 1\text{ and }(\forall j \neq i) \val(x_j)
\geq 0\}$. Complete varieties are bounded as subsets of themselves.

Let $C\models\ACVF$, so $C = \dcl(F)$ where $F = C \meet K$; $F$ is an
algebraically closed valued field. Let $\fF$ be the family of
all $C$-definable functions on $V$ into $\Gamma$. Recall (see
Proposition\ \ref{P:orth}) that a $C$-definable type $r$ concentrating on $V$ is stably
dominated if and only if for any $f \in \fF$ there exists $f(r) :=
\gamma \in \Gamma(M)$ such that if $c \models r |C$ then $f(c) =
\gamma$.

\begin{defn}
  Let $q$ be a $C$-definable type extending a type $q_0$ over $C$, let
  $(p_t: t \models q_0)$ be a family of stably dominated
  $Ct$-definable types concentrating on some variety $V$ over $F$ and
  $U$ be an open affine of $V$. We say that the family
  $(p_t)_{t\models q_0}$ is {\em uniformly bounded at $q$ (on $U$)} if
  for any regular function $f$ on $U$ defined over $D\supseteq C$,
  there exists $\alpha \in \Gamma$ such that if $t \models q| D\alpha$
  and $c \models p_t| D t \alpha$ then $c \in U$ and $\val(f(c)) \geq
  \alpha$.
\end{defn}

\begin{lem}\label{L:bdd-bdd}
  Let $p_t$ and $q$ be as above. Assume $p_t$ concentrates on a
  bounded $W \subseteq V$. Then $(p_t)_{t\models q_0}$ is uniformly
  bounded at $q$.
\end{lem}

\begin{proof}
  The types $p_t$ (when $t \models q_0$) concentrate on one of the
  bounded affine sets $W_i$ in Definition\ \ref{D:bdd-serre}. Let $U$ be
  the corresponding affine $U_i$. Any regular function on $U_i$ is
  bounded on all of $W_i$, hence in particular on realizations
  of the $p_t$.
\end{proof}

Recall Definition\ \ref{D:G-limit}.

\begin{lem}\label{L:sd-limit}
  Let $p_t$ and $q$ be as above. Assume $(p_t)_{t\models q_0}$ is
  uniformly bounded at $q$ on $U$. Then there exists a unique
  $C$-definable type $p_\infty := \lim_q p_t$ such that for any
  regular function $f$ on $U$, if $a \models p_\infty$ then
  $\val(f(a)) = \lim_q \val(f(p_t))$.

  Moreover, $p_\infty$ is stably dominated and if $h: V \to W$ is an
  isomorphism of varieties, then $\lim_q h_\star p_t = h_\star \lim_q p_t$.
\end{lem}

\begin{proof}
  By assumption, we cannot have $\lim_q \val(f(p_t)) = -\infty$. The
  set \(\{f\mid \lim_q \val(f(p_t)) = +\infty\}\) is a prime ideal
  $I$; part of the condition on $p_\infty$ is that $(f=0) \in
  p_\infty$ if and only if $f \in I$. Let $V'\subseteq U$ be the zero
  set of $I$. The affine coordinate ring of $V'$ is $\Oo_V(U) / I$. An
  element of $F(V')$ can, therefore, be written $g/h$ with $g,h \in
  \Oo_V(U)$ and $h \notin I$; equivalently $\lim_q p_t(h) \neq \pm
  \infty$. Hence we can define a valuation $\overline{\val}$ on
  $F(V')$ by $\overline{\val}(g/h) = \lim_q g(p_t) - \lim_q
  h(p_t)$. This determines a valued field extension $F^+$ of $F$ and
  hence gives a complete type $p_\infty$ of elements of $V'$. It is
  clear that $p_\infty$ is definable; and that that $\Gamma(F^+) =
  \Gamma(F)$, so $p_\infty$ is orthogonal to $\Gamma$, and hence, by
  Proposition\ \ref{P:orth}, it is stably dominated. Note that
  $\val(f(p_\infty)) = \lim_q \val(p_t(f))$ for any $f \in F(V')$.  In
  particular the choice of $U$ is immaterial. The functoriality is
  evident.
\end{proof}

Recall Definition\ \ref{D:defcomp}.

\begin{lem}
\label{L:stdom-sup}
Let $G$ be a bounded $C$-definable subgroup of an
algebraic group $\tilde{G}$ over $F$. Let $(H_t)_t$ be a certifiably stably
dominated $C$-definable family of subgroups of $G$ forming a directed system
under inclusion and let $H := \bigcup_t H_t$. Assume $G/H$ is
$\Gamma$-internal. Then $H$ is stably dominated. Moreover $G/H$ is definably
compact.
\end{lem}

\begin{proof}

  Let $q(t)$ be a $C$-definable type cofinal in the partial ordering
  $H_t \subseteq H_{t'}$, given by \cref{cofinal-def}. Let $p_t$ be the principal generic type of
  $H_t$. Note that $H = \bigcup_{t\models q| C} H_t^0$. By
  Lemma\ \ref{L:bdd-bdd}, using the boundedness of $G$, the family $(p_t)$ is uniformly bounded at
  $q$, so $p_\infty = \lim_q p_t$ exists.

  \begin{claim}
    Let $H$ and $H'$ be connected stably dominated definable
    subgroups of $G$.  Let $p$ and $p'$ be their generic types. Then $H
    \subseteq H'$ if and only if $p\star p' = p'$.
  \end{claim}

  \begin{proof}
    If $H \subseteq H'$, then $p\star p' = p'$ by definition of
    genericity for $p$.  Conversely if $p\star p' = p'$ then a generic
    of $H$ is a product of two realizations of $p'$; in particular it
    lies in $H'$. But any element of $H$ is a product of two generics,
    hence any element of $H$ lies in $H'$.
  \end{proof}

  By the functoriality of Lemma\ \ref{L:sd-limit}, we have $\lim_q
  \transl{a}{p_t} = \transl{a}{(\lim_q p_t)}$ for any $a \in G$. Let
  $s \models q|C$ and $t \models q|Cs$. If $a \models p_{s} |Cs$, then
  $\transl{a}{p_t} = p_t$, and hence $\transl{a}{p_\infty} =
  p_\infty$. Thus $H_s^0 \subseteq \Stab(p_\infty)$ and thus $H
  \subseteq \Stab(p_\infty)$. On the other hand, as $p_\infty$ is
  stably dominated and $G / H$ is $\Gamma$-internal, the function $x
  \mapsto x\cdot H$ is constant on $p_\infty$, so $p_\infty$ lies in a
  single coset $x\cdot H$. It follows that $p_\infty\star p_\infty^{-1}$ is
  generic in $H$, so $H$ is stably dominated.

  Now let $r$ be a definable type on $\Gamma$, and let $h: \Gamma \to
  G/H$ be a definable function. For any $t \in \Gamma$, let $p'_t $ be
  the generic type of $h(t)$, viewed as a coset of $H$; this is a
  translate of the generic type of $H$. Let $p'_\infty = \lim_r
  p'_t$. Then $p'_\infty$ is stably dominated, so it concentrates on
  a unique coset of $H$, corresponding to an element $e \in G/H$.
  Tracing through the definitions we see that $e = \lim_r h$. Since
  $r$ and $h$ are arbitrary, $G/H$ is definably compact.
\end{proof}

The following is an immediate consequence of Theorem\ \ref{T:lmstd}
and Lemma\ \ref{L:stdom-sup}. Recall that Abelian varieties are complete and hence bounded.

\begin{cor}
  Let $A$ be an Abelian variety over $K$. Then there is a definably
  compact group $C$ defined over $\Gamma$, and a definable
  homomorphism $\phi:A \to C$ with stably dominated kernel $H$. In
  particular $A$ has a unique maximal stably dominated connected
  $\infty$-definable subgroup---which is definable.
\end{cor}

If $A$ has good reduction, i.e $A = \mathcal{A}_K$ where $\mathcal{A}$ is some Abelian scheme over $\Oo$, then, since $A$ is the zero locus of homegeneous polynomials in some projective space, $A(K) = \mathcal{A}(\Oo)$. By Proposition\ \ref{P:genchar},
$A_K(K)$ is stably dominated. Since $A_K(K)$ is divisible, it
is connected. Thus, in this case, the definably compact quotient $C$ is
trivial.
 
In general, if $F$ is locally compact, then the set $C(F)$ of points
of $C$ lifting to $F$-points will be a {\em finite} subgroup of the
definably compact group $C$. On the other hand if $F= \Qq_p((t))$,
$C(F)$ can be a finite extension of $\Zz$.

The stably dominated group $H$ is dominated via a group homomorphism
$h: H\to \fh$ to a stable---i.e. a $k$-internal---group $\fh$. After
base change to a finite extension, $\fh$ becomes isomorphic to an
algebraic group over $k$. It would be interesting to compare this with
the classical theory of semi-stable reduction.

\subsection{Definable fields}

Before we prove the classification of fields definable in
algebraically closed valued fields, let us recall a
result of Zilber, cf. \cite{Zil-AlgGp}. Note that we are not only finding a $K'$-vector space structure on $A$ for some definable field $K'$, but we are also identifying $K'$ with the field \(K\) from which we started. As far as we know, the latter only appears in later work of Poizat, cf. \cite{Poi-GenNS}.

\begin{lem}\label{L:IrrAction}
  Let $(G,\cdot)$ and $(A,+)$ be infinite Abelian groups definable
  in an algebraically closed field $K$. Assume that $G$ acts definably
  on $A$ by group automorphisms and that the action is irreducible:
  i.e. there are no proper non-trivial $G$-invariant subgroups. Then
  $A$ has a definable $K$-vector space structure and $G$ acts linearly
  on $A$.
\end{lem}

\begin{proof}
  Quotienting by the kernel of the action, we may assume that $G$ acts
  faithfully. By Schur's lemma $R := \End_G(A)$ is a division
  ring. Since $G$ is Abelian, we can identify $G$ with a
  multiplicative subgroup of $R^\star$. As in the proof of
  \cite[Theorem 7.8.9]{Mar-IntroMT}, we can show that the ring
  generated by $G$---which is a field---is definable and hence by
  \cite{Poi-GenNS} it is definably isomorphic to $K$. So we have
  obtained a definable action of $K^\star$ on $A$ by group
  automorphisms---a definable $K$-vector space structure--- and the
  action of $G$ factorizes through the action of $K^\star$.
\end{proof}

We will also need an notion of dimension than generalizes Krull dimension to all definable subsets of the field:

\begin{defn}
Let $D$ be a definable subset of $K^n$. We write $\dimK(D)$ for the Krull dimension of the Zariski closure of $D$ (inside $\Aa^n$).
\end{defn}

\begin{rem}
\begin{enumerate}
\item For any $K_0 = \alg{K_0} \subseteq K$, $\dimK(D) = \max\{\trdeg(a/K_0)\mid a\in D\}$.
\item $\dimK$ is preserved under definable bijection.
\item Let $f : D \to X$ be some definable map with $X$ purely imaginary. One can easily show from the above that $\dimK(D) = \max\{\dimK(f^{-1}(a)) \mid a\in X\}$.
\item Any definable set $D$ has non empty interior in its Zariski closure. It follows that if $D_1\subseteq D_2$ are definable and $\dimK(D_1) = \dimK(D_2)$, then $D_1$ has non empty interior in the Zariski closure of $D_2$.
\item In fact, $\dimK(D)$ is exactly the $C$-minimal dimension of $D$.
\end{enumerate}
\end{rem}

\begin{thm}
\label{T:acvf-fields}
Let $F$ be an infinite field interpretable in $\ACVF$. Then $F$ is
definably isomorphic to the residue field or the valued field.
\end{thm}

\begin{proof}
  If $F$ is stable, then by \cite[Lemma\ 2.6.2]{HasHruMac-ACVF}, it is
  definable in the residue field (over some new parameters), which is
  a pure algebraically closed field. Hence, by \cite{Poi-GenNS},
  it is definably isomorphic to the residue field. There are no
  infinite fields definable over $\Gamma$, since $\Gamma$ is a pure
  ordered divisible Abelian groups.

  Let us assume that $F$ is unstable. By Proposition\ \ref{P:unstable
    fields}, there exists a non-definably compact \(\Gamma\)-internal
  definable homorphic image \(V\) of $F^\star$. By
  \cite[Theorem\ 1.2]{PetSte}, $V$ contains a definable
  one-dimensional torsion-free group. Since $\Gamma$ is a pure ordered
  divisible Abelian group, such $V$ contains a group isomorphic to
  $(\Gamma,+)$---see \cite[Theorem\ 1.5]{EdmEle-GpOMin} for a more
  general statement. Hence \(F\) is not boundedly imaginary.

  Let $D$ be a subrng of $F$, with $(D,+)$ connected stably dominated,
  and such that $F$ is the field of fractions of $D$
  (cf. Proposition\ \ref{P:fields}). There exists a surjective
  definable map $D \times D \to F$---namely $(x,y) \mapsto x/y$ for
  non-zero $y$, $(x,0) \mapsto 0$. Since \(F\) is not boundedly
  imaginary, neither is \(D\). Let $f $ be the universal homomorphism
  of Remark\ \ref{R:embed-univ} from $(F,+)$ into an algebraic
  group. Let $I$ be the purely imaginary kernel. For any \(c\in F\),
  $d \mapsto f(c\cdot d)$ is another homomorphism into an algebraic group,
  so it must factor through $f$; thus if $f(d)=0$ then $f(c\cdot d) = 0$,
  i.e. $I$ is an ideal of $F$. If \(I = F\), then \((D,+)\) is a
  purely imaginary stably dominated and by
  Remark\ \ref{R:pureim-stdom}, \(D\) is boundedly imaginary, a
  contradiction. It follows that $I \neq F$, so $I=(0)$. Hence $f$ is
  an isomorphism onto a subgroup of an algebraic group $G$.

  Let $L$ be the limit stably dominated subgroup of $(F^\star,\cdot)$ and
  let $H := L \ltimes (F,+)$. Let $h: H \to G$ be the homomorphism of
  Proposition\ \ref{P:embed}. The kernel is a purely imaginary subset
  of \(H\subseteq F^2\) so it is finite. But $H$ has no finite
  non-trivial normal subgroups, so $h$ is an embedding and we identify
  $H$ with its image in $G$.

  We now proceed as in \cite{Pil-FieldsQp}, with some local changes
  of reasoning. The Zariski closure $\overline{F}$ of $F$ is a (connected)
  commutative algebraic group and the Zariski closure $\overline{L}$
  acts on $F$ by conjugation. Let $B \leq \overline{F}$ be an
  $\overline{L}$-invariant subgroup of $\overline{F}$. If $a\in B\cap F$, let $Z$ be the $L$-orbit of $a$, which is in definable bijection with $L$. Since $F^\star/L$ is $\Gamma$-internal, we have $\dimK(Z) = \dimK(L) = \dimK(F^\star) = \dimK(F)$. So the Zariski closure of $K$, which is contained in $B$, is equal to $\overline{F}$; and $B = \overline{Z}$. If $B\cap F = \emptyset$ and $B$ is not trivial, quotienting by $B$, we would be
  able to embed $F$ in a variety of strictly lower dimension than
  $\overline{F}$, a contradiction, so $B$ is trivial.
  
  Applying Lemma\ \ref{L:IrrAction},
  we see that $\overline{F}$ is a vector space over the valued field
  $K$ and that the action of $\overline{L}$ is linear. We also denote
  this action by $\cdot$ since it extends the action of $L$ on $F$ by
  multiplication. If $r_i \in L$ and $\sum r_i = 0$, then $\sum r_i
  \cdot y= (\sum r_i)\cdot y = 0$ for any $ y \in F$, and by Zariski
  density of $F$, for $y \in \overline{F}$. Thus the action may be
  extended to an action of the ring $R$ generated by $L$ on $\overline{F}$, again
  extending the action by multiplication of $R$ on $F$. Note that $R$
  contains the subrng $D$ constructed in Proposition\ \ref{P:fields}
  and hence its fraction field is $F$. Finally, if $0 \neq r \in R$
  then $r$ acts on $F$ as an invertible linear transformation, since
  the image contains $r\cdot F = F$ and hence, by Zariski density,
  $\overline{F}$. Thus we can extend the action to a $K$-linear action
  of $F$ on $\overline{F}$ which extends the action by
  multiplication. We also denote that action by $\cdot$.

  Let $Z$ be some non-zero orbit of $L$ on $F$. Then $Z$ is definably
  isomorphic to $L$, and, since $F^\star /L$ is $\Gamma$-internal, it
  has the same $C$-minimal dimension as $F$. So $Z$ contains a
  non-empty open subset $U$ of $\overline{F}$.  Pick any $c \in U$,
  then for any $\alpha \in \Oo$ with $\val(\alpha-1)$ sufficiently
  large, $\alpha c \in U$ and thus there exists $h\in L$ such that
  $h\cdot c = \alpha c$.  Since $F$ is a field, $h$ is uniquely
  defined. For any $b \in L$, we have $\alpha (b \cdot c) = b \cdot
  (\alpha c) = b \cdot (h \cdot c) = h \cdot (b \cdot c)$ so the
  action by $h$ and scalar multiplication by $\alpha$ agree on
  $Z$. Since $Z$ has the same $C$-minimal dimension as $\overline{F}$,
  the linear span of $Z$ is $\overline{F}$ and those two linear
  functions actually agree on $\overline{F}$.

  Let $E = \{\alpha \in K \mid (\exists b \in F) (\forall x \in \overline{F})
  \alpha x = b \cdot x\}$. This is clearly a subfield of $K$, and it
  contains a neighborhood of $1$.  Hence it contains a neighborhood
  $N$ of $0$.  If $0 \neq x \in K$ then $x^{-1} N \cap N$ is open and
  contains some non-zero $u$, so $x\cdot u \in N$ and $x = (x\cdot
  u)/u \in E$. It follows that $E=K$. Moreover, the map sending any
  $\alpha\in K$ to $b\in F$, such that the action by $b$ coincides with
  scalar multiplication by $\alpha$, is an embedding of rings. Since
  $F$ has bounded dimension, for some $m$ no definable subset of $F$
  admits a definable map onto $K^m$; so $\dim_K F < m$.  Since $K$ is
  algebraically closed we have $K \simeq F$.
\end{proof}

We naturally expect any infinite non-Abelian definably simple group
definable in $\ACVF$ to be isomorphic to an algebraic group defined
over the residue field or the valued field. A proof along the above
lines may be possible assuming a positive solution to
Problem\ \ref{Pb:double} for $\ACVF$, along with an interpretation of an
ordered proper semi-group structure on $ H \backslash G / H$.

The existing metastable technology yields a proof of the Abelian case:

\begin{prop}
  Let $A$ be an non-trivial Abelian group definable in $\ACVF$. Then
  there exist definable subgroups $B < C \leq A$ with $C/B$ definably
  isomorphic (with parameters) to an algebraic group over the residue
  field or a definable group over the value group.
\end{prop}

\begin{proof}
  Let $H$ be the limit stably dominated subgroup of $A$, cf. Theorem\ \ref{T:lmstd}. If
  $H < A \neq 0$ we can take $C=A$ and $B = H$.  If
  $A = H$, then $A$ is limit stably dominated. In particular,
  $A$ contains a non-zero stably dominated definable Abelian group $C$
  and, by Proposition\ \ref{P:groupdom}, $C$ has a non trivial stable
  quotient. Since all stable definable sets in $\ACVF$ are internal to
  $k$ and $k$ is a pure algebraically closed field, it follows that
  this quotient is (isomorphic to) an algebraic group over $k$.
\end{proof}

\subsection{Residually Abelian groups}

\begin{example}
In $\ACVF$, there exists connected stably dominated non-Abelian
groups, with Abelian stable part:
  \begin{enumerate}
  \item Let $A = \Ga^2$, and let $\beta: A^2 \to \Ga$ be a
    non-symmetric bilinear map defined over the prime field,
    e.g. $\beta((a_1,a_2),(b_1,b_2))=a_1b_2 - a_2 b_1$. Let $t$ be an
    element with $\val(t)>0$ and consider the following group law on
    $A^2$: \((a,b) \star (a',b') = (a+a', b+b' + t \beta(a,a'))\).
  \item If we work in $\Ga(\Oo/t^2\Oo)$ where $\val(t)>0$, we can also
    just take $a \star b = a+b+t\beta(a,b)$. This dimension $1$
    example is due to Simonetta \cite{Sim-CMin}. This group does not
    lift to an algebraic group over $\Oo$.
  \end{enumerate}
\end{example}

However, we may ask if any connected stably dominated group of weight
$1$ definable is nilpotent.

\sloppy
\bibliographystyle{alpha}
\bibliography{short,biblio}

\end{document}